\pdfoutput=1
\documentclass[11pt,leqno]{article}

\usepackage{amsmath,amsthm,amssymb}
\usepackage{authblk,subfigure,booktabs,enumerate,multirow}
\usepackage{graphicx}
\usepackage{color}
\usepackage{bm}
\usepackage{url}
\usepackage[round]{natbib}
\usepackage{enumitem}

\usepackage{geometry}
\usepackage{listings}
\usepackage{apptools}

\theoremstyle{plain}
\newtheorem{thm}{Theorem}
\newtheorem{prop}{Proposition}

\theoremstyle{remark}
\newtheorem*{rem}{Remark}

\theoremstyle{definition}

\AtAppendix{\counterwithin{thm}{section}}

\makeatletter

\numberwithin{equation}{section}
\numberwithin{figure}{section}
\numberwithin{table}{section}

\title{Continuum-marginal optimal transport:\\
  a mesh-free kernel method}
\author[1]{Yumiharu Nakano\thanks{E-mail: nakano@comp.isct.ac.jp}}
\affil[1]{Department of Mathematical and Computing Science \protect \\
  Institute of Science Tokyo}

\date{Published in \emph{Transactions on Machine Learning Research} (10/2026)\\
  Reviewed on OpenReview: \url{https://openreview.net/forum?id=1bzgCxWYpr}}

\begin{document}

\maketitle
\begin{abstract}
In this paper we study \emph{continuum-marginal optimal transport}.
Given a time-continuous family of probability marginals, the problem
is to recover the minimum-energy velocity field whose flow reproduces
every marginal.
This problem is the continuum limit of the classical two-marginal
Benamou--Brenier formulation, and also the deterministic limit of the
Nelson problem of stochastic optimal transport.
We propose a practical mesh-free solver for this problem.
The weak continuity equation is embedded in a reproducing kernel
Hilbert space, yielding a sample-only objective that requires no spatial discretization.
The velocity is parametrized by any linear-in-parameters dictionary or
neural network, and is optimized by mini-batch stochastic methods.
Synthetic experiments confirm that the method achieves accurate drift
recovery and marginal consistency.
The same computational framework also applies to the Nelson problem.

\begin{flushleft}
\end{flushleft}
\end{abstract}

\section{Introduction}
\label{sec:intro}

The Monge optimal transport problem, in its classical form, consists of
finding a deterministic map that sends a source probability distribution
to a target probability distribution with minimum cost.
\citet{ben-bre:2000} gave a dynamic reformulation of
this problem, in which the static map is replaced by a time-dependent
velocity field: one seeks a pair of a velocity and a density which
minimizes the kinetic energy, subject to the continuity equation and
to the given distributions as endpoint constraints.
The optimal velocity field generates deterministic ODE paths that
realize the Monge map, and no stochasticity is involved.
This two-marginal dynamic formulation has been studied extensively in
computational optimal transport.

\paragraph{From two marginals to a continuum of marginals.}
In many applications, the evolving density is observed not only at the
endpoints but also at \emph{all} intermediate times.
Examples include particle image velocimetry, crowd monitoring, and
density-functional molecular dynamics, in which the density is
observed continuously while individual trajectories are not tracked.
Motivated by such situations, we consider the following generalization
of the Benamou--Brenier problem.
Let $\mu=\{\mu_t\}_{0\le t\le T}$ be a given continuous family of
probability measures on $\mathbb{R}^d$.
We seek a velocity field $u$ that minimizes
\begin{equation}\label{eq:P0-intro}
 \int_0^T\!\int_{\mathbb{R}^d}|u(t,x)|^2\,p(t,x)\,dx\,dt
\end{equation}
over all $u$ for which the continuity equation
\[
 \partial_t p+\mathrm{div}(up)=0
\]
holds, together with $p(t,\cdot)=d\mu_t/dx$ for every $t\in[0,T]$.
We call \eqref{eq:P0-intro} the \emph{continuum-marginal optimal
transport} problem, following the terminology of
\citet{Mikami2021}.%
\footnote{For brevity, in the rest of the paper we also use the
shorthand \emph{all-time optimal transport} (all-time OT), which
reflects Mikami's own phrasing ``marginals at all times''. The
shorthand reads more naturally when the term recurs many times in
the body. 
We retain the formal name \emph{continuum-marginal} in the title,
abstract and keywords, where its descriptive precision is needed,
and emphasize that the two refer to the same problem.}
The name \emph{continuum-marginal} refers to the fact that a marginal
constraint is imposed at every time $t\in[0,T]$, that is, on a
continuum of times.  The problem can be reached from the classical
theory in two ways.  First, as observed by \citet{Mikami2021}, it is
the Benamou--Brenier formulation of two-marginal transport with the
constraint extended from the two endpoints to all times.  Second, it
is the limit of multi-marginal optimal transport as the number of
prescribed marginals grows: constraining $\mu_{t_1},\ldots,\mu_{t_K}$
at $K$ times and letting $K\to\infty$ with the mesh shrinking turns
the finite set of constraints into an entire curve.  The first
description places the problem next to transport with terminal
conditions, and the second places it next to the discrete
multi-marginal methods reviewed below, of which it is the limiting
case.  The second also makes explicit that the unknown is a velocity
field on $[0,T]$ rather than a finite family of transport maps.
When only $\mu_0$ and $\mu_T$ are prescribed, \eqref{eq:P0-intro}
reduces to the classical $W_2^2$ cost.
When the full continuum of marginals is prescribed, the feasible set
is smaller, and the optimal velocity contains dynamical information
that cannot be recovered from the endpoints alone.

\citet{Mikami2021} established existence and uniqueness of the
minimizer as well as the gradient structure
$u^*(t,x)=\nabla_x\psi(t,x)$.
The optimal paths are ODE solutions and can be regarded as
\emph{continuous-time Monge maps}.
Problem \eqref{eq:P0-intro} is also the small-noise limit
($\sigma\to 0$) of the \emph{Nelson problem}
\citep{Nelson1985,Mikami1990,Mikami2021}.
The Nelson problem is the stochastic analogue of \eqref{eq:P0-intro},
in which the ODE $\dot X_t=u(t,X_t)$ is replaced by the It\^o SDE
$dX_t=u\,dt+\sigma\,dW_t$ and the continuity equation is replaced by
the Fokker--Planck equation, under the same marginal constraints
$X_t\sim\mu_t$ for every $t\in[0,T]$.
It can be written as the minimization of relative
entropy with respect to a reference diffusion, and its minimizers are
known as Nelson processes \citep{CattiauxLeonard1994,CattiauxLeonard1995}.
\citet{BackhoffFontbona2020} treat the continuum of marginal
constraints directly through convex duality, for a general class of
entropy-type costs.
The Nelson problem is also called a Schr\"odinger
bridge with marginals fixed at all times. 
Hence \eqref{eq:P0-intro} serves as a deterministic counterpart that
connects Monge transport and the Schr\"odinger bridge.

Marginal data of this form arise in several contexts.  The continuum
is the population target of the estimator, and in practice it is
approximated rather than observed directly.  Three settings come close
to it.  Live-cell and particle-imaging experiments record population
densities at a temporal resolution much finer than the time scale of
the motion and so sample the marginal path densely.  Simulations and
model-reduction pipelines provide $\mu_t$ at any requested time.
Fitted density models allow sampling at arbitrary times.  The
estimator of Section~\ref{sec:estimator} uses only $N$ particles at
each of $M$ time points, so the continuum enters only through the
population quantity that this $M$-sample average estimates, with bias
quantified in~\eqref{eq:Qhat-bias}.  The single-cell data of
Experiment~6 in Section~\ref{sec:exp-eb} have only five observation
days, so they sit far from this continuum and closer to a
multi-marginal setting.  We therefore present that experiment as an
application to real measurements rather than as a test of the
continuum formulation.

\paragraph{Related work.}
The two-marginal Benamou--Brenier formulation \citep{ben-bre:2000} is
solved on a fixed Eulerian grid by augmented Lagrangian or proximal
methods \citep{Papadakis2014}.
Flow matching \citep{Lipman2023} and stochastic interpolants
\citep{Albergo2023} learn an ODE velocity from source and target
distributions by neural networks, and minibatch OT variants
\citep{Tong2024} improve scalability.
These methods treat only the \emph{two-marginal} setting and do not
impose intermediate marginal constraints.

\emph{Learning dynamics from time-indexed samples.}
A separate line of work in machine learning recovers dynamics from
snapshots without paired trajectories.  Action
matching~\citep{Neklyudov2023} learns the velocity field of an
observed evolution by fitting an action functional to samples, and is
the closest in spirit to the present formulation: both recover a
velocity from marginals alone, with action matching regressing a
scalar potential and the estimator here penalizing the weak continuity
residual in a reproducing kernel Hilbert space.
TrajectoryNet~\citep{tong2020trajectorynet} fits a continuous
normalizing flow to consecutive single-cell snapshots with optimal
transport and density regularizers, and simulation-free score and flow
matching~\citep{tong2024sf2m} recovers a Schr\"odinger bridge between
observed marginals without simulating the dynamics inside the loss.
These methods impose the observed marginals at the sampling times
rather than on a continuum, and none of them supplies a convergence
theory for the relaxed problem.

\emph{Discrete multi-marginal approaches.}
Several approaches are available when marginals are observed at a
finite set of time points.
Waddington-OT (WOT) \citep{Schiebinger2019} chains entropic OT
couplings between consecutive snapshots and reconstructs the drift by
barycentric projection.
The multi-marginal Monge formulation of \citet{nak-sai:2026}
parametrizes $K-1$ deterministic transport maps and penalizes an MMD
distance between the pushed-forward and observed marginals.
Like the present paper, the method of \citet{nak-sai:2026} seeks a
minimum-energy velocity field, but under a finite number of marginal
constraints.
\citet{Albergo2024multimarginal} extended stochastic
interpolants to a multimarginal generative model: given $K{+}1$
densities $\rho_0,\ldots,\rho_K$, they defined a barycentric
interpolant on the $K$-simplex and learned a velocity field via
conditional-expectation regression.
Their method gives \emph{a} velocity that reproduces the marginals,
but it does not minimize the kinetic energy.
It therefore addresses a different problem from the one considered
here.
In contrast, the present paper imposes a continuum of marginal
constraints and explicitly minimizes the Benamou--Brenier energy,
recovering the unique Monge-optimal velocity field.
In our experiments, we compare our method with \citet{nak-sai:2026}
and with WOT.  These are the two existing methods that take
all-time marginal data as input and produce a velocity field as
output.  The method of \citet{nak-sai:2026} minimizes a
piecewise-affine multi-marginal OT cost, and WOT applies the
barycentric projection of pairwise entropic couplings.  Neither
method directly minimizes the Benamou--Brenier energy.  A
comparison with \citet{Albergo2024multimarginal} would measure
whether optimal-transport structure matters beyond marginal
reproduction.  This complementary question is left to future work.

\emph{Stochastic counterpart.}
For the Nelson problem introduced above, Lavenant et al.\
\citep{lav:2024} developed a convex variational framework (global
Waddington-OT) for stochastic trajectory inference, and proved
consistency of the estimator.
Their approach uses entropy regularization with respect to a Brownian
reference measure, and so requires $\sigma>0$.
It does not apply to the deterministic case $\sigma=0$.
For the two-endpoint stochastic problem, Schr\"odinger bridge
algorithms \citep{DeBortoli2021} also provide practical solutions, but
these again target the stochastic setting.

Problem \eqref{eq:P0-intro} is the natural deterministic foundation of
these stochastic formulations
(see \citealp[Theorem~3.7]{Mikami2021}), but no practical numerical method
has been available for it.
Grid-based OT solvers scale poorly with the spatial dimension, and
flow matching methods do not incorporate intermediate marginal
information.

\paragraph{Our contribution.}
In this paper, we propose a practical mesh-free solver for
\eqref{eq:P0-intro}.
We replace the continuity-equation constraint by its weak form and
embed the residual in a reproducing kernel Hilbert space (RKHS)
\citep{Scholkopf2002}.
The squared RKHS norm then admits a closed-form expression that
involves only kernel evaluations and values of the velocity at sample
points, and is estimated by a sample average from mini-batches.
The velocity $u(t,x)$ is parametrized by a linear-in-parameters
dictionary or by a neural network, and is optimized by stochastic
methods without any spatial discretization.
The same computational framework also applies to the stochastic
setting ($\sigma>0$): the continuity equation is replaced by the
Fokker--Planck equation, a Laplacian term is added to the RKHS
operator, and the resulting problem is the Nelson problem.  We treat
this stochastic case only as a proof-of-concept extension (see
Sections~\ref{sec:stochastic}, \ref{sec:exp-stochastic}).  A full
theoretical analysis is left for future work.

\paragraph{Outline.}
Section~\ref{sec:2} formulates the continuum-marginal OT problem and
summarizes the theoretical foundations of \citet{Mikami2021}.
Section~\ref{sec:3} derives the RKHS-based objective and provides explicit
kernel formulas, with Section~\ref{sec:stochastic} treating the
stochastic extension to the Nelson problem.
Section~\ref{sec:numerical} presents numerical experiments, beginning with
the original deterministic case and then demonstrating the stochastic
extension.

\paragraph{Notation.}
For a set $A$ in a topological space, a $\sigma$-finite measure $\nu$ on $A$,
and $p\ge 1$, we write $L^p(A;\mathbb{R}^m,\nu)$ for the space of
Borel measurable functions $f\colon A\to\mathbb{R}^m$ with
$\int_A|f|^p\,d\nu<\infty$.  When $\nu$ is the Lebesgue measure on
$\mathbb{R}^k$ we abbreviate $L^p(\mathbb{R}^k):=L^p(\mathbb{R}^k;\mathbb{R},\nu)$.
For an open set $\Omega\subset\mathbb{R}^n$, $C^k(\Omega)$ denotes the space
of $k$-times continuously differentiable functions,
$C^k_b(\Omega)$ is the subspace whose derivatives up to order~$k$ are
bounded, and $C_0^\infty(\Omega)$ is the space of infinitely differentiable
functions with compact support.
We write $\nabla_x$ and $\Delta_x$ for the gradient and Laplacian in
the spatial variable $x\in\mathbb{R}^d$, and
$\mathrm{div}_x(\cdot)$ for the divergence.
For a mean vector $m\in\mathbb{R}^d$ and a symmetric positive definite
matrix $\Sigma\in\mathbb{R}^{d\times d}$, $\mathcal{N}(m,\Sigma)$
denotes the $d$-dimensional Gaussian distribution with mean $m$ and
covariance $\Sigma$.  In the one-dimensional case we write
$\mathcal{N}(m,\sigma^2)$ with variance $\sigma^2$.

\section{Problem formulation}
\label{sec:2}

Let $\mu=\{\mu_t\}_{0\le t\le T}$ be a given continuous family of Borel 
probability measures on $\mathbb{R}^d$, each admitting a density
$p(t,x)=d\mu_t/dx$.
We consider the ODE $\dot X_t = u(t,X_t)$, where the drift $u$ belongs to
\[
 \mathcal{U}:=\bigl\{u\in L^1([0,T]\times\mathbb{R}^d;\mathbb{R}^d,p(t,x)\,dt\,dx)
 \;:\; \mathrm{div}_x(pu)\in L^1((0,T)\times\mathbb{R}^d)\text{ in the weak sense}\bigr\}.
\]
The divergence condition ensures that the distributional residual of
the continuity equation is a function, not merely a distribution.
Some condition of this kind is needed for the problem to be well
posed for rough drifts.
Denote by $\mathcal{U}_0$ the set of admissible velocity fields:
$u\in\mathcal{U}_0$ if
$u\in \mathcal{U}$ and
for any $\varphi\in C^{1}_b([0,T]\times\mathbb{R}^d)$ and $t\in[0,T]$,
\begin{equation}\label{eq:weak-CE}
 \int_{\mathbb{R}^d}\varphi(t,x)\,p(t,x)\,dx
 - \int_{\mathbb{R}^d}\varphi(0,x)\,p(0,x)\,dx
 = \int_0^t\!\int_{\mathbb{R}^d}
   \bigl(\partial_s\varphi + u(s,x)^{\mathsf{T}}\nabla_x\varphi\bigr)
   p(s,x)\,dx\,ds.
\end{equation}
This is the weak formulation of the continuity equation
\begin{equation}\label{eq:continuity}
  \partial_t p + \mathrm{div}(up)=0.
\end{equation}
Our problem is:
\begin{equation}
\tag{$P$}\label{eq:P0}
  V_0(\mu)
  := \inf_{u\in\mathcal{U}_0}
     \int_{[0,T]\times\mathbb{R}^d} |u(t,x)|^2\,p(t,x)\,dx\,dt.
\end{equation}
When only two marginals $\mu_0$ and $\mu_T$ are prescribed, $(P)$ reduces to the
\emph{Benamou--Brenier formulation} of the $W_2$ optimal transport
\citep{ben-bre:2000}.
Problem $(P)$ can equivalently be stated over absolutely continuous paths. Indeed, we have 
\[
 V_0(\mu)= \inf\mathbb{E}\!\int_0^T\!|\dot X(t)|^2\,dt
\]
where the infimum is taken over all absolutely continuous $(X(t))_{0\le t\le T}$ such that $\mathbb{P}\circ X(t)^{-1}=\mu_t$ for all $t\in [0,T]$, 
since any admissible velocity--density pair $(u,p)$ can be lifted to a pathwise
ODE solution by the \emph{superposition principle} of Ambrosio
\citep[Theorem~3.9]{Mikami2021}.

The theoretical properties of $(P)$ are established by
\citet[Section~3.2.2]{Mikami2021}.
\begin{prop}[{\citealp[Theorem~3.8, Proposition~3.7]{Mikami2021}}]
\label{prop:existence}
Suppose $V_0(\mu)<\infty$.
Then $(P)$ admits a unique minimizer $u^*\in\mathcal{U}_0$, determined
$\mu_t(dx)\,dt$-a.e.
Every optimal path satisfies
\begin{equation}\label{eq:ode-minimizer}
  X(t) = X(0) + \int_0^t u^*(s,X(s))\,ds, \quad t\in[0,T],\;\; \text{a.s.}
\end{equation}
Moreover, the minimizer has gradient structure:
$u^*(t,x)=\nabla_x\psi(t,x)$ for some $\psi:[0,T]\times\mathbb{R}^d\to\mathbb{R}$ that is
locally weakly differentiable in~$x$.
\end{prop}

If only $\mu_0$ and $\mu_T$ are prescribed, the minimizer of $(P)$ gives the
Benamou--Brenier velocity field, and the optimal paths are the McCann
displacement interpolants \citep{McCann1997}.
With the full continuum of marginals prescribed, the feasible set is
strictly smaller and the optimal velocity encodes richer dynamical information that cannot be
recovered from the endpoints alone.

Problem $(P)$ arises as the zero-noise limit of the \emph{Nelson problem}
(stochastic optimal transportation with marginals fixed at all times,
in the phrasing of \citealp{Mikami2021})
\citep{Nelson1985,Mikami1990,Mikami2021}.
Write $V_\varepsilon(\mu)$ for the value of the Nelson problem
with diffusion coefficient $\sigma=\sqrt{\varepsilon}$.
\citet[Theorem~3.7]{Mikami2021} shows that
\begin{equation}\label{eq:Veps-to-V0}
  \lim_{\varepsilon\to 0}V_\varepsilon(\mu) = V_0(\mu)
\end{equation}
and any weak limit point of a minimizer of
$V_\varepsilon(\mu)$ is a minimizer of
$V_0(\mu)$.
In the two-marginal case, Theorem~2.9 of \citep{Mikami2021} further shows that
the optimal drift converges to the Monge map
$D\varphi(\cdot)-\mathrm{Id}$,
with $\varphi$ the Brenier convex potential.
Thus $(P)$ sits at the deterministic endpoint of a continuous spectrum that
connects optimal transport ($\sigma=0$) to the Schr\"odinger bridge
($\sigma>0$).

\section{Methods}
\label{sec:3}

\subsection{Kernel embedding of the weak continuity equation}
\label{sec:3-embedding}

The weak continuity equation~\eqref{eq:weak-CE} can be rewritten as:
for any $t\in [0,T]$ and $\varphi\in C^{\infty}_0([0,T]\times\mathbb{R}^d)$, 
\begin{equation}
\label{eq:3.1}
 \int_0^t\int_{\mathbb{R}^d}p(s,x)\mathcal{A}^u\varphi(s,x)\, ds\,dx = \int_{\mathbb{R}^d}p(t,x)\varphi(t,x)\,dx - \int_{\mathbb{R}^d}p(0,x)\varphi(0,x)\,dx,
\end{equation}
where
\[
 \mathcal{A}^u\varphi = \partial_t\varphi + u^{\mathsf{T}}\nabla_x\varphi.
\]
\begin{rem}
In the stochastic extension ($\sigma>0$, Section~\ref{sec:stochastic}),
the operator becomes
$\mathcal{A}^u\varphi = \partial_t\varphi + u^{\mathsf{T}}\nabla_x\varphi
+ (\sigma^2/2)\Delta_x\varphi$.
All derivations below carry through with the Laplacian term included.
We present the case $\sigma=0$ for clarity.
\end{rem}

Let $K:\mathbb{R}^{d+1}\to\mathbb{R}$ be translation-invariant,
$K(\xi,\eta)=K(\xi-\eta)$. 
We require the following: 
\begin{enumerate}
\item[(A1)] $K\in C(\mathbb{R}^{d+1})\cap L^1(\mathbb{R}^{d+1})$ is
positive definite, and the induced reproducing kernel Hilbert space
$\mathcal{H}$ is universal on every compact subset of
$[0,T]\times\mathbb{R}^d$.
\item[(A2)] The Fourier transform $\widehat{K}$ of $K$ satisfies 
\[ 
\int_{\mathbb{R}^{d+1}}(1+|z|^2)\widehat{K}(z)\,dz<\infty. 
\]
\end{enumerate}
By Bochner's theorem, (A1) implies $\widehat{K}\ge 0$, and the
inversion formula
\[
 K(\xi)=(2\pi)^{-(d+1)/2}\int_{\mathbb{R}^{d+1}}e^{iz^{\mathsf{T}}\xi}\widehat{K}(z)\,dz,
 \qquad \xi\in\mathbb{R}^{d+1},
\]
holds with $i=\sqrt{-1}$ the imaginary unit.  Universality on every
compact set is equivalent (for translation-invariant kernels) to
$\mathrm{supp}\,\widehat{K}=\mathbb{R}^{d+1}$, see Sriperumbudur et
al.\ \citep{Sriperumbudur2010}.  We denote the RKHS inner product and norm by
$\langle\cdot,\cdot\rangle$ and $\|\cdot\|$.

We also impose the following regularity condition on the marginal flow:
\begin{enumerate}
\item[(A3)] $p\in C([0,T];L^1(\mathbb{R}^d))\cap L^\infty([0,T]\times\mathbb{R}^d)$,
each $p(t,\cdot)$ is a probability density on $\mathbb{R}^d$, and
$\partial_t p\in L^1((0,T)\times\mathbb{R}^d)$ in the weak sense.
\end{enumerate}
Under (A3), for every $u\in\mathcal{U}$ the \emph{residual}
\begin{equation}\label{eq:residual}
 D^u(s,x):=\partial_s p(s,x)+\mathrm{div}_x\bigl(p(s,x)u(s,x)\bigr)
\end{equation}
belongs to $L^1((0,T)\times\mathbb{R}^d)$, since both summands do
($\partial_s p\in L^1$ by (A3), and $\mathrm{div}_x(pu)\in L^1$
by the definition of $\mathcal{U}$).

The weak continuity equation~\eqref{eq:3.1} is stated for the natural
test class $C_0^\infty$.  The proposition below extends the test class
to the entire RKHS $\mathcal{H}$, so that the residual functional can
later be characterized as an element of $\mathcal{H}$ itself.
\begin{prop}\label{prop:weak-eq}
Assume {\rm (A1)}, {\rm (A2)} and {\rm (A3)}.  Then for any
$u\in \mathcal{U}$ and any
$t\in[0,T]$, the weak continuity equation~\eqref{eq:3.1} holds for all
$\varphi\in C_0^{\infty}([0,T]\times\mathbb{R}^d)$ if and only if
\begin{equation}\label{eq:3.2}
\int_0^t\!\!\int_{\mathbb{R}^d}\!p(s,x)\mathcal{A}^u\varphi\,ds\,dx
 = \int_{\mathbb{R}^d}\!p(t,x)\varphi(t,x)\,dx
 - \int_{\mathbb{R}^d}\!p(0,x)\varphi(0,x)\,dx,
 \quad \varphi\in\mathcal{H}.
\end{equation}
\end{prop}
\begin{proof}
By (A2) and \citep[Theorem~10.45]{Wendland2005}, every
$\varphi\in\mathcal{H}$ is bounded with bounded space-time partial
derivatives of order one (resp.\ two in the $\sigma>0$ case).
Combined with (A3), the bilinear functional
\begin{equation}\label{eq:Lphi}
 L_t(\varphi):=\int_0^t\!\!\int_{\mathbb{R}^d}p(s,x)\mathcal{A}^u\varphi(s,x)\,ds\,dx
 -\int_{\mathbb{R}^d}p(t,x)\varphi(t,x)\,dx
 +\int_{\mathbb{R}^d}p(0,x)\varphi(0,x)\,dx
\end{equation}
is well-defined for every $\varphi\in C_b^1([0,T]\times\mathbb{R}^d)$.
Indeed, (A3) gives $p\in L^1([0,T]\times\mathbb{R}^d)$, which combined
with the boundedness of $\partial_s\varphi$ yields
$p\,\partial_s\varphi\in L^1$.  Similarly, the defining property
$u\in L^1([0,T]\times\mathbb{R}^d;\mathbb{R}^d,p\,dt\,dx)$ of
$\mathcal{U}$ combined with the boundedness of $\nabla\varphi$ yields
$p\,u^{\mathsf{T}}\nabla\varphi\in L^1([0,T]\times\mathbb{R}^d)$, since
$\int|p\,u^{\mathsf{T}}\nabla\varphi|\,dt\,dx
\le\|\nabla\varphi\|_\infty\!\int|u|\,p\,dt\,dx<\infty$.
The spatial boundary at infinity vanishes because $p(s,\cdot)\in L^1(\mathbb{R}^d)$
for every $s$ and $\varphi$ is bounded.
Hence $D^u\in L^1((0,t)\times\mathbb{R}^d)$ and the integration by
parts in $s$ and $x$ yields
\begin{equation}\label{eq:Lphi-D}
 L_t(\varphi) \;=\; -\int_0^t\!\!\int_{\mathbb{R}^d}D^u(s,x)\,\varphi(s,x)\,ds\,dx
 \end{equation}
for every $\varphi\in C_b^1([0,T]\times\mathbb{R}^d)$. 

Assume \eqref{eq:3.1}. Then, equivalently, by
\eqref{eq:Lphi-D}, $\int D^u\varphi=0$ for every
$\varphi\in C_0^\infty([0,T]\times\mathbb{R}^d)$, hence $D^u=0$ a.e.
on $(0,t)\times\mathbb{R}^d$.  Substituting back into \eqref{eq:Lphi-D}
gives $L_t(\varphi)=0$ for every $\varphi\in\mathcal{H}\subset C_b^1([0,T]\times\mathbb{R}^d)$, which is exactly \eqref{eq:3.2}.

Conversely, assume \eqref{eq:3.2}, i.e.\ $L_t(\varphi)=0$
for every $\varphi\in\mathcal{H}$.  Extend $D^u\in L^1((0,t)\times\mathbb{R}^d)$
by zero to $\mathbb{R}^{d+1}$.  For any $y\in\mathbb{R}^{d+1}$ the
section $K(\cdot,y)$ belongs to $\mathcal{H}$ by the reproducing
property, so applying \eqref{eq:Lphi-D} with $\varphi=K(\cdot,y)$
gives
\begin{equation}\label{eq:KstarD}
 \int_{\mathbb{R}^{d+1}}K(z-y)D^u(z)\,dz = 0, \quad y\in\mathbb{R}^{d+1}.
\end{equation}
Taking the Fourier transform of \eqref{eq:KstarD}, we get $\widehat{K}\widehat{D^u}=0$ a.e. 
Since $D^u\in L^1(\mathbb{R}^{d+1})$, the function $\widehat{D^u}$ is continuous. This together with
$\mathrm{supp}\,\widehat{K}=\mathbb{R}^{d+1}$ yields $\widehat{D^u}= 0$ on $\mathbb{R}^{d+1}$, hence
$D^u=0$ a.e.  Inserting this into \eqref{eq:Lphi-D} for any
$\varphi\in C_0^\infty([0,T]\times\mathbb{R}^d)$ yields \eqref{eq:3.1}. 
\end{proof}

\begin{rem}
The argument uses $\mathcal{H}$ only through the family
$\{K(\cdot,y):y\in\mathbb{R}^{d+1}\}$.  Universality enters as the
characteristic-kernel property
$\mathrm{supp}\,\widehat{K}=\mathbb{R}^{d+1}$, which makes
convolution by $K$ injective on $L^1(\mathbb{R}^{d+1})$.  The argument extends to the
stochastic case ($\sigma>0$) verbatim with the operator
$\mathcal{A}^u$ of \eqref{eq:Au-sigma} and the Laplacian-augmented
residual.  In that case (A2) is used in full to ensure
$\mathcal{H}\hookrightarrow C_b^2$.
\end{rem}

\subsection{Reproducing-kernel residual and the relaxed problem}
\label{sec:3-residual}

By the reproducing property, the linear functional $L_t$ of
\eqref{eq:Lphi} is itself an element of $\mathcal{H}$.  Concretely,
for any $u\in L^1([0,T]\times\mathbb{R}^d;\mathbb{R}^d,p(t,x)\,dt\,dx)$
and $s\in[0,T]$, define
\begin{align*}
 H_s^u(t,x)&:= \int_0^s \int_{\mathbb{R}^d}p(r,y)\mathcal{A}^u_{r,y}K(t,x, r,y)\,dr\,dy
  - \int_{\mathbb{R}^d}p(s,y)K(t,x, s,y)\,dy \\
 &\quad + \int_{\mathbb{R}^d}p(0,y)K(t,x, 0,y)\,dy,
\end{align*}
where $K(t,x,s,y)=K((t,x) - (s,y))$ and $\mathcal{A}^u_{t,x}=\mathcal{A}^u$ that acts on the variable $(t,x)$.
\begin{thm}\label{thm:Ht-in-H}
Suppose that {\rm (A1)}, {\rm (A2)} and {\rm (A3)} hold. Then
$H_s^u\in\mathcal{H}$ for any $s\in [0,T]$ and
$u\in L^1([0,T]\times\mathbb{R}^d;\mathbb{R}^d, p(t,x)\,dt\,dx)$.
Moreover, $u\in\mathcal{U}_0$ if and only if $H_s^u=0$ for every
$s\in [0,T]$.
\end{thm}
\begin{proof}
Step (i). Since $u\in\mathcal{U}$ is in $L^1([0,T]\times\mathbb{R}^d;
p\,dt\,dx)$, by Fubini there exists a Lebesgue-null set
$\mathcal{N}\subset[0,T]$ such that for every $t\in[0,T]\setminus\mathcal{N}$
the section $u(t,\cdot)\in L^1(\mu_t)$, hence $u(t,\cdot)<\infty$
$\mu_t$-almost everywhere.  Fix any such $t\in[0,T]\setminus\mathcal{N}$.
The values of $\Phi_t$ on the null set $\mathcal{N}$ do not affect the
Bochner integral in Step (iii).  We claim
\begin{equation}\label{eq:3.3}
 \Phi_t(\cdot):=\int_{\mathbb{R}^d}p(t,x)\mathcal{A}^u_{t,x}K(\cdot, t,x)\,dx\in\mathcal{H}.
\end{equation}
By Lemma 10.44 and Theorem 10.45 in \citep{Wendland2005}, for any
$(t,x)\in [0,T]\times\mathbb{R}^d$ we have
$\mathcal{A}^u_{t,x}K(\cdot, t,x)\in\mathcal{H}$.  Let
$\{X_t^{(n)}\}_{n=1}^{\infty}$ be i.i.d.\ copies of $X_t\sim\mu_t$
and define the Monte-Carlo approximation
\[
 \psi_n(\xi)=\frac{1}{n}\sum_{j=1}^nL(\xi; X_t^{(j)}), \quad \xi=(s,y),
\]
with $L(\xi; X_t^{(j)})=\mathcal{A}^u_{t,x}K(s,y,t,x)|_{x=X_t^{(j)}}$,
so that $\psi_n\in\mathcal{H}$.  By the strong law of large numbers,
\begin{equation}\label{eq:3.4}
 \int_{\mathbb{R}^d}p(t,x)\mathcal{A}^u_{t,x}K(s,y,t,x)\,dx
 = \mathbb{E}\left[\mathcal{A}^u_{t,x}K(s,y,t,X_t)\right]
 = \lim_{n\to\infty}\psi_n(\xi), \quad\text{a.s.}
\end{equation}
By the reproducing property, for $m,n\in\mathbb{N}$,
\begin{align*}
 \langle \psi_m, \psi_n\rangle
 &= \frac{1}{mn}\sum_{j=1}^m\sum_{\ell=1}^n
      \mathcal{A}^u_{t,x}\mathcal{A}^u_{t,y}K(t,y,t,x)
      \bigr|_{y=X_t^{(j)}, x=X_t^{(\ell)}}.
\end{align*}
Denote $i=\sqrt{-1}$. For $\xi=(s,y), \eta=(t,x)\in\mathbb{R}^{d+1}$,
the Fourier inversion formula
\[
 K(\xi-\eta)=(2\pi)^{-(d+1)/2}\int_{\mathbb{R}^{d+1}}e^{iz^{\mathsf{T}}(\xi-\eta)}\widehat{K}(z)\,dz
\]
combined with (A2) allows the differentiations $\mathcal{A}^u_\eta$
and $\mathcal{A}^u_\xi$ to commute with the Fourier integral.  Setting
$\xi^{(j)}=(t,X_t^{(j)})$, we have 
\[
\langle \psi_m,\psi_n\rangle = (2\pi)^{-(d+1)/2}\int_{\mathbb{R}^{d+1}}
 \Biggl(\frac{1}{m}\sum_{j=1}^m\mathcal{A}^u_\xi e^{iz^{\mathsf{T}}\xi^{(j)}}\Biggr)
 \overline{\Biggl(\frac{1}{n}\sum_{\ell=1}^n\mathcal{A}^u_\eta e^{iz^{\mathsf{T}}\xi^{(\ell)}}\Biggr)}
 \,\widehat{K}(z)\,dz.
\]
The second-moment hypothesis (A2) gives an integrable dominating
function (the integrand is of order $|z|^2\,\widehat K(z)$), so by the
dominated convergence theorem
\[
 \lim_{m,n\to\infty}\langle\psi_m,\psi_n\rangle
 = (2\pi)^{-(d+1)/2}\int_{\mathbb{R}^{d+1}}\bigl|\mathbb{E}[\mathcal{A}^u_\xi e^{iz^{\mathsf{T}}\Xi}]\bigr|^2\widehat{K}(z)\,dz
 = \mathbb{E}\bigl[\mathcal{A}^u_{\xi}\mathcal{A}^u_{\eta}K(\Xi,\tilde\Xi)\bigr],
\]
with $\Xi=(t,X_t)$ and $\tilde\Xi=(t,\tilde X_t)$ independent copies.
Hence $\{\psi_n\}$ is Cauchy in $\mathcal{H}$ and converges to some
$\Phi_t^*\in\mathcal{H}$.  Using the reproducing property and
\eqref{eq:3.4},
\[
 \Phi_t^*(\xi) = \langle\Phi_t^*,K(\cdot,\xi)\rangle
 = \lim_n\langle\psi_n,K(\cdot,\xi)\rangle
 = \lim_n\psi_n(\xi) = \Phi_t(\xi), \quad \xi\in\mathbb{R}^{d+1},
\]
so $\Phi_t=\Phi_t^*\in\mathcal{H}$, establishing \eqref{eq:3.3}.

Step (ii). For any $s\in[0,T]$ we claim
\begin{equation}\label{eq:Omega-in-H}
 \Theta_s(\cdot):=\int_{\mathbb{R}^d}p(s,y)K(\cdot,s,y)\,dy\in\mathcal{H}.
\end{equation}
The argument is the same as Step (i) with $\mathcal{A}^u_{t,x}$
suppressed. Take an i.i.d. sequence $\{Y_s^{(j)}\}_{j=1}^{\infty}$ with $Y_s^{(1)}\sim \mu_s$ and set
$\theta_n(\xi):= (1/n)\sum_{j=1}^n K(\xi,s,Y_s^{(j)})$. Then, 
\[
 \langle\theta_m,\theta_n\rangle
 = \frac{1}{mn}\sum_{j=1}^m\sum_{\ell=1}^n K\bigl((s,Y_s^{(j)}),(s,Y_s^{(\ell)})\bigr)
 \longrightarrow \mathbb{E}\bigl[K((s,Y_s),(s,\tilde Y_s))\bigr],
\]
with $Y_s,\tilde Y_s$ independent copies of $\mu_s$. Hence $\{\theta_n\}$ is
Cauchy with limit $\Theta_s\in\mathcal{H}$.

Step (iii). We claim
\begin{equation}\label{eq:Psi-in-H}
 \Psi_s(\cdot):=\int_0^s\int_{\mathbb{R}^d}p(r,y)\mathcal{A}^u_{r,y}K(\cdot,r,y)\,dy\,dr
 = \int_0^s\Phi_r\,dr \in\mathcal{H}.
\end{equation}
By Step (i), $\Phi_r\in\mathcal{H}$ for a.e.\ $r\in[0,T]$.  Each
evaluation $r\mapsto\langle\varphi,\Phi_r\rangle
=\int p(r,y)\mathcal{A}^u_{r,y}\varphi(r,y)\,dy$ with $\varphi\in\mathcal{H}$
is measurable, so $r\mapsto\Phi_r$ is weakly measurable.  Since
$\mathcal{H}$ is separable under (A1) (the kernel $K$ is
translation-invariant with $\widehat K\in L^1$, so $\mathcal{H}$
embeds into a separable space of continuous functions), Pettis' theorem
gives strong measurability.  For the norm estimate, write
\[
 \|\Phi_r\|_{\mathcal{H}}
 \;\le\;\int_{\mathbb{R}^d}p(r,y)\,\bigl\|\mathcal{A}^u_{r,y}K(\cdot,r,y)\bigr\|_{\mathcal{H}}\,dy
 \;\le\;C\!\int_{\mathbb{R}^d}p(r,y)\bigl(1+|u(r,y)|\bigr)dy,
\]
where $C$ depends only on $\int(1+|z|)\widehat K(z)\,dz<\infty$ via
(A2) and the differential reproducing identity.  The finiteness of
this first-moment integral follows from (A2) by Cauchy--Schwarz,
$\int(1+|z|)\widehat K(z)\,dz\le \bigl(\int(1+|z|^2)\widehat K(z)\,dz\bigr)^{1/2}\bigl(\int\widehat K(z)\,dz\bigr)^{1/2}<\infty$.  By (A3) and
$u\in L^1(p\,dt\,dx)$ this bound is in $L^1([0,T])$, so
$\int_0^s\|\Phi_r\|\,dr<\infty$.  Therefore the Bochner integral
$\int_0^s\Phi_r\,dr$ defines an element of $\mathcal{H}$, and its
evaluation at any $\xi$ agrees with $\Psi_s(\xi)$ by the reproducing
property applied under the integral sign.  Combining Steps (i)--(iii), we get 
\[
 H_s^u = \Psi_s - \Theta_s + \Theta_0 \in\mathcal{H}.
\]

Step (iv). For every $\varphi\in\mathcal{H}$,
\begin{equation}\label{eq:Hs-pairing}
 \langle\varphi,H_s^u\rangle
 = \int_0^s\!\!\int_{\mathbb{R}^d}p(r,y)\mathcal{A}^u_{r,y}\varphi(r,y)\,dy\,dr
 - \int_{\mathbb{R}^d}p(s,y)\varphi(s,y)\,dy
 + \int_{\mathbb{R}^d}p(0,y)\varphi(0,y)\,dy.
\end{equation}
The three summands on the right-hand side are treated similarly.
Since $\langle\varphi,\cdot\rangle$ is a bounded linear functional on
$\mathcal{H}$, it commutes with Bochner integration of any integrable
$\mathcal{H}$-valued map
\citep[Theorem~3.7]{Hytonen2016}.
Applied to the map $y\mapsto p(s,y)K(\cdot,s,y)$, whose Bochner
integral is $\Theta_s\in\mathcal{H}$ by Step (ii), and using the
reproducing property
$\langle\varphi,K(\cdot,s,y)\rangle=\varphi(s,y)$, we obtain
\[
 \langle\varphi,\Theta_s\rangle
 = \Bigl\langle\varphi,\int p(s,y)K(\cdot,s,y)\,dy\Bigr\rangle
 = \int p(s,y)\langle\varphi,K(\cdot,s,y)\rangle\,dy
 = \int p(s,y)\varphi(s,y)\,dy.
\]
For the Stein-operator term, the differential form of the reproducing
property \citep[Theorem~10.45]{Wendland2005}
$\langle\varphi,\mathcal{A}^u_{r,y}K(\cdot,r,y)\rangle=\mathcal{A}^u_{r,y}\varphi(r,y)$
together with commutation of $\langle\varphi,\cdot\rangle$ with the
Bochner integral gives
\[
 \langle\varphi,\Psi_s\rangle
 = \int_0^s\!\!\int p(r,y)\langle\varphi,\mathcal{A}^u_{r,y}K(\cdot,r,y)\rangle\,dy\,dr
 = \int_0^s\!\!\int p(r,y)\mathcal{A}^u_{r,y}\varphi(r,y)\,dy\,dr.
\]
The Bochner integrability required in each step follows from the norm
estimates in Steps (i)--(iii), which are uniform on $[0,s]$ under
(A2) and (A3).  Subtracting the $\Theta_s$ term and adding the
$\Theta_0$ term gives \eqref{eq:Hs-pairing}.

Step (v). The right-hand side of \eqref{eq:Hs-pairing} is
precisely the functional $L_s(\varphi)$ of \eqref{eq:Lphi}, which
characterizes $u\in\mathcal{U}_0$ via \eqref{eq:3.2} of
Proposition~\ref{prop:weak-eq}: $u\in\mathcal{U}_0$ if and only if
$L_s(\varphi)=0$ for every $\varphi\in\mathcal{H}$ and every
$s\in[0,T]$.  By \eqref{eq:Hs-pairing} this is equivalent to
$\langle\varphi,H_s^u\rangle=0$ for every $\varphi\in\mathcal{H}$ and
every $s$, which, since $\mathcal{H}$ is a Hilbert space, is
equivalent to $H_s^u=0$ for every $s\in[0,T]$.
\end{proof}

Since $u\in\mathcal{U}_0$ is equivalent to $\|H_t^u\|=0$ for every
$t\in[0,T]$, we replace the hard constraint of $(P)$ by a quadratic
penalty and consider the minimization problem
\begin{equation}
\tag{$P_{\lambda}$}\label{eq:Plambda}
 \inf_{u\in\mathcal{U}} J_{\lambda}(u)
\end{equation}
with 
\[
 J_{\lambda}(u)=\int_0^T\int_{\mathbb{R}^d}|u(t,x)|^2p(t,x)\,dx\,dt + \lambda\int_0^T\|H_t^u\|^2\,dt.
\]
The weight $\lambda>0$ trades off kinetic energy against admissibility.
The subsequent subsections derive a sample-based estimator of the
penalty $\int_0^T\|H_t^u\|^2\,dt$.

Let $\{\varepsilon_n\}_{n=1}^{\infty}$ and
$\{\lambda_n\}_{n=1}^{\infty}$ be positive sequences satisfying
$\lim_{n\to\infty}\varepsilon_n=0$ and $\lim_{n\to\infty}\lambda_n=+\infty$,
respectively, and take $u_n\in\mathcal{U}$ to be an
$\varepsilon_n$-optimal solution for \eqref{eq:Plambda} with
$\lambda=\lambda_n$:
\begin{equation}\label{eq:eps-opt}
 J_{\lambda_n}(u_n)\le \inf_{u\in\mathcal{U}}J_{\lambda_n}(u) + \varepsilon_n.
\end{equation}
\begin{rem}[Relation to the structural theory]
\label{rem:no-structure}
Under confining assumptions on the marginal flow, the deterministic
theory of \citet[Theorem~2.4]{Mikami2000} and
\citet[Theorem~3.8]{Mikami2021} gives existence and uniqueness of the
minimizer together with the gradient structure
$u^*=\nabla_x\psi$.  Theorem~\ref{thm:convergence} uses none of
this.  Existence and uniqueness are part of its conclusion, by the
convexity argument in Step~(iv) of the proof, and $u^*$ itself
serves as the test element there, so the statement is free of any
confining assumption.  Uniqueness holds up to $p\,dt\,dx$-null sets,
the natural equivalence for~\eqref{eq:P0} because the kinetic energy
integrand vanishes on $\{p=0\}$.
\end{rem}

\begin{thm}\label{thm:convergence}
Suppose that {\rm (A1)}--{\rm (A3)} hold and $V_0(\mu)<\infty$.  Then
\begin{align}
 &\lim_{n\to\infty}\lambda_n\int_0^T\|H_t^{u_n}\|^2\,dt = 0, \label{eq:conv-H}\\
 &\lim_{n\to\infty}J_{\lambda_n}(u_n) = V_0(\mu), \label{eq:conv-J}\\
 &\sup_{\lambda>0}\inf_{u\in\mathcal{U}}J_\lambda(u) = V_0(\mu). \label{eq:sup-inf}
\end{align}
Moreover \eqref{eq:P0} admits a minimizer $u^*\in\mathcal{U}_0$,
unique up to $p(t,x)\,dt\,dx$-null sets, and
\begin{equation}\label{eq:conv-u-L2}
 u_n \longrightarrow u^*\quad\text{in }L^2([0,T]\times\mathbb{R}^d; \mathbb{R}^d, p(t,x)\,dt\,dx),
 \quad\text{as }n\to\infty.
\end{equation}
\end{thm}

\begin{rem}[Scope of Theorem~\ref{thm:convergence}]
\label{rem:scope}
Theorem~\ref{thm:convergence} applies to each synthetic experiment of
Section~\ref{sec:numerical}, Experiments~1--5.  The Gaussian kernel satisfies
(A1) and (A2), and each marginal family there is generated by pushing
$\mu_0$ forward along a bounded velocity field known in closed form,
so the family is smooth with integrable time derivative, which is
(A3), and the generating field is feasible for \eqref{eq:P0} with
finite kinetic energy, so $V_0(\mu)<\infty$.  For the real data of
Experiment~6 the hypotheses on the marginal flow cannot be checked, and the
theorem is not claimed there.  Experiment~7 (Nelson, $\sigma=1$)
is presented in Section~\ref{sec:exp-stochastic} as a computational
demonstration that the estimator framework extends to $\sigma>0$.  The
convergence theory of Theorem~\ref{thm:convergence} is stated for
$\sigma=0$ and does \emph{not} apply there.
\end{rem}

\begin{proof}[Proof sketch]
The proof follows the convergence analysis for kernel-based Schr{\"o}dinger bridges in \citet{nak:2024sb}, 
modified here to accommodate constraints imposed at every time point rather than solely at the two endpoints. 
The argument proceeds in four steps.  Step~(i) rules out
$\limsup_n\lambda_n\mathcal{J}_n>0$ by a subsequence argument.
Step~(ii) extracts a weak limit $u_\infty$ of $\{u_n\}$ in
$L^2(p\,dt\,dx)$, identifies its residual as the $L^1$ function
$-\partial_t p$, and concludes $u_\infty\in\mathcal{U}_0$ together with
$\mathcal{E}_n\to V_0(\mu)$.  Step~(iii) derives the sup-inf identity~\eqref{eq:sup-inf} from
\eqref{eq:conv-J} and the $\varepsilon_n$-optimality of $u_n$.  Step~(iv) establishes existence and uniqueness of the
minimizer by convexity and upgrades weak to strong convergence using
$u^*$ as its own test element.  The full proof is in
Appendix~\ref{app:proof-convergence}.
\end{proof}

\begin{rem}
Step (iv) works with the weak convergence of $\{u_n\}$ in
$L^2(p\,dt\,dx)$ obtained in Step (ii), rather than with the RKHS
pairing $\langle\varphi,H_T^{u_n}\rangle$.  The argument needs a test
element of $L^2(p\,dt\,dx)$, and $u^*$ belongs to that space
because $V_0(\mu)<\infty$.  The vanishing penalty
$\|H_t^{u_n}\|\to 0$ enters earlier, in Step (ii), where it places
every weak limit of $\{u_n\}$ in $\mathcal{U}_0$.
\end{rem}

\subsection{Representations of the penalty term}
\label{sec:penalty-expansion}

We derive a computable expression for the penalty $\int_0^T\|H_t^u\|^2\,dt$ appearing in \eqref{eq:Plambda}.
The reproducing property of $\mathcal{H}$ states
\[
 \mathcal{A}^u_{t,x}\varphi(t,x) = \langle \varphi,\, \mathcal{A}^u_{t,x}K(\cdot,(t,x))\rangle, \quad \varphi\in\mathcal{H}.
\]
Using this, the Bochner-integral identity for Hilbert-valued
integrands~\citep[Proposition~1.2.4]{Hytonen2016} together with the
Bochner integrability established in Step~(iii) of the proof of
Theorem~\ref{thm:Ht-in-H} yield, for any Borel measurable $Q\subset [0,T]\times\mathbb{R}^d$,
\begin{equation}\label{eq:RKHS-norm-integral}
 \Bigl\|\int_Q p(y)\mathcal{A}^u_y K(\cdot,y)\,dy\Bigr\|^2
 = \int_{Q\times Q}p(y)p(y')\,\mathcal{A}^u_y\mathcal{A}^u_{y'}K(y,y')\,dy'\,dy.
\end{equation}

\begin{prop}\label{prop:Ht-norm}
Let $(X_t)_{0\le t\le T}$ and $(\tilde{X}_t)_{0\le t\le T}$ be independent continuous processes such that $X_t,\tilde{X}_t\sim\mu_t$ for any $t\in [0,T]$.
Then, under {\rm (A1)}, {\rm (A2)} and {\rm (A3)} we have
\begin{equation}\label{eq:Ht-six-terms}
 \|H_t^u\|^2 \;=\; I_1(t) + I_2(t) + I_3(t) - 2\,I_4(t) - 2\,I_5(t) + 2\,I_6(t), \quad 0\le t\le T, 
\end{equation}
where 
\begin{equation}\label{eq:Ik-defs}
\begin{aligned}
 I_1(t) &= \int_0^t\!\!\int_0^t\mathbb{E}\bigl[\mathcal{A}^u_{s}\mathcal{A}^u_{r}K(s,X_s,r,\tilde{X}_r)\bigr]\,ds\,dr,
 & I_2(t) &= \mathbb{E}\bigl[K(t,X_t,t,\tilde{X}_t)\bigr],\\
 I_3(t) &= \mathbb{E}\bigl[K(0,X_0,0,\tilde{X}_0)\bigr],
 & I_4(t) &= \int_0^t\!\mathbb{E}\bigl[\mathcal{A}^u_{s}K(t,X_t,s,\tilde{X}_s)\bigr]\,ds,\\
 I_5(t) &= \mathbb{E}\bigl[K(0,X_0,t,\tilde{X}_t)\bigr],
 & I_6(t) &= \int_0^t\!\mathbb{E}\bigl[\mathcal{A}^u_{s}K(0,X_0,s,\tilde{X}_s)\bigr]\,ds.
\end{aligned}
\end{equation}
\end{prop}
It is straightforward to see that $H_t^u$ decomposes into an interior
term and two boundary terms, and that expanding the squared norm
yields six cross-terms, three of which are independent of $u$.  
The full calculation is provided in Appendix~\ref{app:penalty-expansion}.


Hereafter we further assume that $K(y,y')=\phi(|y-y'|)$ for some even $\phi\in C^2(\mathbb{R})$.
Since $\phi'$ is odd, $\phi'(0)=0$, and we define
\[
 \phi_1(r):=\frac{\phi'(r)}{r}, \qquad
 \phi_2(r):=\frac{\phi_1'(r)}{r}.
\]
Both $\phi_1$ and $\phi_2$ extend continuously to $r=0$.
With $\sigma=0$, the operator is $\mathcal{A}^u=\partial_t + u^{\mathsf{T}}\nabla_x$, and we write $\tilde{u}(y)=(1,u_1(y),\ldots,u_d(y))^{\mathsf{T}}$ for $y=(t,x)$. Then
\begin{equation}\label{eq:AK-sigma0}
 \mathcal{A}^u_y K(y',y) = \phi_1(|y-y'|)\,\tilde{u}(y)^{\mathsf{T}}(y-y'),
\end{equation}
and
\begin{equation}\label{eq:AA-sigma0}
 \mathcal{A}^u_y\mathcal{A}^u_{y'}K(y,y')
 = -\phi_2(|y-y'|)\,(y-y')^{\mathsf{T}}\tilde{u}(y)\,\tilde{u}(y')^{\mathsf{T}}(y-y')
   -\phi_1(|y-y'|)\,\tilde{u}(y)^{\mathsf{T}}\tilde{u}(y').
\end{equation}

\begin{rem}
In the stochastic extension ($\sigma>0$, Section~\ref{sec:stochastic}),
the operator $\mathcal{A}^u$ acquires a Laplacian term and
\eqref{eq:AA-sigma0} gains additional terms involving
$\phi_3(r):=\phi_2'(r)/r$, $\phi_4(r):=\phi_3'(r)/r$, and powers of $\sigma^2$.
\end{rem}

\subsection{Sample estimator of the penalty}
\label{sec:estimator}
\label{sec:U-statistic}

We construct a consistent (Monte Carlo) estimator of $\int_0^T\|H_t^u\|^2\,dt$ for the approximate problem.
The quantities $I_2(t), I_3(t), I_5(t)$ in \eqref{eq:Ik-defs} do not
depend on $u$ and are therefore dropped when optimizing over~$u$.
The remaining $u$-dependent part is
\begin{equation}\label{eq:Jhat-population}
 \mathcal{Q}(u) \;:=\; \int_0^T\bigl[\,I_1(t) - 2\,I_4(t) + 2\,I_6(t)\,\bigr]\,dt.
\end{equation}
By the change of the order of integration, the bulk contribution $\int_0^T I_1(t)\,dt$ is given by 
\[
 \int_0^T I_1(t)\,dt = \int_0^T\!\!\int_0^t\!\!\int_0^t \mathbb{E}[\mathcal{A}^u_s\mathcal{A}^u_r K]\,ds\,dr\,dt
 = \int_0^T\!\!\int_0^T w(s,r)\,\mathbb{E}[\mathcal{A}^u_s\mathcal{A}^u_r K(s,X_s,r,\tilde{X}_r)]\,ds\,dr,
\]
where $w(s,r):=T - s\vee r$.

Similarly, the boundary contribution $\int_0^T\bigl[-2 I_4(t)+2 I_6(t)\bigr]dt$ is given by 
\begin{align*}
 \int_0^T\bigl[-2 I_4(t)+2 I_6(t)\bigr]dt &= 
  \int_0^T\Bigl[-2\int_0^t\mathbb{E}[\mathcal{A}^u_s K(t,X_t,s,\tilde{X}_s)]\,ds
   + 2\int_0^t\mathbb{E}[\mathcal{A}^u_s K(0,X_0,s,\tilde{X}_s)]\,ds\Bigr]dt \\
 &= 2\int_0^T\!\int_s^T\mathbb{E}\bigl[\mathcal{A}^u_s\bigl\{K(0,X_0,s,\tilde{X}_s) - K(t,X_t,s,\tilde{X}_s)\bigr\}\bigr]\,dt\,ds.
\end{align*}

To compute the expectations in \eqref{eq:Jhat-population}, we 
draw i.i.d.~samples $t_1,\ldots,t_M\sim \mathrm{Unif}[0,T]$ and, for each $t_m$, draw iid samples $X_{t_m}^{(1)},\ldots,X_{t_m}^{(N)}\sim\mu_{t_m}$.
Additionally draw i.i.d.~samples $X_0^{(1)},\ldots,X_0^{(N_0)} \sim \mu_0$.
Write $y_{m,i}:=(t_m, X_{t_m}^{(i)})$.

There are $MN$ sample points in total, which we index as $y_p$,
$p=1,\ldots,MN$, with $y_{(m-1)N+i}=y_{m,i}$.
For two indices $p=(m-1)N+i$ and $q=(l-1)N+j$, define the weight
$\hat{w}_{pq}:=T - t_p\vee t_q$.
Here $\mathbf{1}[A]$ denotes the indicator, equal to $1$ if condition~$A$ holds and $0$ otherwise.
The estimator of $\mathcal{Q}(u)$ is
\begin{equation}\label{eq:Jhat}
\begin{aligned}
 \hat{\mathcal{Q}}(u)
 &= \frac{T^2}{M^2 N^2}\sum_{\substack{p,q=1\\p\neq q}}^{MN}
   \hat{w}_{pq}\,(\mathcal{A}^u_y\mathcal{A}^u_{y'}K)(y_p,\,y_q) + \frac{2T}{MNN_0}\sum_{p=1}^{MN}\sum_{k=1}^{N_0}
   (T - t_p)\,
   \mathcal{A}^u_{y_p}\,K(z_{0,k},\,y_p) \\
 &\quad - \frac{2T^2}{M^2 N^2}\sum_{\substack{p,q=1\\p\neq q}}^{MN}
   \mathbf{1}[t_p\le t_q]\,
   \mathcal{A}^u_{y_p}\,K(y_q,\,y_p),
\end{aligned}
\end{equation}
where $z_{0,k}:=(0,X_0^{(k)})$.
The three sums estimate $\int_0^T I_1(t)\,dt$, $\int_0^T I_6(t)\,dt$, and $\int_0^T I_4(t)\,dt$, respectively.
The $p\neq q$ exclusion removes only the self-pair $(m{=}l,\,i{=}j)$,
whose kernel evaluation $K(y_p,y_p)$ would otherwise introduce a
$u$-independent constant that does not affect the optimum but inflates
the loss.  The $M^2N^2$ normalization, used in place of the strictly
unbiased $M(M{-}1)N^2$, retains the same-time cross-particle pairs
$(m{=}l,\,i{\neq}j)$, trading strict unbiasedness for a substantially
lower variance at finite $(M,N)$.  We verify that the resulting bias
is $O(1/M)$.  Denote by $\hat I_1$ the first sum in \eqref{eq:Jhat}
(the bulk estimator) scaled by $T^2/(M^2N^2)$.  Recall that each
flat index corresponds to a (time-slice, particle) pair via
$p=(m{-}1)N+i$ and $q=(l{-}1)N+j$, so that $t_p=t_q$ iff $m=l$.
The $M^2N^2{-}MN$ off-diagonal pairs thus split into $MN(N{-}1)$
\emph{same-time cross-particle} pairs with $m=l$ and $i\ne j$, whose
two points share a time slice but differ in particle, and $M(M{-}1)N^2$
\emph{different-time} pairs with $m\ne l$, whose two time draws are
independent.  
Under the present sampling scheme, the expectation of a single different-time
pair is equal to $B/T^2$, where
\[
 B\;:=\;\int_0^T\!\!\int_0^T\! w(s,r)\,
 \mathbb{E}\bigl[\mathcal{A}^u_s\mathcal{A}^u_r K(s,X_s,r,\tilde X_r)\bigr]\,ds\,dr
 \;=\;\int_0^T\! I_1(t)\,dt
\]
is the quantity that $\hat I_1$ is designed to estimate.  The
expectation of a single same-time pair is equal to $A$, where
\[
 A\;:=\;\frac{1}{T}\int_0^T\!(T-s)\,
 \mathbb{E}\bigl[\mathcal{A}^u_s\mathcal{A}^u_s K(s,X_s,s,\tilde X_s)\bigr]\,ds.
\]
Multiplying by the number of pairs of each type and by the prefactor
$T^2/(M^2N^2)$ gives
\[
 \mathbb{E}[\hat I_1]
 \;=\;\Bigl(1-\frac{1}{M}\Bigr)\,B
 \;+\;\frac{T^2(N{-}1)}{MN}\,A,
 \qquad
 \mathbb{E}[\hat I_1]-B
 \;=\;\frac{1}{M}\Bigl[\frac{T^2(N{-}1)}{N}\,A - B\Bigr]
 \;=\;O(1/M).
\]
An analogous computation for the cross-time boundary estimator (the
third sum in \eqref{eq:Jhat}) yields the same $O(1/M)$ bias.  The
boundary term involving $X_0$ (the middle sum) is already an unbiased
Monte Carlo average because its time and $X_0$ draws are independent.
Hence
\begin{equation}\label{eq:Qhat-bias}
 \mathbb{E}\bigl[\hat{\mathcal Q}(u)\bigr]-\mathcal Q(u)=O(1/M).
\end{equation}

The derivation above assumes that $t_1,\ldots,t_M$ are i.i.d.\
uniform on $[0,T]$.  Several experiments use a uniform grid of times
instead, and Section~\ref{sec:exp-eb} uses the four observation days.
The estimator is the same in all three cases, because the time sums
in $\hat{\mathcal Q}$ are quadrature rules for the time integrals in
$\mathcal Q$.  Only the error analysis changes.  For i.i.d.\ times
the bias is~\eqref{eq:Qhat-bias}.  For a uniform grid the time sums
form a midpoint rule with error $O(1/M)$, and only the particles are
random.  For the observation days, $M=4$ and the times are fixed and
unevenly spaced, so neither analysis covers this case.  This is a
limitation of Section~\ref{sec:exp-eb}.

Equation~\eqref{eq:Qhat-bias} is a statement about the expectation
alone, and it should not be read as convergence of the random variable
$\hat{\mathcal Q}(u)$.  For the estimator itself, fix
$u\in L^2(p\,dt\,dx)$, the regime in which the energy term
of~\eqref{eq:PlambdaMN} is finite.  The first and third sums
in~\eqref{eq:Jhat} are then $U$-statistics of degree two in the $M$
independent time blocks with integrable kernels, so the strong law for
$U$-statistics applies to them as $M\to\infty$ with $N$ fixed.  The
samples from $\mu_0$ enter the middle sum through an outer average
whose fluctuation shrinks only as $N_0$ grows.  Consequently
\[
 \hat{\mathcal Q}(u)\longrightarrow\mathcal Q(u)
 \qquad\text{almost surely as } M,N_0\to\infty \text{ with } N \text{ fixed},
\]
for each fixed $u$.  The convergence is therefore pointwise in $u$, and
it does not control the minimizers of $\hat{\mathcal Q}$, whose
convergence remains open.  Theorem~\ref{thm:convergence} is a
population result.  It concerns $J_\lambda$ with the marginal flow
known and $\lambda\to\infty$, and it is separate from the sampling
error analyzed here.

Consequently, given $\lambda>0$ and samples as above, we solve
\begin{equation}
\tag{$P_{\lambda}^{M,N}$}\label{eq:PlambdaMN}
 \min_{u\in\mathcal{U}}\;\;
 \frac{T}{MN}\sum_{m=1}^M\sum_{i=1}^N|u(y_{m,i})|^2
 + \lambda\,\hat{\mathcal{Q}}(u),
\end{equation}
where the kinetic energy is likewise estimated by the sample average.
The velocity field $u$ is parametrized by a neural network $u_\theta$ and $(P_{\lambda}^{M,N})$ is optimized by stochastic gradient descent with fresh mini-batches at each iteration.

\subsection{Stochastic extension}
\label{sec:stochastic}

The RKHS framework extends naturally to the stochastic setting.
When the underlying dynamics are governed by the SDE
$dX_t = u(t,X_t)\,dt + \sigma\,dW_t$ with known diffusion coefficient
$\sigma>0$, the continuity equation~\eqref{eq:continuity} is replaced by the
Fokker--Planck equation
$\partial_t p + \mathrm{div}(up) = (\sigma^2/2)\Delta p$,
and the operator in \eqref{eq:3.1} becomes
\begin{equation}\label{eq:Au-sigma}
 \mathcal{A}^u\varphi
 = \partial_t\varphi + u^{\mathsf{T}}\nabla_x\varphi
   + \frac{\sigma^2}{2}\Delta\varphi.
\end{equation}
This is the \emph{Nelson problem} of stochastic optimal transportation
\citep{Nelson1985,Mikami1990,Mikami2021}:
\begin{equation}
\tag{$\mathrm{NP}_\sigma$}
   V_\sigma(\mu)
  := \inf_{u\in\mathcal{U}_\sigma}
     \int_{[0,T]\times\mathbb{R}^d} |u(t,x)|^2\,p(t,x)\,dx\,dt,
\end{equation}
where $\mathcal{U}_\sigma$ consists of drifts satisfying the weak
Fokker--Planck equation with $\mu_t$-marginals at all times.

All derivations earlier in this section carry through with the modified
operator~\eqref{eq:Au-sigma}: the kernel formula for
$\mathcal{A}^u_y\mathcal{A}^u_{y'}K$ acquires additional terms involving
$\phi_3$ and $\phi_4$ (the higher-order derivatives of the radial kernel
function) and powers of $\sigma^2$, while the $H_t^u$-based penalty
retains the same structure with the $(\sigma^2/2)\Delta$
contribution.  Because $\mathcal{A}^u_y\mathcal{A}^u_{y'}K$ now involves
derivatives of $K$ up to fourth order in the spatial variables, the
stochastic extension requires the stronger regularity conditions
$\phi\in C^4(\mathbb{R})$ and $\int_{\mathbb{R}^d}|z|^4\,\widehat{K}(z)\,dz<\infty$
on the radial profile $\phi$ and the Fourier transform $\widehat{K}$ of the
kernel, in place of the $C^2$ and $\int|z|^2\widehat{K}(z)\,dz<\infty$
assumptions used in the deterministic case.  The sample estimator
of~\eqref{eq:Jhat} and the approximate problem~\eqref{eq:PlambdaMN}
keep the same form, and only the closed-form expressions of the
integrand change.
The zero-noise limit~\eqref{eq:Veps-to-V0} guarantees that
$(\mathrm{NP}_\sigma)$ converges to $(P)$ as $\sigma\to 0$, so the
deterministic problem of this paper is the natural foundation on
which the stochastic extension rests.

\begin{rem}
In the stochastic setting, Lavenant et al.\ \citep{lav:2024}
developed an alternative approach based on entropy minimization with
respect to a Brownian reference measure, yielding a convex optimization
problem with consistency guarantees.  Their framework requires
$\sigma>0$ and does not extend to the deterministic case.  Our RKHS
approach and theirs are thus complementary: the present method is the
only available solver for $\sigma=0$ with all-time marginals, while for
$\sigma>0$ both approaches are applicable with different computational
trade-offs.
\end{rem}

\section{Numerical illustration}
\label{sec:numerical}

All source code for the experiments reported in this section is
publicly available at
\begin{center}
  \url{https://github.com/yumiharu-nakano/alltimeOT}
\end{center}

\paragraph{Two error metrics.}
Throughout this section the accuracy of a recovered velocity is
reported in two ways.  The \emph{grid} $\mathrm{MSE}$ averages
$|\hat u-u^*|^2$ uniformly over a rectangle in $(t,x)$, which is the
convention used by the trajectory-inference literature we compare
against.  The \emph{probability-weighted} error
\begin{equation}\label{eq:pweighted}
 \mathcal{D}(\hat u,u^*)
 :=\int_0^T\!\!\int_{\mathbb{R}^d}|\hat u(t,x)-u^*(t,x)|^2\,p(t,x)\,dx\,dt
\end{equation}
is the quantity that Theorem~\ref{thm:convergence} controls, and it
weights each region by the mass the marginal flow actually places
there.  The two differ whenever the error has structure in $x$: a
uniform grid gives equal weight to the tails, where few samples
constrain the fit and where a misfit has little effect on the observed
marginals.  For a linear-in-parameters model
$u_w=w^{\mathsf T}\Phi$, \eqref{eq:pweighted} is available in closed
form as $(\hat w-w^*)^{\mathsf T}G(\hat w-w^*)$ with the feature Gram
matrix $G=\int_0^T\!\int\Phi\Phi^{\mathsf T}p\,dx\,dt$.  We keep the
grid $\mathrm{MSE}$ in the baseline comparisons, so that the baselines
are measured on the convention they were designed for, and we report
\eqref{eq:pweighted} as well for every model that is linear in its
parameters.

\subsection{Experiment~1: 1d Gaussian translation}
\label{sec:exp-gaussian}

We validate the RKHS-penalized estimator \eqref{eq:PlambdaMN} on the simplest
nontrivial instance of the all-time-marginal ODE transport problem.

Let $d=1$, $T=1$, and consider the Gaussian translation
$\mu_t = \mathcal{N}(-1+2t,\,1)$.
The covariance is constant and the mean moves linearly, so the unique
minimum-energy solution of $(P)$ is the constant drift
$u^*(t,x)=2$.

We use the affine model $u(t,x)=w_0+w_1 t+w_2 x$ ($3$~parameters), for which
the true weight vector is $w^*=(2,\,0,\,0)$.
The objective $(P_{\lambda}^{M,N})$ is minimized by L-BFGS-B on the
ensemble-averaged loss $\bar{\mathcal{L}}(w) = K_{\mathrm{ens}}^{-1}\sum_{k=1}^{K_{\mathrm{ens}}}\mathcal{L}^{(k)}(w)$,
where each $\mathcal{L}^{(k)}$ is evaluated on an independent data draw.
Averaging over $K_{\mathrm{ens}}=30$ draws reduces the Monte Carlo variance of the
gradient sufficiently for the quasi-Newton method to converge.
Hyperparameters: $\lambda=1000$, $M=50$ time slices (uniform grid),
$N=25$ particles per time, $N_0=50$ initial particles, and the isotropic
Gaussian kernel
$K(\xi,\eta)=\exp\!\bigl(-|\xi-\eta|^2/(2h^2)\bigr)$
on $\mathbb{R}^{d+1}$ with bandwidth $h=1$, used as the reproducing kernel
throughout all experiments unless stated otherwise.  The Gaussian
kernel automatically satisfies the structural assumptions of
Section~\ref{sec:3-residual}: its Fourier transform
$\widehat K(z)\propto e^{-h^2|z|^2/2}$ is strictly positive on
$\mathbb{R}^{d+1}$, so (A1) holds (universality
\citep{Sriperumbudur2010}), and the polynomial moments
$\int(1+|z|^k)\widehat K(z)\,dz$ are finite for every $k\ge 0$, so
(A2) is met in both the $C^2$ (deterministic) and $C^4$ (stochastic)
versions.

All four random initializations converge to the same minimizer
\[
  \hat{w} = (1.808,\;0.277,\;-0.080)
  \quad\text{vs.}\quad
  w^* = (2,\;0,\;0),
\]
giving drift grid $\mathrm{MSE}(\hat{u},u^*)=0.044$ on the evaluation grid
$(t,x)\in[0,1]\times[-4,4]$.
Figure~\ref{fig:exp1-drift} shows the learned drift at five time slices:
the learned curve is very close to the constant $u^*=2$, with small
finite-sample deviations concentrated in the tails of the evaluation grid.

To assess distributional accuracy, we simulate $5{,}000$ particles from
$\mu_0$ under the learned ODE $\dot{X}_t=\hat{u}(t,X_t)$ with $1{,}000$
Euler steps.
Table~\ref{tab:exp1-w2} reports the Wasserstein-$2$ distance $W_2$ between
the simulated empirical marginal and the true $\mu_t$, computed from
sorted quantiles (exact in $d=1$).  We use $W_2$ here because the
all-time OT objective~\eqref{eq:P0-intro} is a quadratic-cost
Benamou--Brenier energy, so $W_2$ is its natural metric.  All $W_2$
values are below $0.1$, confirming accurate marginal recovery, and
Figure~\ref{fig:exp1-marginal} compares the simulated histograms with
the prescribed densities.

\begin{table}[htbp]
\centering
\caption{Experiment~1 ($d=1$, Gaussian translation): Wasserstein-$2$ distances.}
\label{tab:exp1-w2}
\begin{tabular}{lcccccc}
\toprule
$t$ & $0.00$ & $0.25$ & $0.50$ & $0.75$ & $1.00$ & mean \\
\midrule
$W_2(\hat\mu_t,\mu_t)$ & 0.020 & 0.033 & 0.055 & 0.075 & 0.090 & 0.055 \\
\bottomrule
\end{tabular}
\end{table}

\begin{figure}[htbp]
\centering
\includegraphics[width=\textwidth]{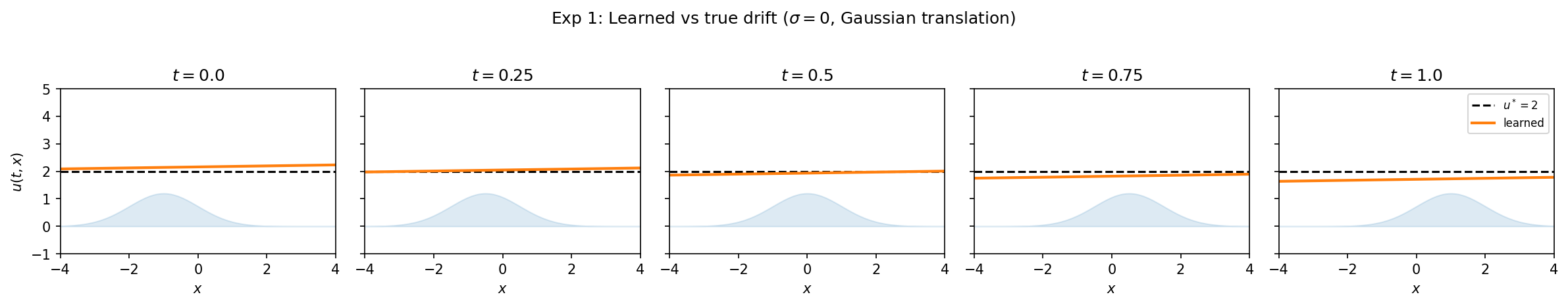}
\caption{Experiment~1: learned drift $\hat u(t,x)$ (solid orange) vs.\ true $u^*=2$ (dashed black) at $t\in\{0,0.25,0.5,0.75,1.0\}$.  Blue shading shows the scaled density $\mu_t$.}
\label{fig:exp1-drift}
\end{figure}

\begin{figure}[htbp]
\centering
\includegraphics[width=\textwidth]{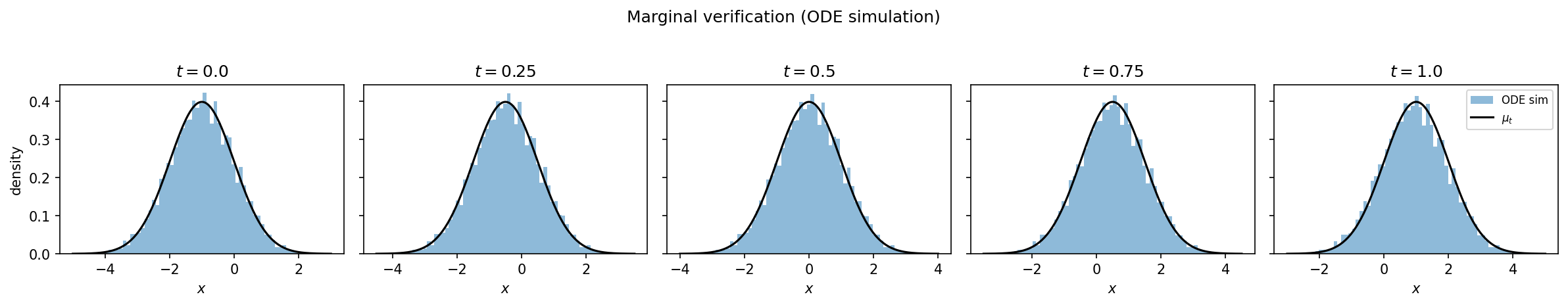}
\caption{Experiment~1: marginal verification.  Histograms of $5{,}000$ ODE-simulated particles (blue) vs.\ true $\mu_t=\mathcal{N}(-1+2t,1)$ (black) at five time slices.}
\label{fig:exp1-marginal}
\end{figure}

\subsection{Experiment~2: Roundtrip motion --- all-time vs.\ two-marginal}
\label{sec:exp-roundtrip}

This experiment shows that all-time marginal constraints carry strictly
more information than two-marginal constraints.

Let $d=1$, $T=1$, and consider the roundtrip flow
$\mu_t = \mathcal{N}(2\sin(\pi t),1)$.
Since $\mu_0=\mu_1=\mathcal{N}(0,1)$, the two-marginal OT problem returns
the identity map (i.e.\ $u\equiv 0$), yet the all-time marginals reveal the
non-trivial time-varying drift $u^*(t,x)=2\pi\cos(\pi t)$.

We parametrize the drift as $u(t,x) = w_0 + w_1 t + w_2 t^2 + w_3 x$
($4$~parameters), giving a quadratic-in-$t$ approximation to $u^*$.
The $L^2$-best quadratic fit of $2\pi\cos(\pi t)$ on $[0,1]$ has
weights $(7.64,\,-15.27,\,0.00)$ (approximately linear, since
$\cos(\pi t)$ is antisymmetric about $t=1/2$).
In the optimization we use the same ensemble-averaged L-BFGS-B scheme as Experiment~1
with $\lambda=1000$, $M=50$, $N=25$, $N_0=50$, $K_{\mathrm{ens}}=30$.

All initializations converge to
$\hat w=(7.20,-14.92,1.48,0.00)$, close to the $L^2$-optimal quadratic fit
with $w_3\approx 0$ (the learned drift is independent of $x$, as expected).
Figure~\ref{fig:exp2-time} compares the time profiles at $x=0$.
Table~\ref{tab:exp2-w2mmd} reports two complementary metrics: the
Wasserstein-$2$ distance $W_2(\hat\mu_t,\mu_t)$ (sorted quantile-based,
exact in $d=1$) and the kernel Maximum Mean Discrepancy $\mathrm{MMD}$
computed with the same Gaussian kernel $K(\xi,\eta)=\exp(-|\xi-\eta|^2/2)$ used
in the RKHS penalty.  Reporting MMD alongside $W_2$ keeps the
evaluation metric consistent with the training objective.
The all-time drift attains mean $W_2=0.111$ and mean $\mathrm{MMD}=0.036$
versus $0.971$ and $0.376$ for the two-marginal baseline
($9\times$ and $10\times$ improvements, respectively), with the largest
gap at $t=0.5$, where $W_2$ is $0.19$ against $1.99$ and
$\mathrm{MMD}$ is $0.053$ against $0.741$.  The two-marginal solution agrees
with $\mu_t$ only at the endpoints.  In drift grid MSE the all-time
method gives $0.64$ vs.\ $20.13$ for zero drift ($32\times$ reduction).

The residual misfit visible in Figure~\ref{fig:exp2-marginal} at $t=0.5$
and $t=1.0$ reflects the limited expressiveness of the $4$-parameter
quadratic-in-$t$ model: $u^*=2\pi\cos(\pi t)$ is not exactly
representable by a quadratic, so even the $L^2$-best fit
$7.64-15.27\,t$ deviates from $u^*$ by up to $\pm1.36$ at the endpoints,
and this bias accumulates in the ODE integration to produce the small
residual $W_2$ values in Table~\ref{tab:exp2-w2mmd}.  Richer dictionaries
(e.g.\ a tanh basis as in Experiment~3, or a neural network as in
Section~\ref{sec:exp-stochastic}) reduce the residual further.  The
minimal quadratic model is kept here to isolate the effect of
all-time marginal information from the effect of model expressiveness.

\begin{table}[htbp]
\centering
\caption{Experiment~2 (Roundtrip): $W_2(\hat\mu_t,\mu_t)$ and
$\mathrm{MMD}(\hat\mu_t,\mu_t)$ at five time slices.}
\label{tab:exp2-w2mmd}
\begin{tabular}{llccccc|c}
\toprule
& $t$ & $0.00$ & $0.25$ & $0.50$ & $0.75$ & $1.00$ & mean \\
\midrule
\multirow{3}{*}{$W_2$}
 & True $u^*$                    & 0.020 & 0.020 & 0.021 & 0.022 & 0.023 & 0.021 \\
 & All-time learned              & 0.020 & 0.067 & 0.193 & 0.023 & 0.252 & 0.111 \\
 & 2-marginal ($u{\equiv}0$)     & 0.020 & 1.409 & 1.995 & 1.409 & 0.020 & 0.971 \\
\midrule
\multirow{3}{*}{$\mathrm{MMD}$}
 & True $u^*$                    & 0.000 & 0.022 & 0.028 & 0.000 & 0.024 & 0.015 \\
 & All-time learned              & 0.000 & 0.000 & 0.053 & 0.000 & 0.128 & 0.036 \\
 & 2-marginal ($u{\equiv}0$)     & 0.000 & 0.555 & 0.741 & 0.561 & 0.022 & 0.376 \\
\bottomrule
\end{tabular}
\end{table}

\begin{figure}[htbp]
\centering
\includegraphics[width=0.65\textwidth]{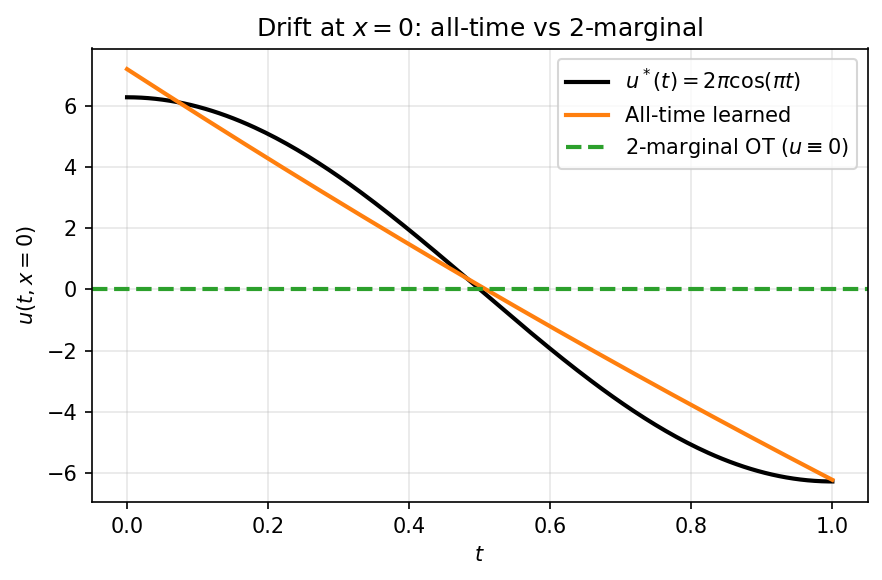}
\caption{Experiment~2 (Roundtrip): drift at $x=0$ as a function of~$t$.
The all-time learned drift (orange) closely tracks $u^*=2\pi\cos(\pi t)$ (black),
while the two-marginal OT solution $u\equiv 0$ (green dashed) is entirely flat.
This demonstrates that all-time marginal constraints recover the nontrivial
optimal velocity even when $\mu_0=\mu_1$.}
\label{fig:exp2-time}
\end{figure}

\begin{figure}[htbp]
\centering
\includegraphics[width=\textwidth]{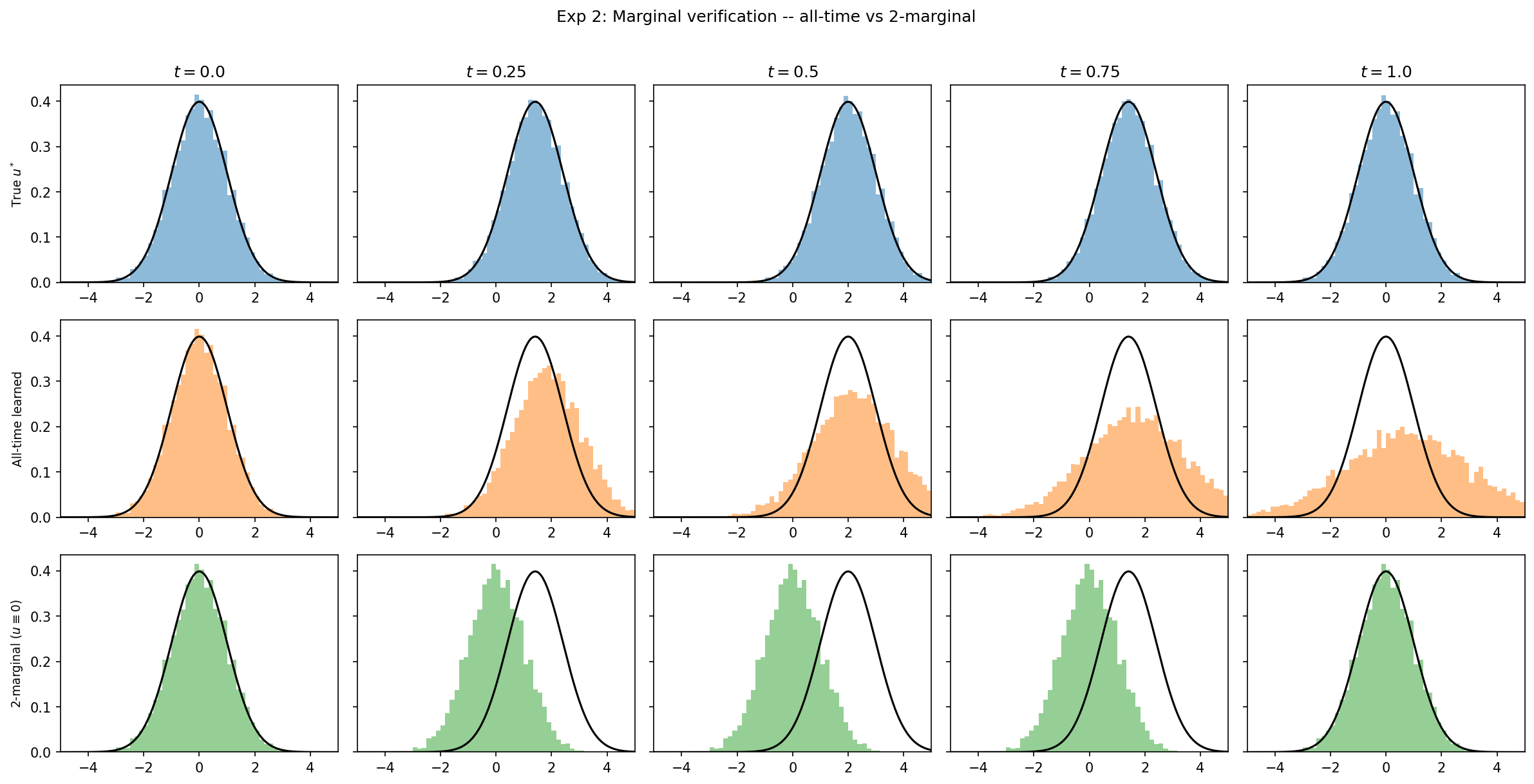}
\caption{Experiment~2: marginal densities at $t\in\{0,0.25,0.5,0.75,1\}$.
Rows from top to bottom show the true drift, the all-time learned drift
and the two-marginal baseline.
At $t=0.5$, the two-marginal baseline remains centered at zero while the true
distribution has shifted to $\mathcal{N}(2,1)$.}
\label{fig:exp2-marginal}
\end{figure}

We compare with three families of methods.
(i)~\emph{Waddington-OT} (WOT)~\citep{Schiebinger2019} chains entropic OT
couplings between $M$ snapshots and reconstructs the drift by finite
differences, $\hat u(t_k,x)\approx(T_k(x)-x)/\Delta t$.
(ii)~The multi-marginal Monge formulation of~\citep{nak-sai:2026} (``MMOT'')
parametrizes $N{-}1$ transport maps and penalizes an MMD distance
$\gamma_K^2(T_k\sharp\mu_0,\mu_{t_k})$.  We use affine maps
$T_k(x)=A_kx+b_k$ (strictly generalizing the translation ansatz) with
$200$ samples per marginal and report the best over
$\lambda\in\{10^3,10^4,10^5\}$, $\alpha\in\{0.5,1,2\}$, three
initializations, and L-BFGS-B.
(iii)~\emph{Vanilla flow matching} (FM)~\citep{Lipman2023,Albergo2023}
regresses a velocity field onto the straight-line interpolant
$X_t=(1-t)X_0+t X_1$ with $X_0,X_1$ drawn independently from $\mu_0,\mu_1$.
The entropic OT coupling used by WOT is solved by our own
log-domain Sinkhorn implementation (so the code has no external OT
dependency).

Tables~\ref{tab:exp2-wot}--\ref{tab:exp2-mmot} summarize the results.
\emph{Flow matching} fails on this instance, giving drift MSE $=20.14$,
mean $W_2=0.979$ and mean $\mathrm{MMD}=0.395$, indistinguishable from
the trivial $u\equiv0$ baseline.  The field it regresses on is not
zero, but its projection onto the model class is.  With the
independent coupling used here, $X_0,X_1\sim\mathcal{N}(0,1)$ and
$X_t=(1-t)X_0+tX_1$, the conditional target is
\[
 \mathbb{E}[X_1-X_0\,|\,X_t]=\beta(t)\,X_t,
 \qquad \beta(t)=\frac{2t-1}{(1-t)^2+t^2}.
\]
Every method in this experiment uses the dictionary
$u_\theta(t,x)=w_0+w_1t+w_2t^2+w_3x$.  The target has zero mean, and
its coefficient on $x$ is
\[
 \frac{\int_0^1\beta(t)\,\mathrm{Var}(X_t)\,dt}{\int_0^1\mathrm{Var}(X_t)\,dt}
 =\frac{\int_0^1(2t-1)\,dt}{\int_0^1\bigl((1-t)^2+t^2\bigr)dt}=0,
\]
because $\beta(t)\mathrm{Var}(X_t)=2t-1$ is odd about $t=1/2$ while the
dictionary carries a constant coefficient on $x$.  The fit therefore
returns $(w_0,w_1,w_2,w_3)=(0.009,-0.021,0.100,-0.018)$, that is
$\hat u\approx 0$.

The experiment shows that the endpoint marginals leave the round-trip
dynamics undetermined.  The zero output is specific to this model
class rather than a property of every two-marginal method.  With $\mu_0=\mu_1$ the identity map is the
$W_2$-optimal coupling, the intermediate law is a modeling choice
rather than data, and all information about the true time-varying flow
sits in $\{\mu_t\}_{0<t<1}$.  A richer dictionary would let flow
matching reproduce $\beta(t)X_t$.  The flow of that field carries
$\mu_0$ along $\mathcal{N}\bigl(0,(1-t)^2+t^2\bigr)$, the law of the
interpolant, which is centered at the origin at every time and has
variance $1/2$ at $t=1/2$.  The intermediate marginals are therefore
still wrong, since $\mu_t=\mathcal{N}(2\sin(\pi t),1)$.
The drift error of \emph{WOT} grows rapidly with $M$, from $0.08$ at
$M{=}5$ to $24.4$ at $M{=}50$, because the $O(1/\sqrt{N})$ barycentric
noise is amplified by the factor $1/\Delta t$.  \emph{Affine-MMOT}
behaves the same way, with drift MSE rising from $0.076$ at $N{=}5$ to
$2.29$ at $N{=}50$ by the same amplification, while the learned $A_k$
stay within $4\%$ of $1$ throughout.  The all-time method is
insensitive to temporal resolution and attains drift MSE $0.64$,
because a $4$-parameter global model pools information across all
times.  Figure~\ref{fig:exp2-wot} shows the two drift profiles and
the marginal tracking side by side.
On marginal tracking, however, WOT and MMOT \emph{outperform} the
all-time method, with $W_2\approx 0.05$ against $0.11$ and
$\mathrm{MMD}\approx 0.01$ to $0.02$ against $0.04$, since they
enforce marginal matching directly while the all-time estimator
enforces it only through the residual penalty.  This is a genuine
limitation.  When the downstream task is to generate samples faithful
to the prescribed marginals at a few discrete time points, a direct
marginal-matching method is often the better choice.  The all-time
formulation is preferable when
(i)~fine temporal resolution makes $O(N)$ scaling prohibitive,
(ii)~the marginals are non-Gaussian or high-dimensional, MMOT being
reported to fail for $d\ge 5$ \citep[Table~8]{nak-sai:2026}, or
(iii)~a continuous velocity field is required rather than consecutive
transport maps.  We treat this experiment as our main quantitative
benchmark and do not repeat the WOT/MMOT comparisons in the subsequent
subsections.

\begin{table}[htbp]
\centering
\caption{Experiment~2: comparison with Waddington-OT ($N=200$ per
snapshot).  MMD uses the same Gaussian kernel ($h=1$) as the training loss.}
\label{tab:exp2-wot}
\begin{tabular}{lccc}
\toprule
Method & drift MSE & mean $W_2$ & mean $\mathrm{MMD}$ \\
\midrule
True $u^*$                & 0.000  & 0.021 & 0.002 \\
All-time (ours)           & 0.635  & 0.109 & 0.040 \\
WOT ($M{=}5$)             & 0.080  & 0.060 & 0.022 \\
WOT ($M{=}10$)            & 0.600  & 0.042 & 0.013 \\
WOT ($M{=}20$)            & 3.620  & 0.042 & 0.013 \\
WOT ($M{=}50$)            & 24.360 & 0.046 & 0.015 \\
Flow matching             & 20.144 & 0.979 & 0.395 \\
2-marginal ($u{\equiv}0$) & 20.134 & 0.971 & 0.376 \\
\bottomrule
\end{tabular}
\end{table}

\begin{figure}[htbp]
\centering
\includegraphics[width=\textwidth]{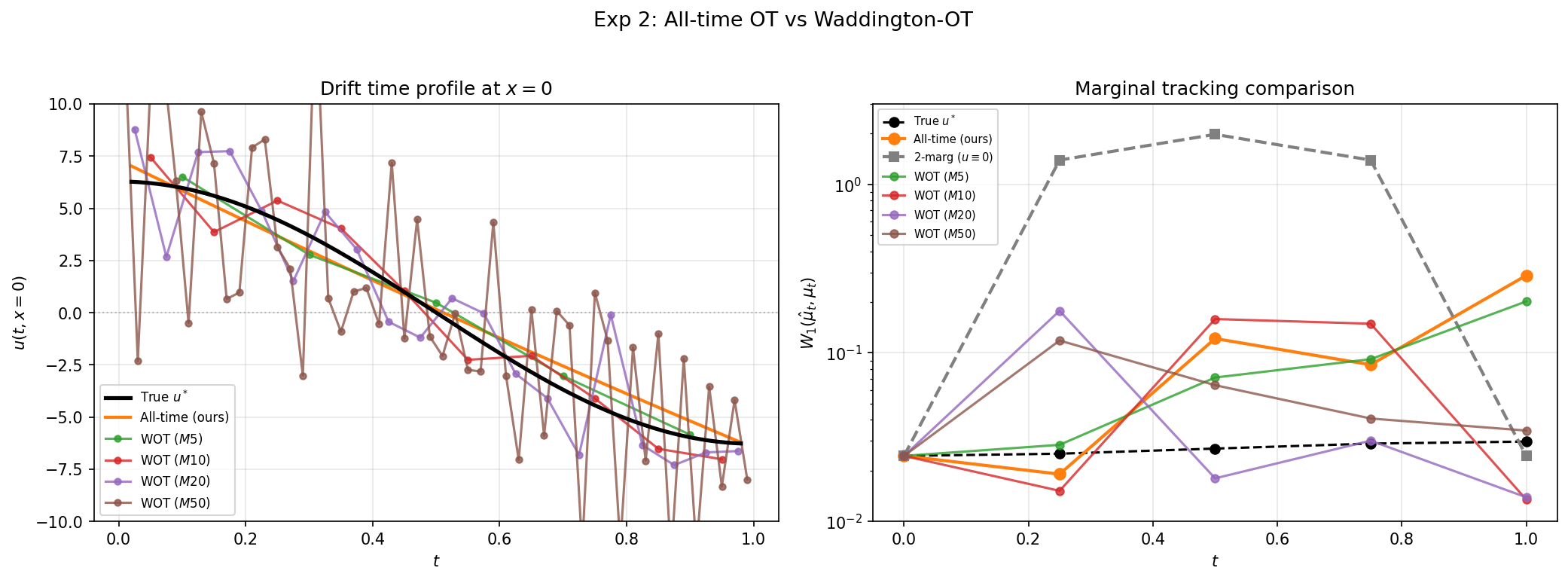}
\caption{Experiment~2: All-time OT vs.\ Waddington-OT.
Left: drift time profile at $x=0$. The all-time learned drift (orange)
closely tracks the true cosine (black), while WOT estimates become
progressively noisier as $M$ increases.
Right: marginal tracking~($W_2$).
WOT with many snapshots achieves lower $W_2$ than the all-time method,
but at the cost of degraded drift recovery.}
\label{fig:exp2-wot}
\end{figure}

\begin{table}[htbp]
\centering
\caption{Experiment~2: comparison with learned multi-marginal OT
(affine MMOT, $T_k(x)=A_kx+b_k$).
MMOT is trained from scratch on sample data without prior knowledge.
The table reports the best result across three $\lambda$ values,
three $\alpha$ values, and three initializations.  MMD uses the same
Gaussian kernel ($h=1$) as the training loss.}
\label{tab:exp2-mmot}
\begin{tabular}{lcccc}
\toprule
Method & \#params & drift MSE & mean $W_2$ & mean $\mathrm{MMD}$ \\
\midrule
True $u^*$        & ---  & 0.000  & 0.021 & 0.002 \\
All-time (ours)   & 4    & 0.635  & 0.111 & 0.036 \\
MMOT affine $N{=}5$  & 10   & 0.076  & 0.052 & 0.016 \\
MMOT affine $N{=}10$ & 20   & 0.212  & 0.047 & 0.014 \\
MMOT affine $N{=}20$ & 40   & 0.334  & 0.055 & 0.017 \\
MMOT affine $N{=}50$ & 100  & 2.292  & 0.053 & 0.020 \\
Flow matching     & 4    & 20.144 & 0.979 & 0.395 \\
2-marginal ($u{\equiv}0$) & 0 & 20.134 & 0.971 & 0.376 \\
\bottomrule
\end{tabular}
\end{table}

\subsection{Experiment~3: Bimodal merging --- nonlinear drift recovery}
\label{sec:exp-bimodal}

Experiments~1 and~2 used drifts that are linear or affine in~$x$.  To
test whether the estimator can recover a genuinely nonlinear drift, we
consider a one-dimensional bimodal flow that merges into a single
symmetric mixture:
\begin{equation}
  \mu_t = \frac{1}{2} \mathcal{N}\bigl(-2+2t, 1\bigr)
        + \frac{1}{2} \mathcal{N}\bigl( 2-2t, 1\bigr), \quad t\in[0,1].
\end{equation}
At $t=0$ the two modes are at $\pm2$, at $t=1$ they coincide at the
origin, and the velocity field must push mass from each mode toward
the center.

Let $\varphi_i(t,x)$ denote the two Gaussian densities, and set
$m_1(t)=-2+2t$, $m_2(t)=2-2t$ (so that $m_1+m_2=0$).  Rewriting
the continuity equation as $\partial_x(u\rho)=-\partial_t\rho$,
integrating in $x$ from $-\infty$, and using $\rho\to 0$ at $-\infty$
gives $u\rho=-\partial_t F$ with $F(t,x)=\int_{-\infty}^{x}\rho(t,y)\,dy$,
so that $u=-\partial_t F/\rho$.  A short calculation then yields
\begin{equation}
  u^*(t,x)
   = \frac{2(\varphi_1-\varphi_2)}{\varphi_1+\varphi_2}
   = 
  -2 \tanh \bigl(2(1-t) x\bigr).
  \label{eq:exp4-true}
\end{equation}
The target is therefore a smooth, odd, bounded nonlinear function of
$x$ whose slope decreases as the two modes merge.

We test three parametric drift classes of increasing expressiveness,
all trained under the \emph{same} all-time RKHS loss and the \emph{same}
sample budget:
\begin{align*}
  \text{(a) bilinear (4p):}\quad &\Phi=[1,\,t,\,x,\,tx],\\
  \text{(b) tanh dictionary (8p):}\quad
    &\Phi=[1,\,t,\,x,\,tx,\,
            \tanh x,\,t\tanh x,\,
            \tanh(2x),\,t\tanh(2x)],\\
  \text{(c) MLP ($\approx 2.5$k):}\quad &\text{two hidden layers of $48$ tanh units,
      input $(t,x)\in\mathbb R^2$.}
\end{align*}
Models (a) and (b) are linear in parameters and optimized by
ensemble-averaged L-BFGS-B as in Sections~\ref{sec:exp-gaussian}
and~\ref{sec:exp-roundtrip}.  Model (c) is a universal approximator
trained by Adam with minibatch SGD under the \emph{same} RKHS loss:
since the loss depends on the drift only through its sample values
$u(t_p,x_p)$, the linear and neural parametrizations share a single
loss implementation, and autograd supplies the gradient with respect
to the network weights.

The $\tanh$ dictionary is designed to approximate~\eqref{eq:exp4-true},
but cannot represent it exactly: any finite linear combination
$a(t)\tanh(2x)+b(t)\tanh(x)+\cdots$ scales the \emph{amplitude} of
$\tanh$ by a polynomial in $t$, whereas $-2\tanh\!\bigl(2(1-t)x\bigr)$
scales the \emph{argument} by $(1-t)$.  The two agree at $t\in\{0,1\}$
and in the small-$x$ linear regime, but differ in general.  The
dictionary fit is only the $L^2$-best projection onto the feature
span.  The MLP, by contrast, is not subject to this approximation gap
and serves as a control for whether the remaining error is driven by
dictionary misspecification or by finite-sample fluctuations of the
sample-based RKHS estimator.

Each of the $M=50$ time slices is sampled by drawing $N=30$ particles
from a fair coin mixture of the two Gaussians. $N_0=60$ particles are
drawn from $\mu_0$.  Hyperparameters are $h=1$, $\lambda=5000$, and
$K_{\mathrm{ens}}=20$ (linear models) / $K_{\mathrm{batch}}=3$ SGD
batches per step for $N_{\mathrm{iter}}=4000$ Adam iterations with
cosine schedule $3{\cdot}10^{-3}\!\to\!10^{-5}$ (MLP).  All three
models are trained on the identical $K_{\mathrm{ens}}=20$ pre-cached
batches.

\begin{table}[htbp]
\centering
\begin{tabular}{lccccc}
\toprule
model & \#params & drift MSE & mean $W_2$ & mean $\mathrm{MMD}$ & max $W_2$ \\
\midrule
zero drift                 & 0                 & $2.544$ & $0.756$ & $0.321$ & $1.274$ \\
bilinear                   & 4                 & $5.264$ & $0.242$ & $0.092$ & $0.475$ \\
\textbf{tanh dictionary}   & 8                 & $\mathbf{0.187}$ & $\mathbf{0.074}$ & $\mathbf{0.019}$ & $\mathbf{0.162}$ \\
\textbf{MLP}               & $\approx 2{,}500$ & $\mathbf{0.181}$ & $\mathbf{0.070}$ & $\mathbf{0.019}$ & $\mathbf{0.124}$ \\
\bottomrule
\end{tabular}
\caption{Experiment~3 (bimodal merging).  The drift grid MSE is
computed on the uniform grid $[0,1]\times[-4,4]$ against the true
drift~\eqref{eq:exp4-true}.  The $W_2$ and $\mathrm{MMD}$ columns report
the Wasserstein-2 distance and the Gaussian-kernel MMD (bandwidth
$h=1$, matching the training loss) between the ODE-simulated marginal
and $\mu_t$, averaged or maximized over $t\in\{0,0.25,0.5,0.75,1\}$.}
\label{tab:exp4}
\end{table}

Table~\ref{tab:exp4} exhibits a striking decoupling in the bilinear
column: its drift grid MSE ($5.26$) is \emph{worse} than that of
the zero drift ($2.54$), yet its mean $W_2$ ($0.24$) and mean
$\mathrm{MMD}$ ($0.09$) remain small.  The reason is geometric: a
linear-in-$x$ drift grows without bound in the tails of $[-4,4]$,
where the true drift saturates to $\mp 2$, and the grid MSE is
dominated by that extrapolation gap.  Within the data support (near
$x=\pm 2$) the bilinear fit still pushes mass toward the origin at
roughly the right speed, so the ODE-simulated marginals stay close to
$\mu_t$.  This illustrates that pointwise drift error and
distributional marginal error are genuinely different objectives, and
small $W_2$ does not imply correct drift (nor vice versa).

The tanh dictionary and the MLP recover the drift to comparable
accuracy: drift MSE $0.187$ vs.\ $0.181$ ($\approx 3\%$ relative
difference) and mean $W_2$ $0.074$ vs.\ $0.070$ ($\approx 5\%$),
despite a $\approx 300$-fold difference in parameter count.  These
gaps are on the order of the Monte Carlo standard deviation of the
penalty estimator at the $(M,N,N_0,K_{\mathrm{ens}})$ values used
here, so we do not read the small numerical gap as a meaningful
difference between the two model classes.  Two consequences follow.
First, although the tanh dictionary cannot realize
$-2\tanh(2(1-t)x)$ exactly, its $L^2$-best projection is
already within the statistical noise, so the missing argument-scaling
does not limit accuracy at this sample budget.  Second, using a
universal approximator in place of a hand-crafted dictionary yields no
additional precision, confirming that the residual error is dominated
by finite-sample fluctuations of the sample-based RKHS estimator
rather than by the expressive power of the model class.  The all-time RKHS loss is
therefore effectively \emph{model-agnostic}: the same objective drives
the convex linear fit and the non-convex neural fit to comparable
accuracy, and the choice between them can be made on grounds of
interpretability, training cost, or convex optimization guarantees
rather than recovery accuracy.

\begin{figure}[htbp]
\centering
\includegraphics[width=\textwidth]{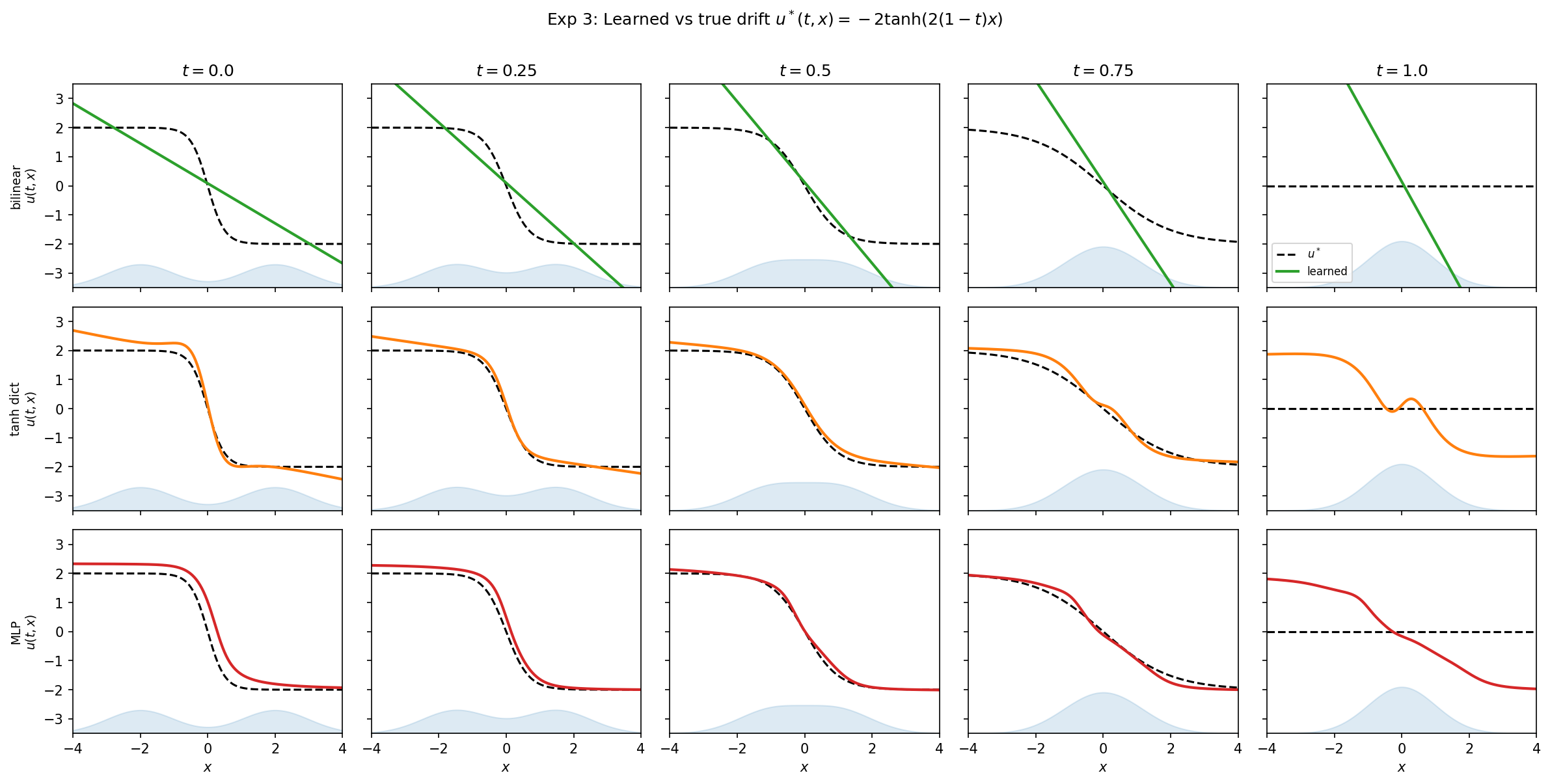}
\caption{Experiment~3: learned drift $\hat u(t,x)$ (colored) vs.\
true drift $u^*(t,x)=-2\tanh(2(1-t)x)$ (dashed black) at
$t\in\{0,0.25,0.5,0.75,1\}$.  Rows: bilinear, tanh dictionary, MLP.
Gray shading shows the bimodal density~$\mu_t$.}
\label{fig:exp4-drift}
\end{figure}

\begin{figure}[htbp]
\centering
\begin{minipage}[t]{0.48\textwidth}
\centering
\includegraphics[width=\linewidth]{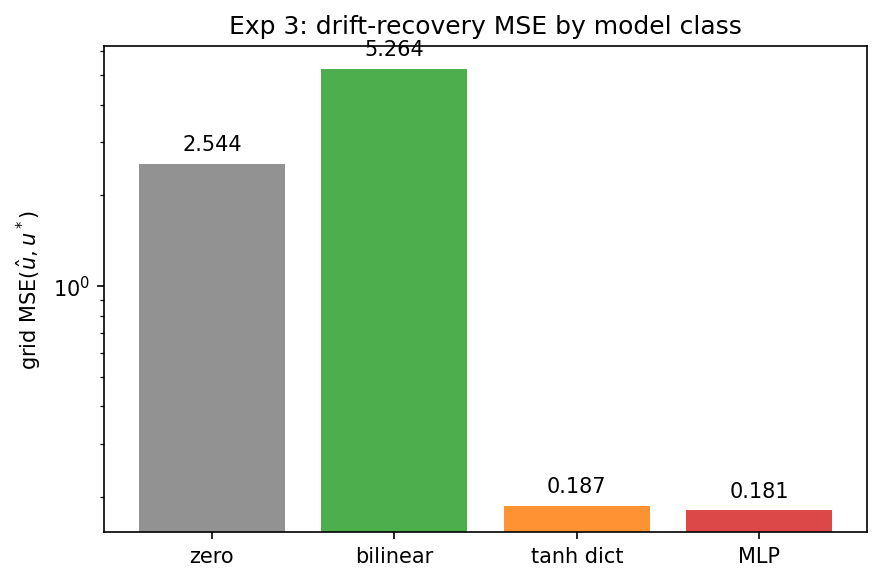}
\end{minipage}\hfill
\begin{minipage}[t]{0.48\textwidth}
\centering
\includegraphics[width=\linewidth]{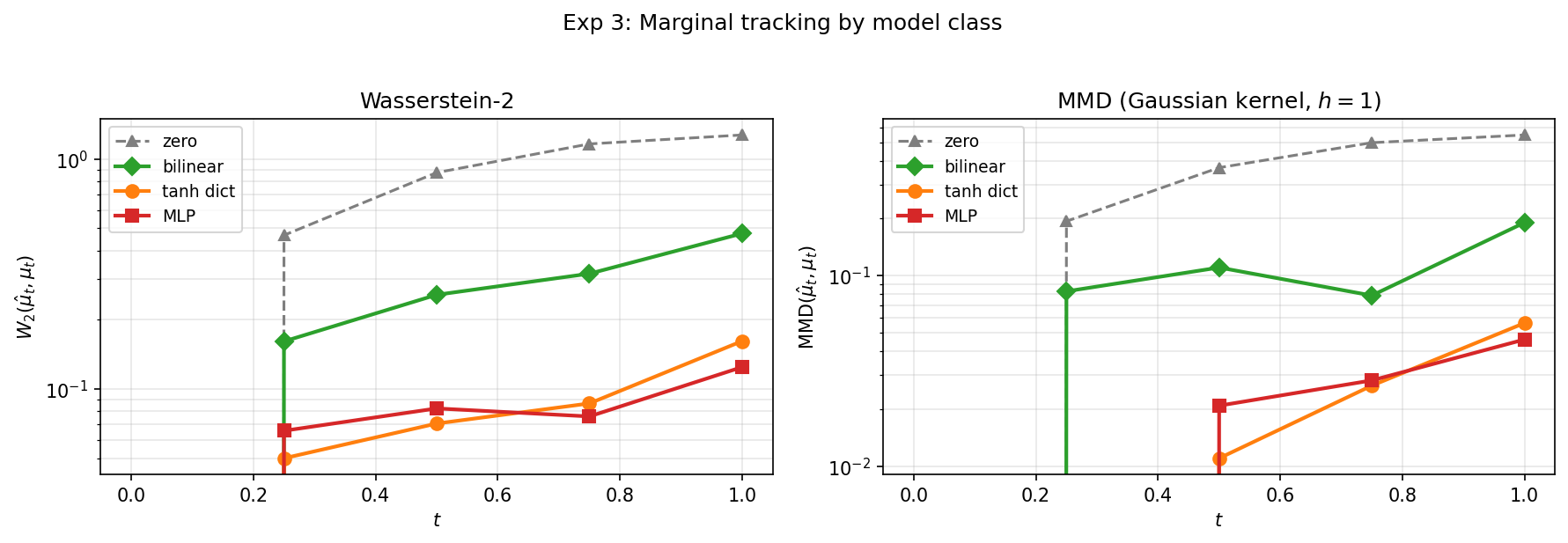}
\end{minipage}
\caption{Experiment~3: drift MSE (left, log scale) and marginal
$W_2$ / $\mathrm{MMD}$ distances (right, two panels) for the three
model classes plus the zero-drift baseline.  Note the bilinear model's
large MSE coexisting with a small $W_2$ / $\mathrm{MMD}$: the MSE is
inflated by extrapolation in the tails of the grid, but the
marginal-level transport is still broadly correct.
Figure~\ref{fig:exp4-drift} shows the recovered drift curves,
Figure~\ref{fig:exp4-metrics} places the three model classes on the
two axes at once, and Figure~\ref{fig:exp4-marginal} shows the
simulated marginals.  The tanh
dictionary and the MLP produce curves whose relative MSE gap
($\approx 3\%$) is within the Monte Carlo standard deviation of the
loss, indicating that model-class expressiveness is not the
limiting factor at this sample budget.}
\label{fig:exp4-metrics}
\end{figure}

\begin{figure}[htbp]
\centering
\includegraphics[width=\textwidth]{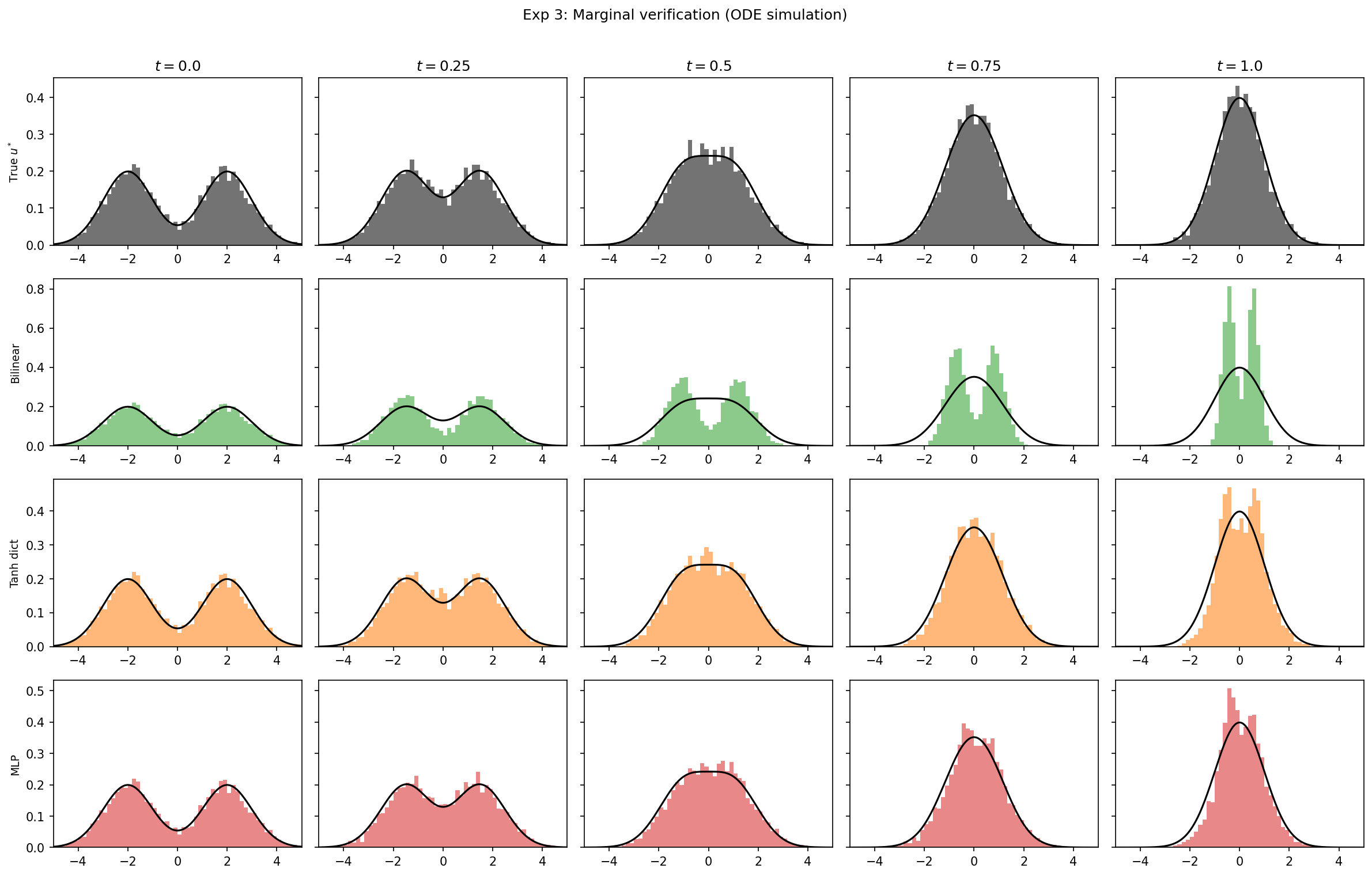}
\caption{Experiment~3: marginal verification.  Histograms of
$5{,}000$ ODE-simulated particles (colored) vs.\ the true bimodal
density~$\mu_t$ (black curve) at five time slices.  Rows: true $u^*$,
bilinear, tanh dictionary, MLP.  The tanh dictionary and the MLP
both reproduce the two modes faithfully throughout the merge.  The
bilinear fit misses the tails but tracks the bulk correctly.}
\label{fig:exp4-marginal}
\end{figure}

We re-run the same WOT and affine-MMOT baselines as in
Section~\ref{sec:exp-roundtrip} on this bimodal flow in order to
examine their behavior away from Gaussian marginals.  An affine map
cannot transport one bimodal distribution to another of the same family
(shrinking $A_k$ to merge the two modes also shrinks the per-mode
variance), so affine-MMOT incurs an irreducible misspecification error.
WOT is free of this issue.

\begin{table}[htbp]
\centering
\begin{tabular}{lrrrr}
\toprule
Method & \#params & drift MSE & mean $W_2$ & mean $\mathrm{MMD}$ \\
\midrule
All-time tanh (ours, 8p)   & 8          & $\mathbf{0.187}$  & $\mathbf{0.074}$ & $\mathbf{0.019}$ \\
All-time bilin (ours, 4p)  & 4          & $5.264$           & $0.242$          & $0.092$ \\
\midrule
MMOT affine $N{=}3$         & 6          & $0.621$           & $0.184$          & $0.090$ \\
MMOT affine $N{=}5$         & 10         & $0.868$           & $0.192$          & $0.079$ \\
MMOT affine $N{=}10$        & 20         & $1.354$           & $0.198$          & $0.068$ \\
MMOT affine $N{=}20$        & 40         & $2.336$           & $0.206$          & $0.062$ \\
\midrule
WOT $M{=}5$                & non-par.   & $2.140$           & $0.100$          & $0.027$ \\
WOT $M{=}10$               & non-par.   & $11.783$          & $0.104$          & $0.026$ \\
WOT $M{=}20$               & non-par.   & $38.692$          & $0.154$          & $0.050$ \\
WOT $M{=}50$               & non-par.   & $212.323$         & $0.231$          & $0.064$ \\
\midrule
Zero drift                 & 0          & $2.544$           & $0.756$          & $0.321$ \\
\bottomrule
\end{tabular}
\caption{Experiment~3: baseline re-assessment on the bimodal merging
flow.  Drift grid MSE is computed on $[0,1]\times[-4,4]$ against the
true drift $u^*=-2\tanh(2(1-t)x)$.  Mean $W_2$ and mean $\mathrm{MMD}$
(Gaussian kernel, $h=1$) are the distributional distances to $\mu_t$
averaged over $t\in\{0,0.25,0.5,0.75,1\}$.  The all-time method with
the $8$-feature tanh basis dominates both baselines on all three metrics.}
\label{tab:exp4-baselines}
\end{table}

\begin{figure}[htbp]
\centering
\includegraphics[width=\textwidth]{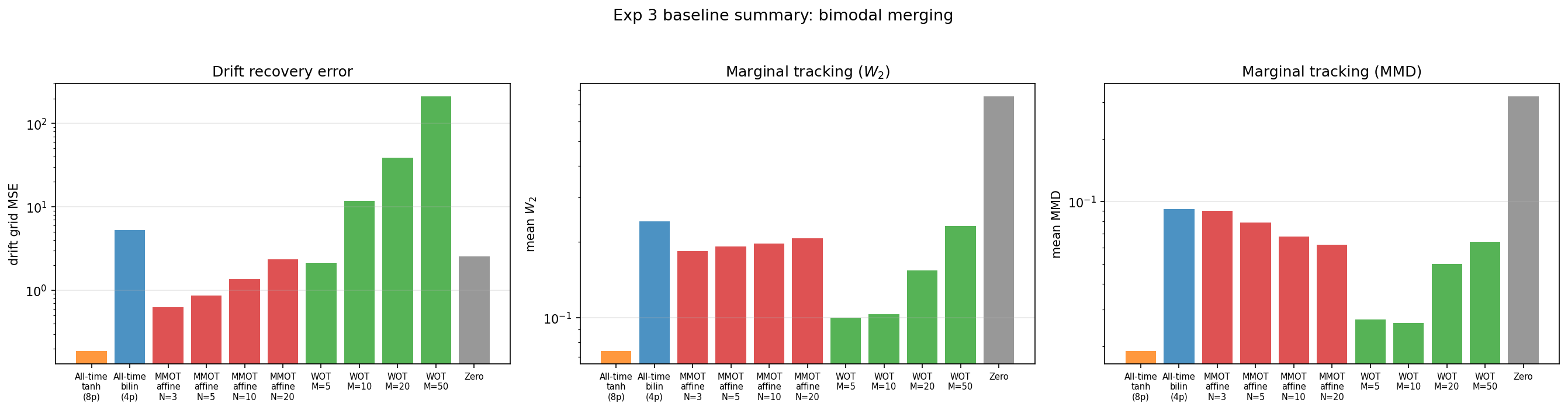}
\caption{Experiment~3 baseline summary: drift grid MSE (left),
mean marginal $W_2$ (middle), and mean $\mathrm{MMD}$ (right),
all in log scale.  The all-time tanh model is best on all three
metrics.  Affine MMOT saturates at an irreducible error from model
misspecification, and WOT's drift MSE explodes as $\Delta t$ shrinks.}
\label{fig:exp4-baselines-summary}
\end{figure}

Table~\ref{tab:exp4-baselines} summarizes the results, and
Figure~\ref{fig:exp4-baselines-summary} shows the same three metrics on
a log scale.  Affine MMOT
saturates at an irreducible error ($A_k\approx 0.51$--$0.58$, a
compromise that is incompatible with the bimodal-plus-fixed-variance
structure), with $W_2\approx 0.18$--$0.21$ and
$\mathrm{MMD}\approx 0.06$--$0.09$ essentially independent of $N$.  WOT
tracks the marginals well at moderate $M$ ($W_2\approx 0.10$,
$\mathrm{MMD}\approx 0.027$ at $M{=}5,10$), but its drift MSE explodes
as $\Delta t$ shrinks (exceeding $200$ at $M{=}50$), by the same
$O(M^2/N_{\text{sample}})$ amplification observed in
Section~\ref{sec:exp-roundtrip}.  We note that this drift MSE
combines two effects: (i) the $O(1/\Delta t)$ amplification of
sampling noise in finite-difference velocities reconstructed from
discrete couplings, and (ii) extrapolation off the data cloud, since
the evaluation grid $[0,1]\times[-4,4]$ includes regions where
$\mu_t$ has little or no mass.  Distributional metrics ($W_2$,
$\mathrm{MMD}$) restricted to the support degrade much more
gracefully.  The qualitative comparison is unchanged, but the gap to
the all-time method is smaller.  The all-time tanh model outperforms
both baselines on all three metrics (drift MSE $0.21$, $W_2=0.07$,
$\mathrm{MMD}=0.015$), while the $4$-feature bilinear all-time model is
too rigid to capture the sigmoidal drift.  Time-global variational
regularization with a well-chosen linear-in-parameters dictionary thus
offers a clearly better trade-off than pairwise-OT reconstruction for
non-Gaussian marginal dynamics.

\subsection{Experiment~4: 2d Gaussian translation}
\label{sec:exp-2d}

The preceding experiments are all in spatial dimension $d=1$.  In this
subsection we verify that the same RKHS all-time OT estimator applies
without modification to $d\ge 2$.  As discussed in
Section~\ref{sec:penalty-expansion}, the formula
$\mathcal{A}^{u}\mathcal{A}^{u'}K(y,y')
 = \bigl[-a^{2}\tau\tau' + a(1+u^{\mathsf{T}} u')\bigr]K(y,y')$,
with $\tau=\Delta t + u^{\mathsf{T}}\Delta x$, $\tau^{\prime}=\Delta t + (u^{\prime})^{\mathsf{T}}\Delta x$, $\Delta t=t-t^{\prime}$, $\Delta x=x-x^{\prime}$, and
$\Delta y=(\Delta t,\Delta x)$, is dimension-independent: only
first-order derivatives of $K$ appear, and the spatial contribution
enters only through the Euclidean inner products. Hence the $d=1$ implementation used in Experiment~1
extends to $d>1$ simply by replacing scalar products with inner
products.

Let $d=2$, $T=1$, and take
\[
  \mu_t = \mathcal{N}((-1+2t, t/2)^{\mathsf{T}},  \mathbf{I}_{d}),
  \quad t\in[0,1],
\]
where $\mathbf{I}_{d}$ denotes the $d\times d$ identity matrix (here
$d=2$). 
The Benamou--Brenier minimizing drift is the constant vector field
$u^*(t,x) = (2, 1/2)^{\mathsf{T}}$.
We use the linear-in-parameters affine dictionary
\[
  u_i(t,x) = w_{i,0} + w_{i,1}\, t + w_{i,2}\, x_1 + w_{i,3}\, x_2,
  \qquad i=1,2,
\]
with $8$ parameters in total.  The ground-truth optimum has
intercepts $(2, 1/2)$ and all other coefficients equal to $0$.

We take $h=1$, $\lambda=10^{3}$, $M=25$ time slices, $N=20$ particles
per slice, $N_0=50$ initial-distribution particles, and
ensemble-average over $K=15$ independent draws as in Experiment~1.
Optimization uses L-BFGS-B with four different initializations, and all
four converge to the same loss value
$\approx -211.94$.

The learned intercepts are $(2.50,0.11)$ versus the true $(2,0.5)$, all
other coefficients have magnitude at most $0.95$, and the drift grid
MSE is $\approx 0.21$.  Despite this per-coefficient residual,
Figure~\ref{fig:exp5-marginal} shows that the ODE-simulated samples
reproduce $\mu_t$ accurately: the sliced Wasserstein-$2$ distance
$\mathrm{SW}_2(\widehat\mu_t,\mu_t)$ remains below $0.14$ with mean
$\approx 0.099$, and the Gaussian-kernel MMD (with the same kernel as
the training loss, $h=1$) remains below $0.07$ with mean
$\approx 0.039$.  This is consistent with the remaining error being
concentrated in the identifiability valley of OT drifts that leaves the
marginal flow unchanged.  The experiment thus confirms that the
dimension-independent kernel operator formula yields a functional $d=2$
estimator.

\begin{figure}[t]
\centering
\includegraphics[width=\linewidth]{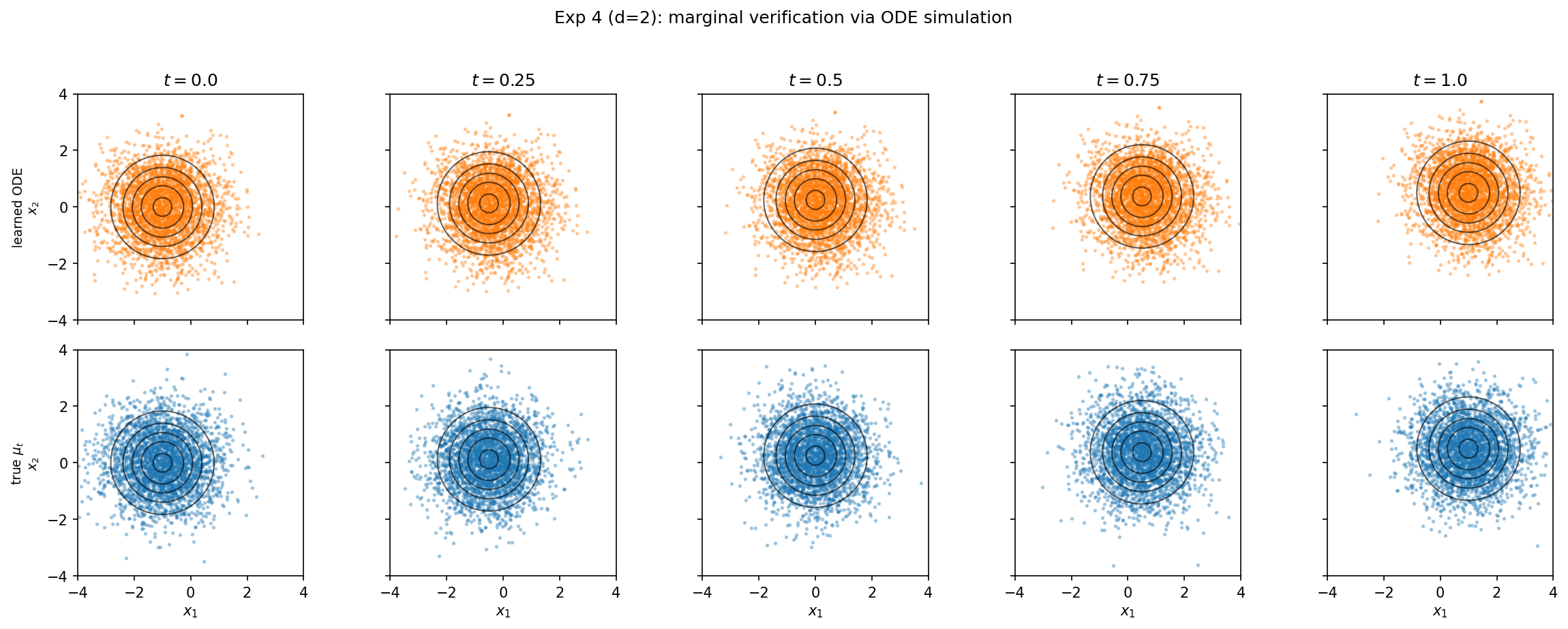}
\caption{Experiment~4: marginal verification.  The top row shows ODE-simulated
particles using the learned drift $\widehat u$, and the bottom row shows
independent reference samples from~$\mu_t$.  Black contours show the
true density of~$\mu_t$.}
\label{fig:exp5-marginal}
\end{figure}

\subsection{Experiment~5: Bimodal merging in 2D --- a time-reversed surrogate for cell-state branching}
\label{sec:exp-2d-bifurcation}

Experiment~4 demonstrated that the dimension-independent kernel
operator formula extends the estimator to $d=2$ on a simple Gaussian
translation.  In this subsection we combine two features of the
preceding experiments, nonlinear drift structure
(Section~\ref{sec:exp-bimodal}) and higher dimension
(Section~\ref{sec:exp-2d}), in a bimodal merging flow in $d=2$
that serves as a surrogate for scRNA-seq cell-state branching.
Note that the temporal direction here is the time-reverse of typical
biological branching: in our setup the bimodal initial distribution
($\pm 2$) merges to a unimodal centered cloud as $t\to T$, whereas
biological branching has unimodal progenitors diverging to multiple
fates.  Because the deterministic continuity equation
$\partial_t p + \nabla\!\cdot(up) = 0$ is time-reversal symmetric
under $t'=T-t$, $u'(t',x)=-u(T-t',x)$, the merging flow studied here
is mathematically equivalent to a branching flow under
reparametrization.  The estimator's ability to capture nonlinear
drift, identifiability valleys, and bimodal geometry therefore
transfers directly to the branching direction.

Let $d=2$, $T=1$, and
\[
  \mu_t(x_1, x_2)
    = \Bigl[\tfrac12\mathcal{N}(-2+2t,\,1) + \tfrac12\mathcal{N}(2-2t,\,1)\Bigr](x_1)
      \cdot \mathcal{N}(0,\,1)(x_2).
\]
The $x_1$-marginal is the bimodal merge of Experiment~3, and the $x_2$-marginal
is the standard normal for all $t$.  Because the $x_2$-marginal is time-invariant, the Benamou--Brenier minimizing
drift is
\[
  u_1^*(t,x) = -2\tanh\bigl(2(1-t)\,x_1\bigr),
  \qquad u_2^*(t,x) = 0.
\]

We compare three linear-in-parameters dictionaries that use the
\emph{same} feature set for both output components:
\begin{itemize}
  \item[(A)] \textbf{Affine} --- $[1,\,t,\,x_1,\,x_2]$, $8$~params.
  \item[(B)] \textbf{Bilinear} --- $[1,\,t,\,x_1,\,x_2,\,tx_1,\,tx_2]$, $12$~params.
  \item[(C)] \textbf{Tanh} --- $[1,\,t,\,x_1,\,x_2,\,tx_1,\,tx_2,\,
             \tanh x_1,\,t\tanh x_1,\,\tanh 2x_1,\,t\tanh 2x_1]$, $20$~params.
\end{itemize}
Model~(C) contains the true drift in component~$u_1$ (up to a linear
combination) and requires the estimator to learn that component~$u_2$ is
identically zero.

$h=1$, $\lambda=3\times 10^{3}$, $M=25$ time slices, $N=25$ particles
per slice, $N_0=60$ initial particles, $K=10$ ensemble replicates.
Optimization uses L-BFGS-B with multiple initializations.

Table~\ref{tab:exp6-mse} reports the drift grid MSE on
$[0,1]\times[-4,4]\times[-3,3]$, separated by output component, together
with the marginal-tracking metrics $\mathrm{SW}_2$ and $\mathrm{MMD}$
from a $4{,}000$-particle Euler simulation.
\begin{table}[h]
\centering
\begin{tabular}{lccccc}
\toprule
 Model    & total MSE & $u_1$ MSE & $u_2$ MSE & mean $\mathrm{SW}_2$ & mean $\mathrm{MMD}$ \\
\midrule
 Zero                    & $2.548$ & $2.548$ & $0.000$ & --- & --- \\
 Affine   ($8$ params)   & $\mathbf{2.182}$ & $\mathbf{2.126}$ & $\mathbf{0.057}$ & $0.122$ & $0.052$ \\
 Bilinear ($12$ params)  & $8.913$ & $8.021$ & $0.892$ & $0.162$ & $0.087$ \\
 Tanh     ($20$ params)  & $7.112$ & $5.192$ & $1.920$ & $\mathbf{0.121}$ & $\mathbf{0.050}$ \\
\bottomrule
\end{tabular}
\caption{Experiment~5 (bimodal merge in $d=2$): drift grid MSE against
$u^*=(-2\tanh 2(1{-}t)x_1,\,0)$ and marginal-tracking metrics
(sliced $W_2$ and Gaussian-kernel MMD, averaged over the five
evaluation times $t\in\{0, 0.25, 0.5, 0.75, 1\}$ of a $4{,}000$-particle
Euler simulation).  Affine minimizes drift MSE by settling on the best
linear approximation of $u_1^*$.  The tanh model attains the same marginal
accuracy but incurs a larger $u_2$-component MSE within the
identifiability valley.}
\label{tab:exp6-mse}
\end{table}

The three dictionaries trade drift accuracy against marginal accuracy in
distinct ways.  The affine model attains the lowest drift grid MSE
($2.18$) by tracking the best linear approximation of $u_1^*$ near
$x_2=0$ while leaving $u_2\approx 0$, and already matches the marginal
flow well (mean $\mathrm{SW}_2=0.122$, $\mathrm{MMD}=0.052$).  The tanh
dictionary is the only class that can reproduce the nonlinear shape of
$u_1^*$: Figure~\ref{fig:exp6-drift} shows its $u_1$-component
(orange) follows the sigmoid profile (black dashed) while the
affine and bilinear fits remain near their linear approximation.  In
exchange, the extra degrees of freedom spill into the $u_2$-component,
whose MSE is $\approx 1.92$, inflating the total to $7.11$, while
marginal tracking is essentially unchanged from the affine fit
(mean $\mathrm{SW}_2=0.121$, $\mathrm{MMD}=0.050$).

The constraint~\eqref{eq:weak-CE} is affine in $u$ and $u^*$
satisfies it, so $u^*+v$ lies in $\mathcal{U}_0$ exactly when
$\mathrm{div}(p\,v)=0$.  We call such a $v$ feasible, and adding one
leaves every $\mu_t$ unchanged.  Stationarity of the $x_2$-marginal
does \emph{not} by itself make a perturbation feasible.  For
$v=(0,\partial_{x_2}\psi)$ the condition reads
$\partial_{x_2}\bigl(p\,\partial_{x_2}\psi\bigr)=0$, and with the decay
of $p$ it forces $\partial_{x_2}\psi\equiv 0$ on $\{p>0\}$.  So such
perturbations are \emph{not} feasible here.  One family of exactly
feasible perturbations in this geometry is the rotational fields
$v=R(x-m_t)$ with $R$ skew, for which
$\mathrm{div}(v)=\mathrm{tr}(R)=0$ and
$\nabla p\cdot v=-p\,(x-m_t)^{\mathsf T}R(x-m_t)=0$, so
$\mathrm{div}(p\,v)=0$.  The observed $u_2$ error is therefore not a
feasible perturbation.  It is a direction that the penalty constrains
only weakly at finite $\lambda$ and finite sample size.  This is why
the marginal metrics barely move while the drift MSE grows.  The
bilinear class is dominated on every metric, and for a different
reason.  Without the tanh nonlinearity it cannot represent the shape
of $u_1^*$, so its $u_1$ error rises to $8.02$ against $2.13$ for
the affine fit.  That component is the one the marginals see, and the
marginal metrics register the loss ($\mathrm{SW}_2=0.162$,
$\mathrm{MMD}=0.087$).  The four extra parameters therefore do not
improve the fit in the direction that the marginals constrain.

\paragraph{Drift MSE versus marginal accuracy across experiments.}
\label{par:valley-summary}
The relationship between drift grid MSE and marginal-tracking error
varies systematically with the dimension and with the geometry of the
marginal flow.  Two levels must be kept apart.  The \emph{feasible}
set $\mathcal{U}_0$ contains $u^*+v$ for every $p$-weighted
divergence-free perturbation, $\mathrm{div}(p\,v)=0$, and in $d\ge2$
that set is non-trivial.  The \emph{minimizer} is nevertheless unique,
by the convexity argument in Step~(iv) of the proof of
Theorem~\ref{thm:convergence}.  When
$u^*$ is a gradient the cross term
$\int p\,(u^*)^{\mathsf T}v$ vanishes, so
$\int p\,|u^*+v|^2=V_0(\mu)+\int p\,|v|^2>V_0(\mu)$ for $v\neq0$.
There is therefore no population valley among minimizers, only among
feasible drifts.  At the population level $u^*$ and $u^*+v$ differ in energy by
$\int p\,|v|^2$.  With finite samples and finite~$\lambda$ the
empirical objective resolves that difference only down to its own
sampling error, so where $\int p\,|v|^2$ is small it cannot prefer
$u^*$, and the estimate moves along $v$.  The size of the perturbation set depends on
the marginal geometry.
\begin{itemize}
  \item In \emph{one dimension}
        (Sections~\ref{sec:exp-gaussian}--\ref{sec:exp-bimodal}),
        $\mathrm{div}(p\,v)=0$ together with integrability and the
        decay of $p$ forces $v\equiv 0$.  The population valley
        collapses, so drift MSE and marginal $W_2$ are tightly coupled
        and small drift error implies small marginal error.
  \item In the \emph{2D translation} setting
        (Section~\ref{sec:exp-2d}), the valley contains rotational
        perturbations $v=R(x-m_t)$ around the moving mean ($R$ skew),
        which the kinetic term suppresses but does not annihilate at
        finite~$\lambda$.  This explains the small but persistent
        cross-coordinate residual in $\widehat W$ alongside excellent
        marginal accuracy.
  \item In the \emph{2D bimodal merging} setting (this experiment), the
        $x_2$-marginal is standard normal at every $t$, so
        $\partial_{x_2}\psi$-type perturbations are admissible to high
        order, giving a much wider valley.  This is the regime where
        a flexible dictionary inflates drift MSE without affecting
        marginal $W_2$.
  \item In the \emph{dimension-scaling} sweep
        (Appendix~\ref{sec:exp-dim-scaling}), the valley dimension
        grows with~$d$, producing the per-component MSE plateau at
        roughly $0.7/d$ across $d\ge 2$.
\end{itemize}
The overall picture is that drift grid MSE under-reports estimator
quality whenever the marginal geometry admits a non-trivial valley.
Marginal-tracking metrics ($W_2$, sliced $W_2$, MMD) provide the
complementary measure that the variational objective actually
penalizes, and the two should be read together.
Figure~\ref{fig:exp6-marginals} shows the simulated marginals of the
tanh model against independent reference samples.

\begin{figure}[t]
\centering
\includegraphics[width=\linewidth]{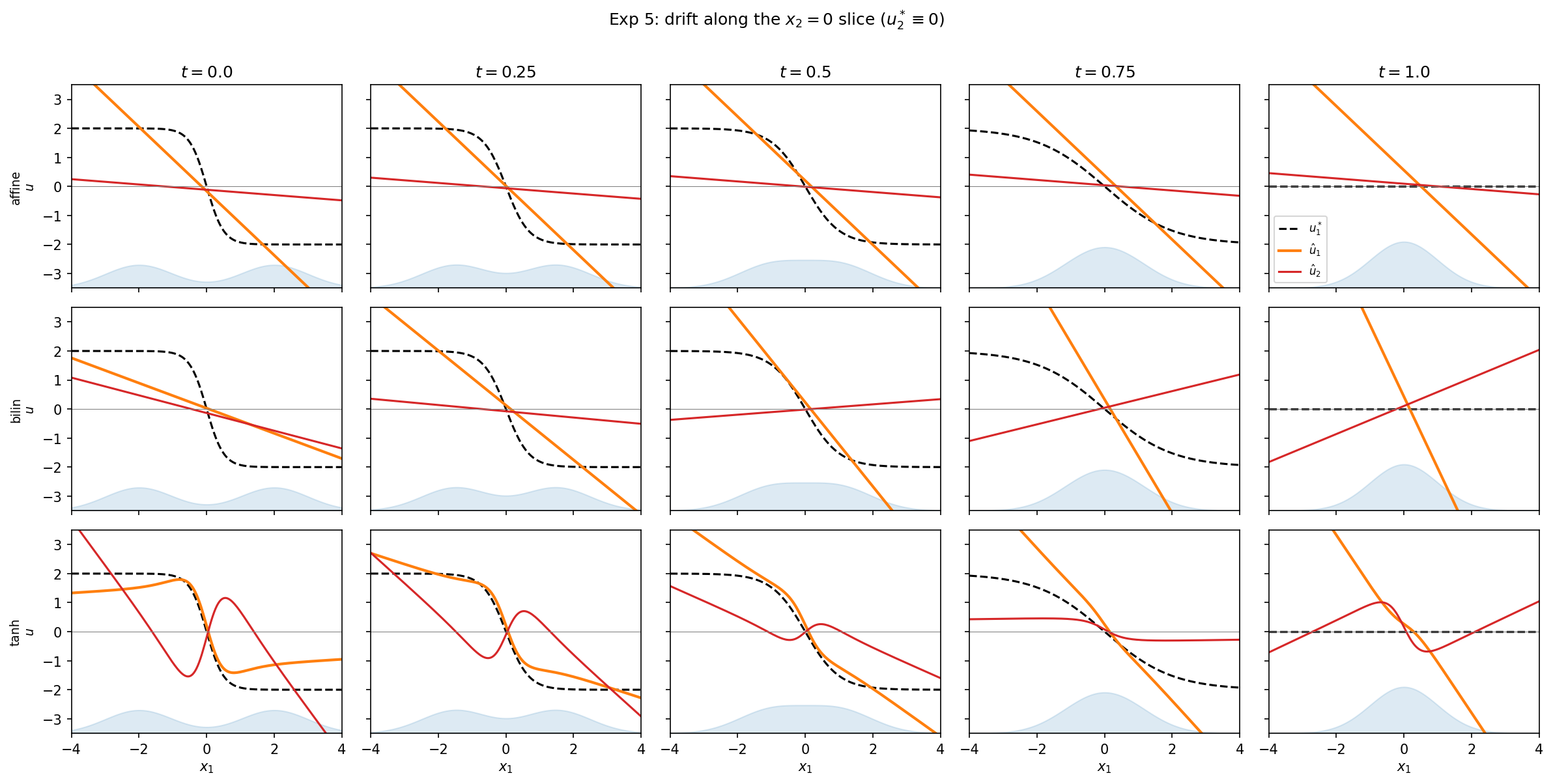}
\caption{Experiment~5: learned drift along the $x_2=0$ slice at five
times for the three model classes.  The black dashed line is the true $u_1^*$,
orange is the learned $\hat u_1$, and red is the learned $\hat u_2$.  Shaded region
shows the $x_1$-marginal density.  Only the Tanh dictionary recovers
the merging (time-reversed branching) structure.}
\label{fig:exp6-drift}
\end{figure}

\begin{figure}[t]
\centering
\includegraphics[width=\linewidth]{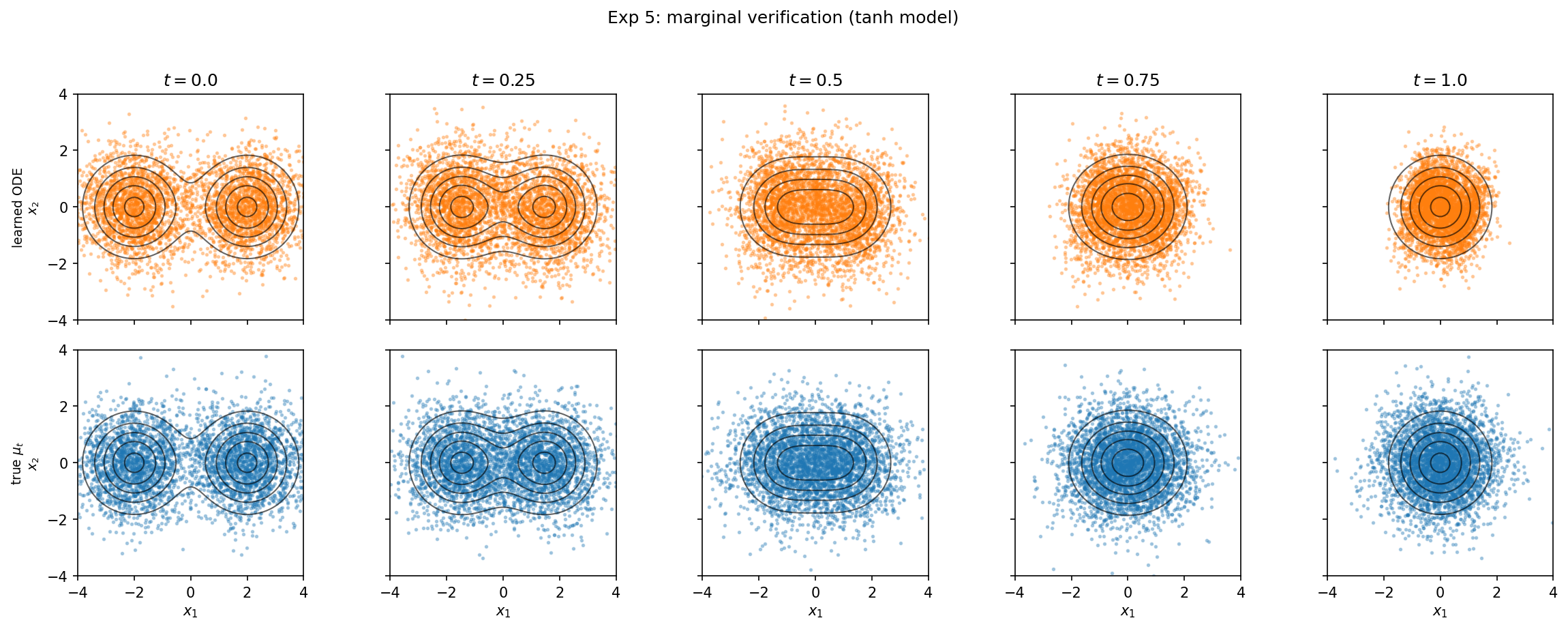}
\caption{Experiment~5 (tanh model): marginal verification.  The top row shows
ODE-simulated particles under $\hat u$, the bottom row shows independent
reference samples from $\mu_t$, and black contours are the true joint
density.}
\label{fig:exp6-marginals}
\end{figure}

\subsection{Experiment~6: real-data application --- single-cell trajectory inference}
\label{sec:exp-eb}

The deterministic experiments above are synthetic.  As a real-data
application we consider single-cell trajectory inference on the
embryoid body (EB) scRNA-seq dataset of Moon et
al.~\citep{Moon2019}, the canonical benchmark for the Waddington-OT
framework~\citep{Schiebinger2019}.  The dataset comprises five
snapshots of human embryoid body development at days 0--3, 6--9,
12--15, 18--21 and 24--27, which we map to
$t\in\{0,\,0.25,\,0.5,\,0.75,\,1\}$.  After standard preprocessing
(QC, normalization, log-transformation, top-2000 highly-variable
genes, PCA to $d=30$ components) and subsampling to $1{,}500$ cells
per time point, we obtain $7{,}500$ samples in
$\mathbb{R}^{30}$.

The all-time estimator uses an MLP drift
($1{+}d\to 96\to 96\to d$ with tanh activations, $15{,}294$
parameters), Adam ($8{,}000$ iterations, cosine LR
$5\!\times\!10^{-4}\to 5\!\times\!10^{-6}$, $\lambda=10^{4}$,
$h=8$, $M=4$, $N=80$, $K_{\mathrm{ens}}=25$).
We evaluate against the pairwise Waddington-OT baseline on two axes,
held-out marginal tracking in the canonical neighbor-interpolation
setting and bootstrap stability of the inferred velocity field.

\paragraph{Marginal-tracking accuracy (interpolation regime).}
We hold out the middle time point $t=0.5$ (day 12--15) and train
on the remaining four marginals
$\{\mu_0,\mu_{0.25},\mu_{0.75},\mu_1\}$.  Each method predicts
the held-out day-15 marginal using only training data from the two
adjacent training time points, day~9 and day~21.  Concretely:
\begin{itemize}
  \item \textbf{All-time (ours)}: forward-simulate the learned
        velocity field $\hat u$ by Euler ODE integration from a
        day-9 sample at $t=0.25$ to $t=0.5$ (one interval, $100$
        steps).
  \item \textbf{Waddington-OT} (entropic OT, log-domain
        Sinkhorn~\citep{Cuturi2013}): compute the Sinkhorn coupling
        between the day-9 and day-21 clouds and apply McCann
        (displacement) interpolation at fraction
        $\tau=(15-9)/(21-9)=0.5$.  This is the canonical WOT
        use case for inferring an intermediate marginal between two
        observed snapshots.
  \item \textbf{Zero drift}: cells remain at their day-9 position.
\end{itemize}
All methods use $1{,}500$ cells.  Unlike the synthetic experiments of
Sections~\ref{sec:exp-gaussian}--\ref{sec:exp-2d-bifurcation}, the
EB benchmark provides no ground-truth drift, so we report only
marginal-tracking metrics, the standard evaluation for
trajectory inference in single-cell
applications~\citep{Schiebinger2019}.
Marginal accuracy at the held-out day~15 is measured by the
sliced Wasserstein-2 distance ($\mathrm{SW}_2$, $200$ random
projections) and the Gaussian-kernel MMD (bandwidth equal to the
median pairwise distance of the held-out cloud) against an
independent $1{,}500$-cell reference draw from day~15.  As a
Monte-Carlo floor we report the same metrics evaluated between two
disjoint halves of the held-out sample.

\begin{table}[htbp]
\centering
\caption{Experiment~6, interpolation regime: held-out day-15
marginal-tracking metrics on the EB scRNA-seq benchmark ($d=30$).
All methods use the same $1{,}500$-cell day-9 starting sample (all-time,
zero) or day-9/day-21 sample pair (WOT).  Lower is better.  The
Monte-Carlo floor is the irreducible finite-sample error.}
\label{tab:exp-eb}
\begin{tabular}{lcc}
\toprule
Method                         & $\mathrm{SW}_2$ & $\mathrm{MMD}$ \\
\midrule
Monte-Carlo floor (true split) & $0.429$         & $0.000$        \\
All-time (ours, MLP)           & $1.003$         & $0.030$        \\
Waddington-OT (Sinkhorn)       & $\mathbf{0.775}$ & $\mathbf{0.025}$ \\
Zero drift                     & $1.092$         & $0.058$        \\
\bottomrule
\end{tabular}
\end{table}

Waddington-OT achieves the best marginal-tracking accuracy in this
regime, with $\mathrm{SW}_2 = 0.775$ versus $1.003$ for our method.
This is the regime in which Sinkhorn-WOT is structurally favored:
a single coupling between two adjacent observed clouds is exactly
the pairwise-OT primitive on which WOT is built, and McCann
interpolation at $\tau=0.5$ delivers a near-optimal displacement
midpoint.  Both methods substantially improve over the zero-drift
baseline, indicating that the inferred dynamics are informative.
Figure~\ref{fig:exp-eb} shows the predicted and held-out clouds.

\paragraph{Discussion of the marginal-tracking comparison.}
On this benchmark, pairwise WOT and the all-time estimator differ
not primarily in marginal-tracking sharpness but in what each
method produces.  Waddington-OT returns a discrete coupling between
the two snapshots used at evaluation time and achieves the best
marginal-tracking accuracy when those snapshots bracket the
held-out time.  Querying a different time point requires re-solving
an OT problem.  The all-time estimator returns a single
\emph{continuous} velocity field $\hat u(t,x)$ defined on
$[0,T]\times\mathbb{R}^{d}$, which can be evaluated, simulated and
differentiated at any $(t,x)$ without re-solving an OT problem.  On
the neighbor-interpolation regime this continuous parametrization
comes at a modest cost in marginal accuracy ($\mathrm{SW}_2$ gap
of $\sim 0.23$ to WOT).  On the synthetic experiments of
Sections~\ref{sec:exp-roundtrip}
and~\ref{sec:exp-bimodal}--\ref{sec:exp-2d-bifurcation}, where
ground-truth drift is available, the all-time estimator recovers
$u^*$ accurately whereas pairwise WOT cannot
(Section~\ref{sec:exp-roundtrip}, Table~\ref{tab:exp4-baselines}).
The EB data carry no ground-truth drift, so the next paragraph reports
a counterpart that needs none, the bootstrap stability of the inferred
velocity.

\paragraph{Bootstrap stability of the inferred velocity field.}
A stable estimator returns nearly the same velocity field when the
training data are resampled.  We refit both methods on
$K_{\mathrm{boot}}=20$ independent $80\%$-subsamples of the five
training snapshots and evaluate them at a fixed set of $1000$ query
points $(t_q,x_q)$, with $t_q$ uniform on $[0.05,0.95]$ and $x_q$
drawn from the data cloud.  We report the
per-query across-bootstrap standard deviation
$s_q:=\bigl(\sum_{d\le 30}\mathrm{Var}_k\,\hat u_k^d(t_q,x_q)\bigr)^{1/2}$
and, because the two fields differ in magnitude, the scale-free ratio
$\mathrm{rsd}_q:=\|\mathrm{std}_k\,\hat u_k\|/\|\mathrm{mean}_k\,\hat u_k\|$.
Waddington-OT is evaluated in two ways, by the nearest-neighbor rule
used in the original diagnostic and by a Nadaraya-Watson smoothing of
the same barycentric displacements with the median-distance bandwidth,
which produces a continuous field like the network and so isolates the
statistics of the estimator from the interpolation rule.

\begin{table}[htbp]
\centering
\caption{Experiment~6: bootstrap stability on EB over
$K_{\mathrm{boot}}=20$ resamples, median and interquartile range over
$1000$ queries.  Lower is more stable.  The raw dispersion $s_q$
follows the field magnitude.  On the scale-free ratio $\mathrm{rsd}_q$
the all-time estimator and nearest-neighbor Waddington-OT are close,
and smoothed Waddington-OT is the most stable.}
\label{tab:exp-eb-stability}
\begin{tabular}{lcccc}
\toprule
Method & median $s_q$ & IQR of $s_q$ & median $\mathrm{rsd}_q$ & mean $\|\hat u\|$ \\
\midrule
All-time (ours, MLP)             & $5.49$ & $[4.83,6.28]$  & $0.339$ & $16.49$ \\
Waddington-OT, nearest neighbor & $8.60$ & $[6.37,12.66]$ & $0.364$ & $28.63$ \\
Waddington-OT, smoothed          & $\mathbf{0.91}$ & $[0.76,0.93]$ & $\mathbf{0.048}$ & $17.84$ \\
\bottomrule
\end{tabular}
\end{table}

Against the nearest-neighbor rule, the all-time field has the
smaller raw dispersion, median $s_q=5.49$ against $8.60$.  This
ordering follows the field magnitudes, $16.49$ against $28.63$.  On
the scale-free ratio the two fields are close, $0.339$ against
$0.364$, so the all-time estimator is about as stable as
nearest-neighbor Waddington-OT.  The smoothed Waddington-OT field is
more stable than both by an order of magnitude, with median
$s_q=0.91$ and $\mathrm{rsd}_q=0.048$ at the comparable magnitude
$17.84$.

The diagnostic therefore measures the interpolation rule rather than
the estimator.  The nearest-neighbor rule is unstable under resampling and the
smoothed rule is stable.  The global continuous
parametrization of the all-time estimator adds no stability of its
own.

The stability of the smoothed field is the stability of a sample
mean.  Table~\ref{tab:exp-eb-smooth} varies the bandwidth of the
smoothing in the interpolation setting of Table~\ref{tab:exp-eb}.  As
the bandwidth grows, the spatial variation of the field disappears.
At the median-distance bandwidth it carries $0.9\%$ of the field
energy, so the field moves every day-9 cell by almost the same vector,
and the pure translation by the mean displacement has the same
held-out scores and the same stability.  Held-out tracking on EB does
not detect this loss of structure.  The nearest-neighbor rule, the
smoothed field and the translation all reach $\mathrm{SW}_2$ between
$0.76$ and $0.79$, the translation has the best MMD, and the all-time
estimator has $1.00$.  The held-out marginal at day 15 is therefore
explained by the bulk displacement of the cloud, and the metric is
insensitive to the velocity assigned to individual cells.  This is the
real-data form of the gap between marginal tracking and drift recovery
observed in Sections~\ref{sec:exp-bimodal}
and~\ref{sec:exp-2d-bifurcation}.

The bimodal flow of Section~\ref{sec:exp-bimodal} has a known drift, so
the accuracy of the smoothed field can be measured there.  Smoothing
the Waddington-OT displacements at the median-distance bandwidth
lowers the drift grid MSE from $1.98$ to $0.92$ at $M=5$, averaged
over five sampling seeds.  That is five times the all-time error
$0.187$ of Table~\ref{tab:exp4-baselines}, and for $M\ge10$ the
smoothed error exceeds the zero-drift error $2.54$, reaching $3.4$,
$14.1$ and $84$ at $M=10$, $20$ and $50$.  Smoothing therefore buys
stability with a field that is close to constant in space on EB and
far from the true drift on the synthetic flow.

\begin{table}[htbp]
\centering
\caption{Experiment~6: held-out accuracy and bootstrap stability of
the smoothed Waddington-OT field as a function of the smoothing
bandwidth $h=c\,h_{\mathrm{med}}$, where $h_{\mathrm{med}}=13.3$ is the
median pairwise distance among day-9 cells.  The coupling is
Sinkhorn(day 9, day 21) as in Table~\ref{tab:exp-eb}, and the
prediction of day 15 moves each day-9 cell by half its smoothed
displacement.  The spatial fraction is the share of the field energy
carried by its spatial variation.  Stability is measured over
$K_{\mathrm{boot}}=20$ resamples of $80\%$ of the day-9 and day-21
cells at $1000$ fixed day-9 query cells.  The row $c=0$ is the
nearest-neighbor rule and reproduces the Waddington-OT row of
Table~\ref{tab:exp-eb}.  The translation row moves every cell by the
mean displacement.  The last two rows repeat Table~\ref{tab:exp-eb}
for reference.}
\label{tab:exp-eb-smooth}
\begin{tabular}{lccccc}
\toprule
Field & $\mathrm{SW}_2$ & $\mathrm{MMD}$ & spatial fraction & median $\mathrm{rsd}_q$ & median $s_q$ \\
\midrule
$c=0$, nearest neighbor & $0.775$ & $0.025$ & $0.64$  & $0.203$ & $3.67$ \\
$c=0.2$                  & $0.762$ & $0.024$ & $0.58$  & $0.060$ & $0.90$ \\
$c=0.5$                  & $0.768$ & $0.022$ & $0.21$  & $0.032$ & $0.44$ \\
$c=1$, rule of Table~\ref{tab:exp-eb-stability} & $0.788$ & $0.020$ & $0.009$ & $0.033$ & $0.45$ \\
translation              & $0.785$ & $0.019$ & $0$     & $0.035$ & $0.48$ \\
\midrule
All-time (ours, MLP)     & $1.003$ & $0.030$ & --      & --      & -- \\
Zero drift               & $1.092$ & $0.058$ & --      & --      & -- \\
\bottomrule
\end{tabular}
\end{table}

\paragraph{Comparison with the global trajectory-inference
framework of Lavenant et al.}
A complementary global approach to pairwise WOT is the framework
of Lavenant et al.~\citep{lav:2024}, called gWOT.  gWOT minimizes
relative entropy with respect to a Brownian reference measure
under all-time-marginal constraints.  Unlike pairwise WOT, gWOT
is globally defined and does not share the bracketing
limitation.  It applies at held-out time points outside any
training-snapshot pair.  The structural separation between gWOT
and the all-time RKHS estimator is along the noise axis rather
than the bracketing axis.  gWOT is formulated as a
Schr\"odinger bridge with positive noise level $\sigma > 0$ and
does not reduce to the deterministic case.  Our method is
designed for $\sigma = 0$, that is Problem~$(P)$.  To our
knowledge it is the only available solver in this regime.
Section~\ref{sec:stochastic} discusses how the two approaches
address complementary problem classes by the value of $\sigma$.
A direct numerical comparison on EB requires either running gWOT
at small $\sigma$ or extending the present estimator via the
stochastic framework, and is left to future work.

\begin{figure}[htbp]
\centering
\includegraphics[width=\linewidth]{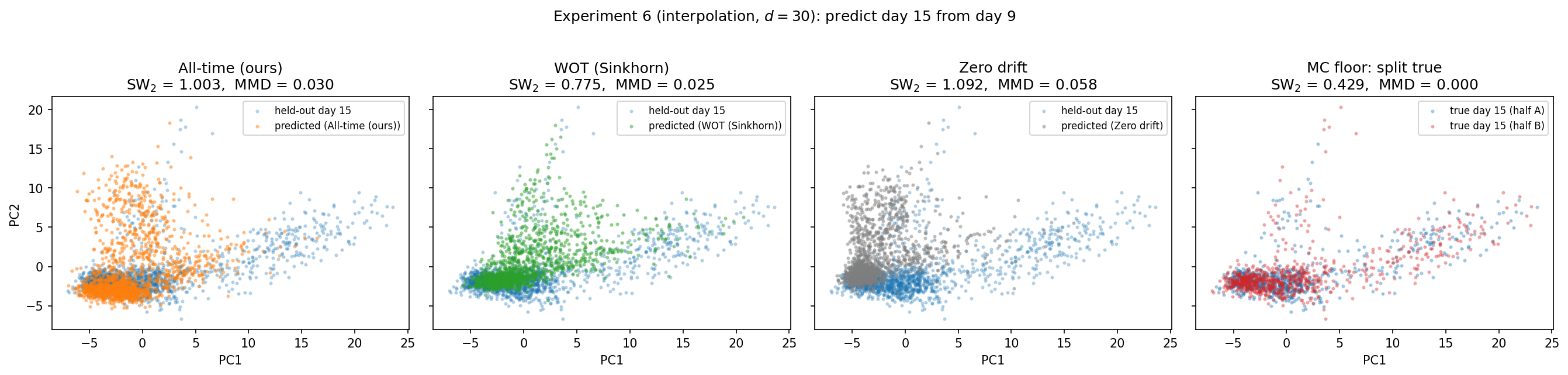}
\caption{Experiment~6: held-out day-15 prediction on the EB
scRNA-seq benchmark, projected onto the first two PCA components,
in the neighbor-interpolation setting.  Each panel overlays the
predicted day-15 cloud (colored) on the held-out true day-15 cloud
(blue).  Right-most panel shows two disjoint halves of the held-out
sample as the Monte-Carlo floor reference.}
\label{fig:exp-eb}
\end{figure}

\subsection{Experiment~7: Stochastic extension}
\label{sec:exp-stochastic}

The convergence analysis in Theorem~\ref{thm:convergence} is
established for the deterministic case $\sigma=0$.  The corresponding theory for the
Nelson problem with $\sigma>0$ is beyond the scope of this paper and is
left for future work.  The present subsection does \emph{not} serve
as a numerical verification of a convergence theorem.  Instead, its
purpose is twofold: (a)~to demonstrate that the computational
estimator of Section~\ref{sec:penalty-expansion} extends without
structural modification to the stochastic case, requiring only that
the operator $\mathcal{A}^u$ be augmented with the $(\sigma^2/2)\Delta$
term and that the forward ODE simulation be replaced by an
Euler--Maruyama SDE simulation, and (b)~to provide a direct comparison
with the deterministic experiments of
Sections~\ref{sec:exp-gaussian}--\ref{sec:exp-2d-bifurcation} on a
setting where a closed-form optimal drift is available.  We do not
repeat the WOT/MMOT benchmarks here, as the $O(1/\Delta t^2)$ noise
amplification of pairwise Sinkhorn combined with finite-difference
drift reconstruction is independent of $\sigma$.  A detailed
multi-dimensional study with real-data applications is deferred to a
follow-up paper.

We validate the RKHS-penalized estimator on the one-dimensional Nelson
problem with $T=1$, $d=1$, diffusion coefficient $\sigma=1$, and
marginal flow $\mu_t=\mathcal{N}(m_t,1)$ with $m_t=-1+2t$.  For this
setup the optimal drift is
\[
  u^*(t,x) = -\frac{x}{2} + \frac{3}{2} + t,
\]
with weight vector $w^*=(3/2,\,1,\,-1/2)$ in the affine
parametrization $u(t,x)=w_0+w_1 t+w_2 x$.

We use the affine model $u(t,x)=w_0+w_1 t+w_2 x$ ($3$ parameters,
convex objective) trained with Adam ($15{,}000$ iterations, cosine LR
$5\times10^{-4}\to 5\times10^{-5}$, $\lambda=1000$, $M=N=30$,
$N_0=60$, $K_{\mathrm{ens}}=30$, $h=1$).
Figure~\ref{fig:drift} shows the learned drift at five time slices.
The final weight vector is $\hat w=(1.339,1.362,-0.579)$ versus
$w^*=(1.5,1.0,-0.5)$, giving grid MSE $=0.031$ on
$[0,1]\times[-3,3]$.  The finite-$\lambda$ bias can be computed directly in this setting.
Write $\Phi(t,x)=(1,t,x)^{\mathsf T}$ and
\[
  G:=\int_0^T\!\!\int_{\mathbb{R}}\Phi\Phi^{\mathsf T}\,p\,dx\,dt
   =\begin{pmatrix}1 & 1/2 & 0\\ 1/2 & 1/3 & 1/6\\ 0 & 1/6 & 4/3\end{pmatrix}
\]
for the present marginal flow, so that the energy term is
$\int_0^T\!\int_{\mathbb{R}}|u_w|^2p\,dx\,dt=w^{\mathsf T}Gw$.  Integration by parts
in the definition of $H_t^u$ gives
\[
  H_t^{u}(s,x)
  = -\int_0^t\!\!\int_{\mathbb{R}}
    \Bigl(\partial_r p+\partial_y(pu)-\tfrac12\partial_y^2p\Bigr)(r,y)\,
    K(s,x,r,y)\,dy\,dr ,
\]
and for $p(r,y)=\varphi(y-m_r)$, with $\varphi$ the standard normal
density, and the affine $u_w$ the Fokker--Planck residual factors as
$p\,q_w$ with
\[
  q_w(r,y)=\Bigl(w_2+\tfrac12\Bigr)(1-z^2)
  +\bigl[\dot{m}_r-u_w(r,m_r)\bigr]\,z,
  \qquad z=y-m_r,
\]
which is linear in $w-w^*$ and vanishes for every $r$ exactly when
$w=w^*$.  The penalty is therefore the quadratic form
$\int_0^T\|H_t^{u_w}\|^2dt=(w-w^*)^{\mathsf T}B\,(w-w^*)$, and $B$ is
positive definite because a zero penalty forces $q_w\equiv0$ by (A1).
The objective
$J_\lambda(u_w)=w^{\mathsf T}Gw+\lambda(w-w^*)^{\mathsf T}B(w-w^*)$
has the unique minimizer
\[
  w_\lambda=w^*-\frac{1}{\lambda}\Bigl(\frac{1}{\lambda}G+B\Bigr)^{-1}Gw^* .
\]
The bias is of order $1/\lambda$, its direction is set by $Gw^*$ and by
the curvature $B$ of the penalty, and it vanishes as $\lambda\to\infty$.
This is a statement about the population objective for the affine
class of this experiment.  Theorem~\ref{thm:convergence} is stated for
$\sigma=0$ and is not used here.

\begin{figure}[htbp]
\centering
\includegraphics[width=\textwidth]{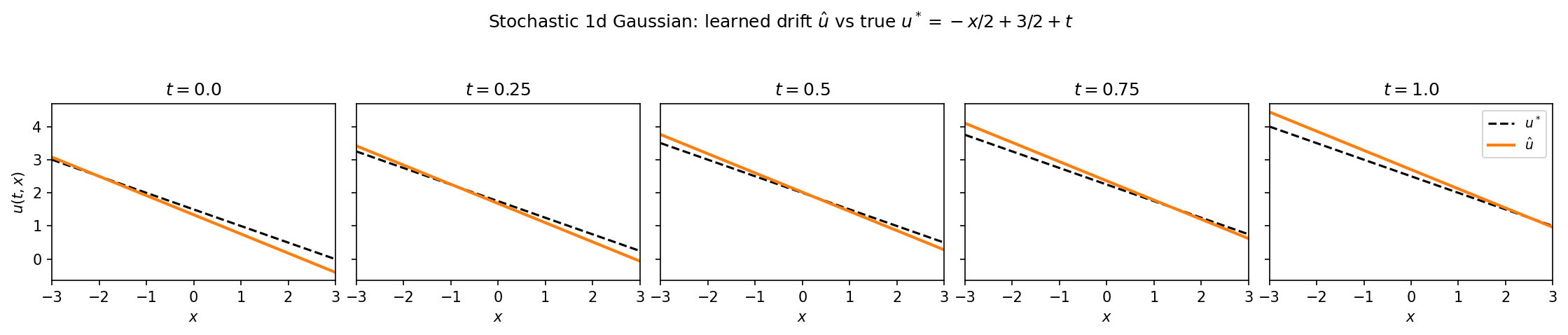}
\caption{Stochastic 1d Gaussian: learned drift $\hat u(t,x)$ (solid)
vs.\ true $u^*(t,x)=-x/2+3/2+t$ (dashed) at
$t\in\{0,0.25,0.5,0.75,1\}$.  Affine model, $\lambda=1000$, $M=N=30$.}
\label{fig:drift}
\end{figure}

To verify that the learned drift reproduces the prescribed marginal
flow as a generative model, we simulate $20{,}000$ particles from
$X_0\sim\mathcal{N}(-1,1)$ under the Euler--Maruyama scheme 
with $2{,}000$ steps.  Table~\ref{tab:marginal} reports the
Wasserstein-$2$ distance and the Gaussian MMD ($h=1$) against
$N=20{,}000$ independent samples from the true marginal
$\mu_t=\mathcal{N}(-1+2t,1)$, averaged and maximized over
$t\in\{0,0.25,0.5,0.75,1\}$.  We report $W_2$ for consistency with
Sections~\ref{sec:exp-gaussian}--\ref{sec:exp-2d-bifurcation}: both the
deterministic problem and the Nelson problem are quadratic-cost
Benamou--Brenier energies, so $W_2$ is the natural metric for the
objective.  The learned drift attains
$\mathrm{mean}\,W_2=0.036$ and $\mathrm{max}\,W_2=0.056$ over the five
evaluation times, about $0.02$ above the Monte Carlo floor achieved
by the true drift.  The zero-drift baseline (plain Brownian
motion) diverges linearly in $t$ and reaches $W_2\approx 2.04$ at
$t=1$.

\begin{table}[htbp]
\centering
\caption{Stochastic 1d Gaussian: marginal-consistency metrics over
$t\in\{0,0.25,0.5,0.75,1.0\}$.
$N=20{,}000$ particles, $2{,}000$ Euler--Maruyama steps.
$W_2$ is computed via sorted order statistics, and MMD with a Gaussian
kernel of bandwidth $h=1$.}
\label{tab:marginal}
\begin{tabular}{lcccc}
\toprule
Drift & mean $W_2$ & max $W_2$ & mean MMD & max MMD \\
\midrule
True $u^*$       & 0.0171 & 0.0200 & 0.0034 & 0.0136 \\
Learned $\hat u$ & 0.0357 & 0.0560 & 0.0082 & 0.0362 \\
Zero drift       & 1.0230 & 2.0371 & 0.3544 & 0.6582 \\
\bottomrule
\end{tabular}
\end{table}

\subsection{Comparison of the two error metrics}
\label{sec:two-metrics}

Table~\ref{tab:two-metrics} evaluates both metrics of
Section~\ref{sec:numerical} on the experiments whose model is linear
in its parameters, where~\eqref{eq:pweighted} is exact.

\begin{table}[htbp]
\centering
\caption{Grid $\mathrm{MSE}$ against the probability-weighted error
\eqref{eq:pweighted} for the linear-in-parameters experiments, where
the latter is exact.  The uniform grid inflates the error by a factor
of three to four when the misfit has an $x$-component, and the two
agree in Experiment~2, where the recovered drift is independent of
$x$.}
\label{tab:two-metrics}
\begin{tabular}{lccc}
\toprule
 & grid $\mathrm{MSE}$ & $\mathcal{D}(\hat u,u^*)$ & ratio \\
\midrule
Experiment~1 (Gaussian translation) & $0.044$ & $0.010$ & $4.2$ \\
Experiment~2 (roundtrip)            & $0.635$ & $0.634$ & $1.0$ \\
Experiment~7 (stochastic, $\sigma=1$) & $0.031$ & $0.010$ & $3.1$ \\
\bottomrule
\end{tabular}
\end{table}

The gap is a property of the metric rather than of the estimator.  In
Experiment~7, for instance, the grid $[0,1]\times[-3,3]$ devotes a
third of its area to $|x|>2$, which carries under $5\%$ of the mass of
$\mu_t$.

\section{Conclusion}
\label{sec:conclusion}

We have proposed a practical mesh-free solver for the
continuum-marginal optimal transport problem $(P)$, with an
accompanying convergence guarantee in the deterministic case:
recovering the
minimum-energy velocity field from a continuously observed flow of probability
distributions, with a continuum of marginal constraints enforced through the
continuity equation.
The key idea is to embed the weak continuity equation in a reproducing kernel
Hilbert space, yielding a sample-only objective with $O(1/M)$ bias that can be optimized
by mini-batch stochastic gradient descent with a neural-network velocity
parametrization.
The method requires only snapshot samples from the observed marginals and no
spatial grid.

The numerical experiments demonstrate that the RKHS-based
continuity-equation residual provides a reliable supervisory signal:
all trained models substantially reduce the zero-velocity baseline in
both pointwise velocity MSE and distributional metrics.
Neural network parametrization is essential for nonlinear marginal flows,
achieving roughly $4$--$6\times$ lower velocity MSE than affine models.
Marginal constraints are faithfully enforced: ODE simulations driven by the
learned velocity track the true $\mu_t$ accurately across all time slices.

The regularization weight $\lambda$ must currently be tuned by trial
and error. An adaptive selection rule would be desirable.
The computational cost scales quadratically in the mini-batch size
$N$. For large $N$, more efficient approximations such as random
Fourier features \citep{Scholkopf2002} or Nystr\"om subsampling may be
needed.
The convergence of $\hat{\mathcal Q}$ in
Section~\ref{sec:estimator} is pointwise in $u$.  Convergence
of the corresponding minimizers, and rates as $M,N\to\infty$, are
important open problems.

\emph{Stochastic extension and trajectory inference.}
As discussed in Section~\ref{sec:stochastic}, the framework
extends to the Nelson problem with positive noise level
$\sigma > 0$.  A rigorous convergence theory in this setting is
left to future work.  Section~\ref{sec:exp-eb} reports our
results on the EB scRNA-seq benchmark, where we compare the
all-time estimator with pairwise Sinkhorn-based Waddington-OT.
In the canonical neighbor-interpolation setting, pairwise WOT
attains the best marginal-tracking accuracy and the all-time
estimator is competitive, as reported in
Table~\ref{tab:exp-eb}.  A bootstrap-stability diagnostic of the inferred velocity field,
Table~\ref{tab:exp-eb-stability}, shows that the all-time estimator
and nearest-neighbor WOT are about equally stable on a scale-free
measure and that a smoothed WOT field is more stable than both, so
this diagnostic measures the interpolation rule rather than the
estimator.  The smoothed field is close to a translation of the
cloud, tracks the held-out marginal as well as any method, and is
far from the true drift on the synthetic bimodal flow, so held-out
tracking on EB is insensitive to the spatial structure of the
velocity field.
A separate global trajectory-inference framework by Lavenant et
al.\ \citep{lav:2024} is also globally defined, but it requires
$\sigma > 0$ and does not reduce to the deterministic case.  The
two approaches address complementary problem classes by the
value of $\sigma$.  On the synthetic roundtrip benchmark of
Section~\ref{sec:exp-roundtrip}, pairwise WOT cannot recover the
true velocity at intermediate times while the all-time estimator
can.  The EB experiment lacks a ground-truth velocity field, so drift
accuracy cannot be measured there.  A detailed study of these
three axes on real single-cell RNA-seq data, and the systematic
comparison with entropy-based methods such as \citep{lav:2024},
is left to future work.

\emph{Scalable kernel approximations.}
The quadratic dependence on the batch size limits the method to
moderate $N$. Random Fourier features and Nystr\"om approximations
offer a path to linear complexity and to higher spatial dimensions.
A systematic study of these approximations in the all-time OT
setting, including their interaction with the sample estimator
and with the convergence analysis of Section~\ref{sec:3-residual},
is also left to future work.

\emph{Physical applications.}
The mesh-free nature of the method makes it suitable for velocity recovery
in fluid mechanics (particle image velocimetry), crowd dynamics, and other
settings where density observations are available but particle tracking is not. 

\emph{Connection to flow matching.}
The minimizer of $(P)$ is a gradient field (Proposition~\ref{prop:existence}),
which connects to score-based generative models \citep{Song2021} and flow
matching \citep{Lipman2023,Albergo2023,Tong2024}.
Making this connection explicit, and understanding the benefit of all-time
versus two-marginal constraints in generative modeling, is a promising direction. 

\emph{Convergence analysis.}
A rigorous analysis of the estimator ($P_{\lambda}^{M,N}$) as $M,N\to\infty$,
characterizing the trade-off between the RKHS regularization bias and the
statistical variance of the sample estimator, is an important
theoretical open problem.

\subsection*{Acknowledgments}
The author thanks the action editor and the three anonymous reviewers
for careful and constructive reports.  Their comments led to the
removal of a structural assumption from the convergence theorem, to the
repair of a circular step in its proof, to corrections of the
flow-matching derivation and of the stochastic first-order condition,
and to the withdrawal of a stability claim that a closer experiment did
not support.  This study is supported by JSPS KAKENHI Grant Number
JP24K06861.  Any remaining errors are the author's own.

\subsection*{Use of AI assistants}
The author used a large language model, Anthropic Claude, as a
coding and writing assistant during the preparation of this work.  It
was used to refactor and comment the Python implementation, to generate
plotting code for the baselines, and to copy-edit the manuscript for
grammar and clarity.  All mathematical results,
experimental designs, and numerical conclusions were produced and
verified by the author, who takes full responsibility for the
content of this paper.

\bibliographystyle{plainnat}
\bibliography{mybib}

\appendix

\section{Proof of Theorem~\ref{thm:convergence}}
\label{app:proof-convergence}

\begin{proof}
Throughout the proof, for brevity, write
$\mathcal{J}_n:=\int_0^T\|H_t^{u_n}\|^2\,dt\ge 0$ and
$\mathcal{E}_n:=\int_0^T\!\int_{\mathbb{R}^d}|u_n|^2p\,dx\,dt$, so
$J_{\lambda_n}(u_n)=\mathcal{E}_n+\lambda_n\mathcal{J}_n$.
Since $u^*\in\mathcal{U}_0$, Theorem
\ref{thm:Ht-in-H} gives $H_t^{u^*}=0$ for every $t$, hence
$J_\lambda(u^*)=V_0(\mu)$ for every $\lambda$.  Together with
\eqref{eq:eps-opt} this yields
\begin{equation}\label{eq:Jbound}
 J_{\lambda_n}(u_n)\le \inf_{u\in\mathcal{U}}J_{\lambda_n}(u)+\varepsilon_n
 \le J_{\lambda_n}(u^*)+\varepsilon_n = V_0(\mu)+\varepsilon_n.
\end{equation}

Step (i). Proof of \eqref{eq:conv-H}.  We follow the
penalty-method argument of \citet[Theorem~3.1]{nak:2024sb}.  
Assume for contradiction that
\[
 \limsup_{n\to\infty}\lambda_n\mathcal{J}_n=5\delta
\]
for some $\delta>0$. Then extract a subsequence along which the lim sup is attained.
By \eqref{eq:Jbound}, the sequence $\{J_{\lambda_n}(u_n)\}$ is bounded, so a
further subsequence (relabeled $\bar\lambda_m=\lambda_{n_m}$,
$\bar u_m=u_{n_m}$, $\bar{\mathcal J}_m=\mathcal{J}_{n_m}$,
$\bar\varepsilon_m=\varepsilon_{n_m}$) satisfies
$\lim_m J_{\bar\lambda_m}(\bar u_m)=\kappa$ for some finite
$\kappa\le V_0(\mu)$ and
$\lim_m\bar\lambda_m\bar{\mathcal J}_m=5\delta$.  Choose $m_0$ with
$\kappa<J_{\bar\lambda_{m_0}}(\bar u_{m_0})+\delta$, then $m_1>m_0$
with $J_{\bar\lambda_{m_1}}(\bar u_{m_1})<\kappa+\delta$,
$\bar\lambda_{m_1}>7\bar\lambda_{m_0}$, and
$3\delta+\bar\varepsilon_{m_0}<\bar\lambda_{m_1}\bar{\mathcal J}_{m_1}<7\delta$.
Using \eqref{eq:eps-opt} and the fact that $u_{m_1}$ is feasible for
$(P_{\bar\lambda_{m_0}})$,
\begin{align*}
 \kappa
 &< J_{\bar\lambda_{m_0}}(\bar u_{m_0}) + \delta
 \le V_{\bar\lambda_{m_0}}(\mu)+\bar\varepsilon_{m_0}+\delta
 \le J_{\bar\lambda_{m_0}}(\bar u_{m_1})+\bar\varepsilon_{m_0}+\delta\\
 &= \mathcal{E}_{m_1}+\frac{\bar\lambda_{m_0}}{\bar\lambda_{m_1}}
     \bar\lambda_{m_1}\bar{\mathcal J}_{m_1}
     +\bar\varepsilon_{m_0}+\delta
 < \mathcal{E}_{m_1}+\frac{\bar\lambda_{m_1}\bar{\mathcal J}_{m_1}}{7}
     +\bar\varepsilon_{m_0}+\delta\\
 &< \mathcal{E}_{m_1} + 2\delta + \bar\varepsilon_{m_0}
 < \mathcal{E}_{m_1}+\bar\lambda_{m_1}\bar{\mathcal J}_{m_1}-\delta
 = J_{\bar\lambda_{m_1}}(\bar u_{m_1})-\delta
 < \kappa,
\end{align*}
a contradiction.  Hence $\lim_n\lambda_n\mathcal{J}_n=0$, proving
\eqref{eq:conv-H}.

Step (ii). Proof of \eqref{eq:conv-J}.  From \eqref{eq:Jbound}
and \eqref{eq:conv-H},
$\limsup_n J_{\lambda_n}(u_n)\le V_0(\mu)$.

We first argue that $\mathcal{E}_n\to V_0(\mu)$.
The upper bound $\mathcal{E}_n\le J_{\lambda_n}(u_n)\le V_0(\mu)+\varepsilon_n$
gives immediately
\begin{equation}\label{eq:En-limsup}
  \limsup_{n\to\infty}\mathcal{E}_n \;\le\; V_0(\mu).
\end{equation}
This  shows $\{u_n\}$ is bounded in the Hilbert space
$L^2(p\,dt\,dx):=L^2([0,T]\times\mathbb{R}^d;\mathbb{R}^d,p(t,x)\,dt\,dx)$.
For the matching lower bound, let $\{u_{n_k}\}$ be any subsequence
along which $\mathcal{E}_{n_k}\to\liminf_n\mathcal{E}_n$.
Since $\{u_{n_k}\}$ is still bounded, by reflexivity we may pass to a
further subsequence (relabeled $u_{n_k}$) with
$u_{n_k}\rightharpoonup  u_\infty$ weakly in $L^2(p\,dt\,dx)$.
For every $\varphi\in\mathcal{H}$,
Theorem~\ref{thm:Ht-in-H} gives
$\langle\varphi,H_t^{u_{n_k}}\rangle =L_t(\varphi;u_{n_k})$.
By Cauchy--Schwarz, $|\langle\varphi,H_t^{u_{n_k}}\rangle|\le\|\varphi\|\,\|H_t^{u_{n_k}}\|$.
Since $\lambda_{n_k}\to\infty$ and $\lambda_{n_k}\mathcal{J}_{n_k}\to 0$
by \eqref{eq:conv-H}, we have $\mathcal{J}_{n_k}\to 0$, that is
$\int_0^T\|H_t^{u_{n_k}}\|^2\,dt\to 0$.  Passing to a further
subsequence (still denoted $u_{n_k}$) we may assume
$\|H_t^{u_{n_k}}\|\to 0$ for a.e.\ $t\in[0,T]$, and therefore
$\langle\varphi,H_t^{u_{n_k}}\rangle\to 0$ for those $t$.
The functional $L_t(\varphi;\cdot)=\int_0^t\!\int p(s,x)\bigl(\partial_s\varphi
+u(s,x)^{\mathsf T}\nabla_x\varphi\bigr)dx\,ds-\int p(t,x)\varphi(t,x)dx
+\int p(0,x)\varphi(0,x)dx$ is, in its dependence on $u$, the linear
form $u\mapsto\int_0^t\!\int p\,u^{\mathsf T}\nabla_x\varphi\,dx\,ds$.
By (A2) and \citet[Theorem~10.45]{Wendland2005}, every
$\varphi\in\mathcal{H}$ has $\nabla_x\varphi\in C_b\subset L^\infty$.
Combined with $p\in L^1\cap L^\infty$ from (A3), this makes
$u\mapsto L_t(\varphi;u)$ a bounded linear functional on
$L^2(p\,dt\,dx)$.  Passing to the weak limit therefore gives
$L_t(\varphi;u_\infty)=0$ for every $\varphi\in\mathcal{H}$ and a.e.\
$t\in[0,T]$.  Since $t\mapsto L_t(\varphi;u_\infty)$ is continuous on
$[0,T]$ (by $p\in C([0,T];L^1)$ from (A3) and the boundedness of
$\varphi,\,\partial_t\varphi,\,\nabla_x\varphi$), the equality
$L_t(\varphi;u_\infty)=0$ extends from a.e.\ $t$ to every $t\in[0,T]$.
We next show that $u_\infty\in\mathcal{U}_0$.  Since
$u_\infty\in L^2(p\,dt\,dx)$ and
each $p(t,\cdot)$ is a probability density on the finite interval
$[0,T]$, Cauchy--Schwarz gives $u_\infty\in L^1(p\,dt\,dx)$, which is
the first condition defining $\mathcal{U}$.  For
$\varphi\in\mathcal{H}$, taking $t=T$ in $L_T(\varphi;u_\infty)=0$ and
integrating $\partial_t p\in L^1$ by parts in time yields
\begin{equation}\label{eq:div-identification}
 \bigl\langle \mathrm{div}_x(p\,u_\infty),\varphi\bigr\rangle
 = -\bigl\langle \partial_t p,\varphi\bigr\rangle ,
 \qquad \varphi\in\mathcal{H},
\end{equation}
the left-hand side being the distributional divergence, well defined
because $u_\infty\in L^1(p\,dt\,dx)$ and $\nabla_x\varphi$ is bounded.
Both sides of~\eqref{eq:div-identification} are continuous for the
topology of uniform convergence on compact sets, and by (A1)
$\mathcal{H}$ is universal on every compact subset of
$[0,T]\times\mathbb{R}^d$, so the identity extends from $\mathcal{H}$
to $C_0^\infty([0,T]\times\mathbb{R}^d)$.  It therefore determines the
distribution $\mathrm{div}_x(p\,u_\infty)$ as the $L^1$ function
$-\partial_t p$ supplied by (A3).  In particular
$\mathrm{div}_x(p\,u_\infty)\in L^1$, so $u_\infty\in\mathcal{U}$, and
the residual $D^{u_\infty}=\partial_t p+\mathrm{div}_x(p\,u_\infty)$
vanishes as a distribution.  The weak continuity
equation~\eqref{eq:3.1} thus holds against every
$\varphi\in C_0^\infty([0,T]\times\mathbb{R}^d)$.  This extends to
every $\varphi\in C_b^1$ by a standard mollification argument.
For $\varphi\in C_b^1$ choose $\varphi_n\in C_0^\infty$ with
$\varphi_n\to\varphi$, $\partial_t\varphi_n\to\partial_t\varphi$ and
$\nabla_x\varphi_n\to\nabla_x\varphi$ pointwise, with
$\sup_n(\|\varphi_n\|_\infty+\|\partial_t\varphi_n\|_\infty
+\|\nabla_x\varphi_n\|_\infty)<\infty$.  Then $p\in L^1$ and
$u_\infty\in L^1(p\,dt\,dx)$ allow dominated convergence in every
term of \eqref{eq:weak-CE}.  Hence $u_\infty\in\mathcal{U}_0$.
By weak lower semicontinuity of the norm along
$u_{n_k}\rightharpoonup u_\infty$, and by the choice of the
subsequence $\{u_{n_k}\}$,
\[
 \mathcal{E}_\infty:=\int_0^T\!\!\int_{\mathbb{R}^d} |u_\infty(t,x)|^2p(t,x)\,dx\,dt
 \le \liminf_k\mathcal{E}_{n_k}
 = \liminf_n\mathcal{E}_n.
\]
Since $u_\infty\in\mathcal{U}_0$, also $\mathcal{E}_\infty\ge V_0(\mu)$,
so
\[
 V_0(\mu) \le \mathcal{E}_\infty \;\le\; \liminf_{n\to\infty}\mathcal{E}_n.
\]
Together with \eqref{eq:En-limsup} this gives $\mathcal{E}_n\to V_0(\mu)$
and $u_\infty$ is a minimizer of $(P)$.
Combined with \eqref{eq:conv-H},
$J_{\lambda_n}(u_n)=\mathcal{E}_n+\lambda_n\mathcal{J}_n\to V_0(\mu)$.

Step (iii).  Proof of \eqref{eq:sup-inf}.  Since
$u^*\in\mathcal{U}_0$, Theorem~\ref{thm:Ht-in-H} gives
$\|H_t^{u^*}\|=0$, so for every $\lambda>0$,
\[
 \inf_{u\in\mathcal{U}}J_\lambda(u)\le J_\lambda(u^*) = V_0(\mu),
\]
whence $\sup_{\lambda>0}\inf_u J_\lambda(u)\le V_0(\mu)$.  Conversely,
by \eqref{eq:eps-opt},
\[
 J_{\lambda_n}(u_n)\le \inf_u J_{\lambda_n}(u)+\varepsilon_n
 \le \sup_\lambda\inf_u J_\lambda(u)+\varepsilon_n,
\]
and letting $n\to\infty$ using \eqref{eq:conv-J} yields
$V_0(\mu)\le\sup_\lambda\inf_u J_\lambda(u)$.  This proves
\eqref{eq:sup-inf}.

Step (iv). Existence, uniqueness, and $L^2$-convergence.  Step~(ii)
shows that the weak limit $u_\infty$ lies in $\mathcal{U}_0$ and
attains $V_0(\mu)$, so \eqref{eq:P0} admits a minimizer.  It is unique
up to $p\,dt\,dx$-null sets, by convexity.  The
constraint~\eqref{eq:weak-CE} is affine in $u$ and the two conditions
defining $\mathcal{U}$ are convex, so the finite-energy region of
$\mathcal{U}_0$ that \eqref{eq:P0} searches is convex, while the
energy $u\mapsto\int_0^T\!\int_{\mathbb{R}^d}|u|^2p\,dx\,dt$ is
strictly convex on it.  If $u_1,u_2\in\mathcal{U}_0$ both attained
$V_0(\mu)$ while differing on a set of positive $p\,dt\,dx$ measure,
their midpoint would lie in $\mathcal{U}_0$ and satisfy
$\int_0^T\!\int_{\mathbb{R}^d}|\tfrac{u_1+u_2}{2}|^2p\,dx\,dt<V_0(\mu)$, a
contradiction.  Write $u^*$ for the minimizer.  By Step~(ii) every
weak cluster point of $\{u_n\}$ is a minimizer, hence equals
$u^*$.  Since every weakly convergent
subsequence of the bounded sequence $\{u_n\}\subset L^2(p\,dt\,dx)$
has this same limit, the full sequence satisfies
\begin{equation}\label{eq:un-weak}
 u_n\rightharpoonup u^*\qquad\text{weakly in }L^2(p\,dt\,dx;\mathbb{R}^d).
\end{equation}
Moreover $\mathcal{E}_n\to V_0(\mu)=\int|u^*|^2p\,dx\,dt$.

Because $V_0(\mu)<\infty$ and $u^*$ attains it, $u^*$ lies in
$L^2(p\,dt\,dx;\mathbb{R}^d)$ and may therefore be used as its own test
element in~\eqref{eq:un-weak}, which gives
\[
 \lim_{n\to\infty}\int_0^T\int_{\mathbb{R}^d} p(t,x)u_n^{\mathsf{T}}(t,x)u^*(t,x)\,dx\,dt
 =\int_0^T\int_{\mathbb{R}^d} p(t,x)|u^*(t,x)|^2\,dx\,dt.
\]
Combining this cross-term convergence with norm convergence
$\mathcal{E}_n\to\int|u^*|^2p\,dx\,dt$, we obtain 
\begin{align*}
 &\int_0^T\int_{\mathbb{R}^d} |u_n(t,x)-u^*(t,x)|^2p(t,x)\,dx\,dt \\ 
 &= \mathcal{E}_n
   - 2\int_0^T\int_{\mathbb{R}^d} p(t,x)u_n^{\mathsf{T}}(t,x)u^*(t,x)\,dx\,dt
   + \int_0^T\int_{\mathbb{R}^d} p(t,x)\,|u^*(t,x)|^2\,dx\,dt
 \;\longrightarrow\; 0. 
\end{align*}
Thus \eqref{eq:conv-u-L2} follows.
\end{proof}

\section{Expansion of the penalty term}
\label{app:penalty-expansion}

\begin{proof}
Decompose $H_t^u=A_1+A_2+A_3\in\mathcal{H}$ where
\begin{equation}\label{eq:H-split}
\begin{aligned}
 A_1 &= \int_0^t\!\!\int_{\mathbb{R}^d}p(r,x)\mathcal{A}^u_{r,y}K(\cdot,r,x)\,dx\,dr, \\
 A_2 &= - \int_{\mathbb{R}^d}p(t,x)K(\cdot,t,x)\,dx, \\ 
 A_3 &= \int_{\mathbb{R}^d}p(0,x)K(\cdot,0,x)\,dx. 
\end{aligned}
\end{equation}
Expanding the squared norm gives 
\begin{equation}\label{eq:square-expand}
 \|H_t^u\|^2 = \|A_1\|^2 + \|A_2\|^2 + \|A_3\|^2
 + 2\langle A_1,A_2\rangle + 2\langle A_1,A_3\rangle + 2\langle A_2,A_3\rangle.
\end{equation}
We evaluate each contribution using the reproducing property. 
By \eqref{eq:RKHS-norm-integral}
applied on $Q=[0,t]\times\mathbb{R}^d$,
\[
 \|A_1\|^2
 = \int_0^t\!\!\int_0^t\!\!\int_{\mathbb{R}^d}\!\!\int_{\mathbb{R}^d}
 p(s,x)p(r,y)\mathcal{A}^u_{s,x}\mathcal{A}^u_{r,y}K(s,x,r,y)\,dx\,dy\,ds\,dr
 \;=\; I_1(t),
\]
on identifying $p(r,y)dy$ with the law of $\tilde{X}_r\sim\mu_r$.
Next, a direct use of
$\langle K(\cdot,y),K(\cdot,y')\rangle  = K(y,y')$ yields
\[
 \|A_2\|^2 = \int_{\mathbb{R}^d}\int_{\mathbb{R}^d} p(t,x)p(t,x^{\prime})K(t,x,t,x^{\prime})\,dx\,dx^{\prime} = I_2(t)
\]
and $\|A_3\|^2 = I_3(t)$.  Applying the reproducing
property with $\varphi=A_1\in\mathcal{H}$, we find 
\begin{align*}
 \langle A_1,A_2\rangle
 &= -\int_{\mathbb{R}^d} p(t,x^{\prime})\langle A_1,K(\cdot,t,x^{\prime})\rangle\,dx^{\prime}
  = -\int_{\mathbb{R}^d} p(t,x^{\prime})\,A_1(t,x^{\prime})\,dx^{\prime}\\
 &= -\int_0^t\int_{\mathbb{R}^d}\int_{\mathbb{R}^d} p(r,x)p(t,x^{\prime})\,
      \mathcal{A}^u_{r,y}K(t,x^{\prime},r,x)\,dx\,dx^{\prime}\,dr
   = -I_4(t).
\end{align*}
By the same reasoning with $A_2$ replaced by $-A_3$ and the time argument $t\mapsto 0$,
\[
 \langle A_1,A_3\rangle
 = \int_{\mathbb{R}^d} p(0,x^{\prime})A_1(0,x^{\prime})\,dx^{\prime}
 = I_6(t).
\]
Further, directly,
\[
 \langle A_2,A_3\rangle
 = - \int_{\mathbb{R}^d}\int_{\mathbb{R}^d} p(t,x)p(0,x^{\prime})K(t,x,0,x^{\prime})\,dx\,dx^{\prime}
 = -I_5(t).
\]
Collecting all six contributions in \eqref{eq:square-expand} yields
\eqref{eq:Ht-six-terms}.
\end{proof}

\section{Sensitivity analysis in $(M, N, \lambda)$}
\label{sec:exp-sensitivity}

\paragraph{Choosing the bandwidth.}
The sweep below holds the bandwidth $h$ fixed, so we first explain how
$h$ is chosen.  All synthetic experiments use $h=1$ and the single-cell
experiment uses $h=8$.  Both values come from one rule.  The kernel
acts on space-time, and the relevant quantity is the size of a typical
off-diagonal entry $K(\xi_p,\xi_q)$ on the sample.  An entry near one
makes the penalty insensitive to differences between sample points.
An entry near zero makes the kernel matrix numerically diagonal, and
the penalty then contains no cross-sample information.  We therefore
target entries of order $10^{-1}$.  For marginals with unit covariance,
two independent draws satisfy $\mathbb{E}|\Delta x|^2=2d$, so an
isotropic kernel with fixed $h$ has typical entry $\exp(-d/h^2)$.  At
$h=1$ this is $0.37$ in dimension one and $0.14$ in dimension two, both
in the target range.  The $30$-dimensional principal-component
representation of Section~\ref{sec:exp-eb} has a per-coordinate
standard deviation near $2$, and the same rule gives $h=8$, with
typical entries between $0.14$ and $0.22$.  The rule also predicts
where a fixed $h$ fails.  At $d=10$ with $h=1$ the typical entry is
$4.5\times10^{-5}$, and the estimator degrades accordingly, as
Appendix~\ref{sec:exp-dim-scaling} reports.

The main-text experiments all fix the sampling budget and the penalty
weight $\lambda$ at hand-chosen values.  To quantify the robustness of
the estimator to these choices, we perform a one-at-a-time sensitivity
analysis on the simplest setting, Experiment~1 ($d=1$, $\sigma=0$,
$\mu_t=\mathcal{N}(-1+2t,1)$, $u^*(t,x)\equiv 2$).  We sweep each of
\[
  M \in \{10,15,20,30,50\}, \quad
  N \in \{5,10,15,25,40\}, \quad
  \lambda \in \{10^{1},10^{2},10^{3},10^{4},10^{5}\},
\]
keeping the other two parameters at their default values
$(M_0, N_0, \lambda_0) = (30, 20, 10^{3})$.  For every
$(M, N, \lambda)$ point we run $K_{\mathrm{seed}}=4$ independent
realizations, each of which ensemble-averages $K_{\mathrm{ens}}=5$
RNG draws inside the L-BFGS-B objective.  Grid
MSE$(\hat u, u^*)$ is computed on
$[0,1]\times[-3,3]$.

Table~\ref{tab:exp7} reports mean $\pm$ standard deviation across
seeds, and Figure~\ref{fig:exp7} plots the three sweeps on log-log
axes.

\begin{table}[ht]
\centering
\small
\begin{tabular}{lccccc}
\toprule
 $M$      & 10 & 15 & 20 & 30 & 50 \\
 MSE      & $0.821\pm0.426$ & $0.180\pm0.091$ & $0.083\pm0.009$
          & $\mathbf{0.054\pm0.013}$ & $0.111\pm0.098$ \\
\midrule
 $N$      & 5 & 10 & 15 & 25 & 40 \\
 MSE      & $0.267\pm0.119$ & $0.073\pm0.052$ & $0.130\pm0.084$
          & $\mathbf{0.062\pm0.071}$ & $0.107\pm0.097$ \\
\midrule
 $\lambda$& $10$ & $10^{2}$ & $10^{3}$ & $10^{4}$ & $10^{5}$ \\
 MSE      & $1.934\pm0.045$ & $0.346\pm0.065$ & $\mathbf{0.054\pm0.013}$
          & $0.096\pm0.053$ & $0.104\pm0.059$ \\
\bottomrule
\end{tabular}
\caption{Appendix~\ref{sec:exp-sensitivity}: grid MSE mean $\pm$ std over
$K_{\mathrm{seed}}=4$ seeds for one-at-a-time sweeps in $M$, $N$ and
$\lambda$ on the base Experiment~1 problem.  Bold entries mark the
best setting in each row.  Defaults are
$(M, N, \lambda)=(30, 20, 10^{3})$.}
\label{tab:exp7}
\end{table}

Three observations emerge.  The \textbf{$M$-dependence} is nearly
monotone decreasing up to $M\approx 30$ and then degrades slightly:
the monotone phase reflects the standard quadrature-error reduction of
the outer time integral.  At $M=50$, adjacent time slices are
$\approx 0.02$ apart, well below $h=1$, so the block-diagonal mask
removes essentially no pairs and the bias correction becomes
ineffective.  This is consistent with the heuristic $M\asymp h^{-1}T$.
The \textbf{$N$-dependence} improves sharply from $N=5$ to $N=10$ and
then saturates: with only $8$ free parameters, the error becomes
quadrature-limited rather than particle-limited, and increasing $N$
beyond $25$ brings little benefit.  The \textbf{$\lambda$-dependence}
exhibits a clear sweet spot near $\lambda=10^3$.  At $\lambda=10$ the
penalty is too weak and the optimizer collapses to $u\approx 0$ (MSE
approaches $(u^*)^2=4$).  For $\lambda\gtrsim 10^4$ the penalty
approaches a hard constraint and finite-sample noise gradually
degrades the estimator.  The plateau spans roughly two decades
($10^3\le\lambda\le 10^5$), so fine-grained tuning is not required.

\begin{figure}[ht]
\centering
\includegraphics[width=\linewidth]{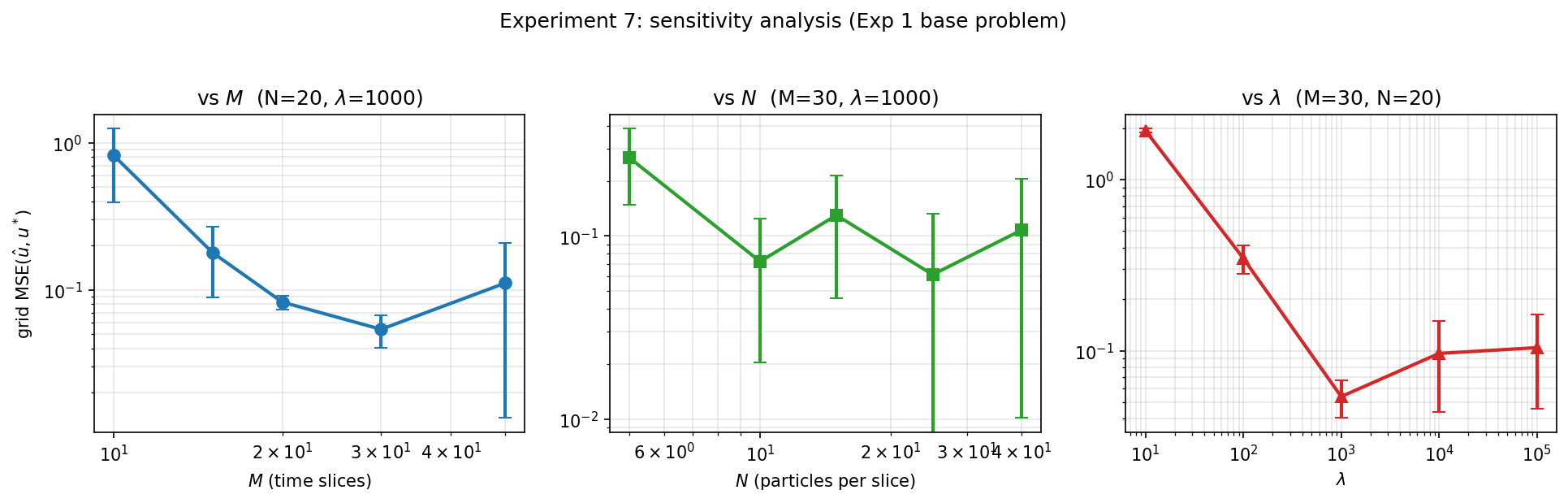}
\caption{Appendix~\ref{sec:exp-sensitivity}: sensitivity of the grid MSE to the three main
hyperparameters $M$, $N$, $\lambda$ on the Experiment~1 base
problem.  Each point is the mean over
$K_{\mathrm{seed}}=4$ seeds (with $K_{\mathrm{ens}}=5$ RNG draws per
seed) and error bars are one standard deviation.  Defaults are
$(M_0,N_0,\lambda_0)=(30, 20, 10^{3})$.}
\label{fig:exp7}
\end{figure}

In summary, the estimator prefers $M\asymp h^{-1}T$, saturates in $N$
beyond $N\simeq 15$, and has a two-decade plateau in $\lambda$.
Excluding the two under-resolved settings ($M=10$ and $\lambda=10$),
the worst-case MSE remains below $0.2$.  We recommend the defaults
$(M,N,\lambda)\approx(30,20,10^3)$ and re-tuning only $\lambda$ when
the magnitude of the RKHS residual changes.

\section{Dimension scaling}
\label{sec:exp-dim-scaling}

Experiment~4 established that the dimension-independent kernel operator
formula $[-a^{2}\tau\tau' + a(1+u^{\mathsf{T}} u')]K$ extends Experiment~1 to
$d=2$ without modification.  In this appendix we push the same Gaussian
translation problem to higher spatial dimensions and quantify how the
accuracy and the wall-clock time scale with $d$.

For each $d\in\{1,2,3,5,8,10\}$ we take $T=1$, $\sigma=0$ and
\[
  \mu_t = \mathcal{N}\left(\frac{t}{\sqrt{d}}\mathbf{1}_d, \mathbf{I}_{d}\right),
  \quad t\in[0,1].
\]
The all-ones direction is rescaled by $1/\sqrt{d}$ so that the optimal
drift $u^*(t,x)\equiv\mathbf{1}_d/\sqrt{d}$ has unit Euclidean norm
$|u^*|=1$ independently of $d$.  This normalization makes the
total grid MSE $\mathbb{E}|\hat u-u^*|^{2}$ directly comparable
across dimensions.

Same affine linear-in-parameters dictionary as in Experiment~4,
extended to any~$d$:
\[
  u_i(t, x) = w_{i,0} + w_{i,1}\,t + \sum_{j=1}^{d} w_{i,j+1}\,x_j,
  \quad i = 1, \ldots, d,
\]
so the total parameter count is $n_{\mathrm{params}} = d(d+2)$.

Same as Experiment~4: $h=1$, $\lambda=10^{3}$, $M=25$ time slices,
$N=20$ particles per slice, $N_{0}=50$ initial particles,
$K_{\mathrm{ens}}=5$ ensemble draws inside the objective,
$K_{\mathrm{seed}}=10$ independent realizations per dimension.
The grid MSE is estimated by Monte-Carlo sampling $2{,}000$ uniform
points in $[-3,3]^{d}$ at each of $15$ time slices.

Table~\ref{tab:exp8} reports the mean and standard deviation of the
total grid MSE, the per-component MSE (i.e.\ total MSE divided
by~$d$), the average wall-clock optimization time and the parameter
count.

\begin{table}[ht]
\centering
\small
\begin{tabular}{rrrrrr}
\toprule
 $d$ & $n_{\mathrm{params}}$ & total MSE (mean $\pm$ std)
     & median MSE (IQR) & per-dim med.\ MSE & time (s) \\
\midrule
  1 &    3 & $0.139 \pm 0.145$ & $0.098$ $[0.050,\,0.180]$ & $0.098$ & $0.4$ \\
  2 &    8 & $0.701 \pm 0.394$ & $0.700$ $[0.380,\,0.881]$ & $0.350$ & $1.3$ \\
  3 &   15 & $0.697 \pm 0.407$ & $0.582$ $[0.482,\,0.793]$ & $0.194$ & $2.2$ \\
  5 &   35 & $0.816 \pm 0.193$ & $0.761$ $[0.684,\,1.018]$ & $0.152$ & $3.2$ \\
  8 &   80 & $0.674 \pm 0.101$ & $0.669$ $[0.596,\,0.759]$ & $0.084$ & $2.4$ \\
 10 &  120 & $0.726 \pm 0.072$ & $0.733$ $[0.709,\,0.761]$ & $0.073$ & $2.4$ \\
\bottomrule
\end{tabular}
\caption{Appendix~\ref{sec:exp-dim-scaling}: dimension scaling of the affine RKHS all-time
OT estimator on $\mu_t=\mathcal{N}((t/\sqrt{d})\mathbf{1}_d, \mathbf{I}_d)$
with $|u^*|=1$.  Mean~$\pm$~standard deviation and median with
inter-quartile range over $K_{\mathrm{seed}}=10$ seeds.  The per-dim
median MSE is the median total MSE divided by~$d$.  The mean and
median agree to within one standard error across all dimensions,
confirming that no single seed is driving the trend.}
\label{tab:exp8}
\end{table}

Three features are worth noting.

(i)~\textbf{Per-component error is essentially dimension-free.}  The
total grid MSE $\mathbb{E}|\hat u-u^*|^{2}$ saturates at roughly
$0.7$ for all $d\ge 2$, both in mean and in median.  Since this
quantity is a sum over the $d$ output components, the per-component
median MSE $\approx 0.7/d$ actually \emph{decreases} with $d$, and
for $d\ge 5$ it is of the same order as, or smaller than, the $d=1$
benchmark of $\approx 0.1$.  The estimator thus does not suffer
from a curse of dimensionality on this simple constant-drift
problem.  With $K_{\mathrm{seed}}=10$ seeds the median is robust to
an occasional noisy realization, and the mean agrees with it.

(ii)~\textbf{The plateau reflects an identifiability invariant, not
an algorithmic failure.}  This is the same divergence-free
perturbation set discussed in the cross-experiment summary at the
end of Section~\ref{sec:exp-2d-bifurcation}: for $d\ge 2$ there are
non-trivial $p$-weighted divergence-free fields $v$ whose addition
to $u^*$ leaves the marginal flow invariant, and finite samples
plus finite~$\lambda$ allow the estimator to drift along these
directions.  In higher dimensions this valley has more degrees of
freedom, which explains why the per-coefficient deviation
$|\widehat W - W^*|$ does not vanish.  What remains controlled
is the push-forward marginal error, which is the quantity that the
variational objective actually penalizes.

(iii)~\textbf{Runtime scales mildly.}  The optimization time grows from
$0.6$~s at $d=1$ to $\approx 5$~s at $d=10$, even though the parameter
count increases by a factor of $40$ (from $3$ to $120$).  The L-BFGS-B
iteration count remains between $25$ and $50$ across the sweep, so the
dominant cost is the per-iteration gradient assembly, which in our
implementation consists of an $\mathcal{O}(d)$ loop over spatial
components inside the fixed-cost $MN\times MN$ kernel block.  The
empirical runtime, plotted in Figure~\ref{fig:exp8}, is therefore
consistent with the predicted
$\mathcal{O}(d\cdot M^2 N^2)$ complexity and is dominated by the
$(MN)^{2}=2.5\times 10^{5}$ kernel block rather than by the number
of parameters.

\begin{figure}[ht]
\centering
\includegraphics[width=\linewidth]{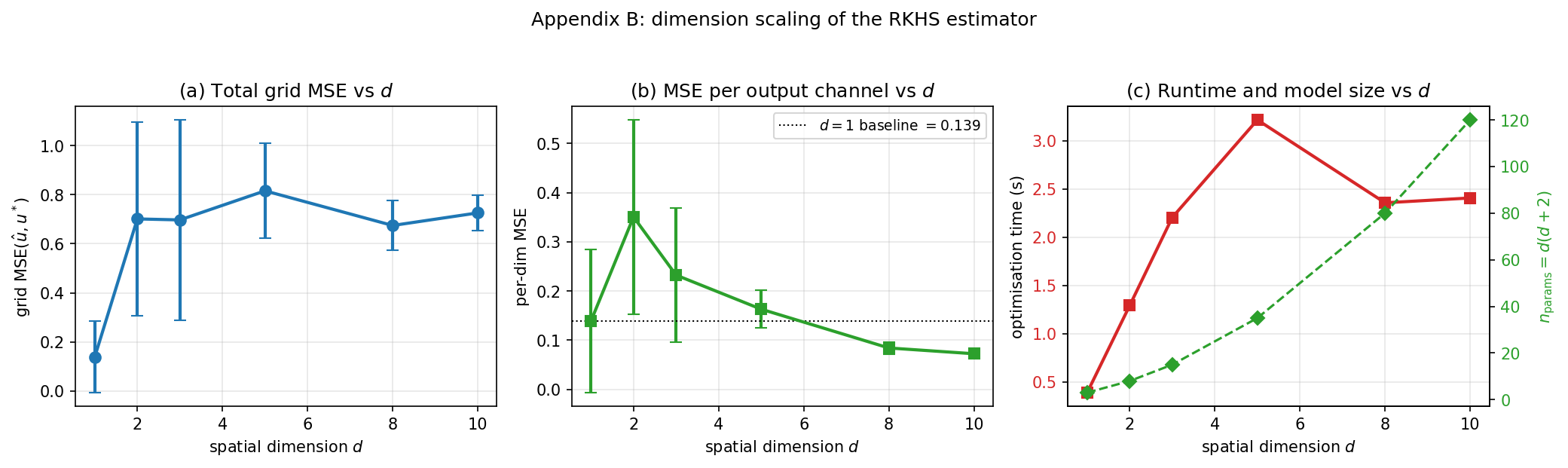}
\caption{Appendix~\ref{sec:exp-dim-scaling}: dimension scaling.  Panel (a) shows the total grid
MSE$(\hat u, u^*)$ against $d$.  Panel (b) shows the per-component MSE,
which is nearly dimension free.  Panel (c) shows the wall-clock
optimization time in red and the parameter count $n_{\mathrm{params}}=d(d+2)$
in green on the right axis.  Error bars are one standard deviation over
$K_{\mathrm{seed}}=10$ seeds.}
\label{fig:exp8}
\end{figure}

In summary, on a problem where $|u^*|$ is held constant, the
estimator exhibits essentially dimension-free per-component accuracy,
and its runtime grows quasi-linearly in $d$ at a fixed sample budget.
The remaining per-coefficient slack is a consequence of the OT
identifiability valley, not of any algorithmic limitation.
\end{document}